\documentclass[10pt]{amsart}

\usepackage{amsfonts}
\usepackage{amsmath}% math enhancements
\usepackage{amssymb,latexsym}% AMSFonts and LaTeX symbol names
\usepackage[usenames,dvipsnames]{xcolor}
\usepackage[inline]{enumitem}
\usepackage{wasysym}

\newcommand{\subref}[1]{\dosubref\cref#1\relax}

\def\dosubref#1#2:#3\relax{#1{#2}\ref{#2:#3}}

\usepackage[all]{xy}
\usepackage{graphicx}
\usepackage{longtable}

\usepackage[left=2.41cm,top=2.6cm,bottom=2.6cm,right=2.41cm,a4paper]{geometry}
\usepackage{microtype}  % improved typesetting: character protrusion and extension

\definecolor{linkblue}{RGB}{1,1,190}
\definecolor{citegreen}{RGB}{1,190,1}
\usepackage[linkcolor=linkblue
          ,citecolor=citegreen
          ,ocgcolorlinks        % This ensures that colored links are
          ,bookmarksopen=True   % Automatically expand bookmarks in PDF
           ]{hyperref}

\usepackage{amsthm}% proclamations with style

\usepackage[capitalize          % Always capitalize references
           ]{cleveref}
\crefname{equation}{Equation}{Equations}  % arXiv does not support noabbrev
\Crefname{equation}{Equation}{Equations}  % arXiv does not support noabbrev
\theoremstyle{plain}
\newtheorem{theorem}{Theorem}[section]
\newtheorem{corollary}[theorem]{Corollary}
\newtheorem{lemma}[theorem]{Lemma}
\newtheorem{proposition}[theorem]{Proposition}

\theoremstyle{definition}
\newtheorem{definition}[theorem]{Definition}
\newtheorem{construction}[theorem]{Construction}

\newtheorem{remark}[theorem]{Remark}

\newtheorem{example}[theorem]{Example}

\numberwithin{equation}{section} %equations numbered within each section

\setlist[enumerate,1]{label=\textup{(\arabic*)},ref=\textup{(\arabic*)}}
\setlist[enumerate,2]{label=\textup{(\roman*)},ref=\textup{(\roman*)}}
\newlist{enumerateequiv}{enumerate}{1}  % for enumerating equivalent statements
\setlist[enumerateequiv,1]{label=\textup{(\alph*)},ref=\textup{(\alph*)}}

\newcommand{\envnewline}{%
  \mbox{}%
  \setenumerate{beginpenalty=52}%
}

\newcommand{\lbar}{\overline}

\newcommand{\bC}{\mathbb C}
\newcommand{\bG}{\mathbb G}
\newcommand{\bN}{\mathbb N}
\newcommand{\bQ}{\mathbb Q}
\newcommand{\bR}{\mathbb R}
\newcommand{\bZ}{\mathbb Z}
\newcommand{\cA}{\mathcal A}
\newcommand{\cB}{\mathcal B}
\newcommand{\cC}{\mathcal C}
\newcommand{\cF}{\mathcal F}
\newcommand{\cG}{\mathcal G}
\newcommand{\cH}{\mathcal H}
\newcommand{\cI}{\mathcal I}
\newcommand{\cL}{\mathcal L}
\newcommand{\cO}{\mathcal O}
\newcommand{\cP}{\mathcal P}
\newcommand{\cQ}{\mathcal Q}
\newcommand{\cR}{\mathcal R}
\newcommand{\cX}{\mathcal X}
\newcommand{\cY}{\mathcal Y}

\newcommand{\fp}{\mathfrak p}
\newcommand{\fq}{\mathfrak q}
\newcommand{\fS}{\mathfrak S}
\renewcommand{\sc}{\mathsf c}   % This ovverides \sc = smallcaps
\newcommand{\sd}{\mathsf d}
\newcommand{\st}{\mathsf t}
\newcommand{\sC}{\mathsf C}
\newcommand{\sD}{\mathsf D}
\newcommand{\sL}{\mathsf L}
\newcommand{\sV}{\mathsf V}
\newcommand{\sZ}{\mathsf Z}
\newcommand{\fmon}{\mathcal F^*}
\newcommand{\fpath}{\mathcal F^*}
\newcommand{\cgrp}{C}
\newcommand{\quo}{ \mathbf{q} }
\newcommand{\isomto}{\overset{\sim}{\rightarrow}}
\DeclareMathOperator{\ann}{ann}
\DeclareMathOperator{\spec}{spec}
\DeclareMathOperator{\nr}{nr}
\DeclareMathOperator{\Cl}{\mathcal C}
\DeclareMathOperator{\id}{id}
\DeclareMathOperator{\Norm}{N}

\DeclareMathOperator{\Mono}{Mono}
\newcommand{\nwr}{\not\wr\;}

\newcommand{\ab}[1]{{#1}_{\textup{ab}}}
\newcommand{\rab}[1]{{#1}_{\textup{rab}}}
\newcommand{\red}[1]{{#1}_{\textup{red}}}
\newcommand{\abrel}{\;\ab\equiv\;}
\newcommand{\rabrel}{\;\rab\equiv\;}

\newcommand{\ceq}{\textup{eq}}
\newcommand{\cmon}{\textup{mon}}
\newcommand{\cadj}{\textup{adj}}
\newcommand{\dsim}{\textup{sim}}  % symbol for similarity distance (Cohn)
\newcommand{\dsubsim}{\textup{subsim}}  % symbol for subsimilarity distance (Brungs)
\renewcommand{\vec}[1]{\mathbf{#1}}

\providecommand{\length}[1]{\lvert#1\rvert}

\providecommand{\babs}[1]{\big\lvert#1\big\rvert}
\providecommand{\Bigabs}[1]{\Big\lvert#1\Big\rvert}
\providecommand{\card}[1]{\lvert#1\rvert}
\providecommand{\meet}{\wedge}
\providecommand{\join}{\vee}

\def\rfop{*}
\makeatletter
\newcommand\rigidfactorization[2][]{%
  \def\rf@delim{\rfop}
  \newif\ifrf@notfirst
  #1
  \@for\next:=#2\do{%
    \ifrf@notfirst
      \rf@delim
    \fi
    \rf@notfirsttrue
    \next
  }%
}
\makeatother
\newcommand\rf\rigidfactorization
\newcommand\permutablefactorization[2][]{[\rf[#1]{#2}]_p}
\newcommand\pf\permutablefactorization

\begin{document}

\hyphenation{iso-atom-ic}

\title{Factorization theory: From commutative to noncommutative settings}

\author[N.~Baeth]{Nicholas R.~Baeth}
\author[D.~Smertnig]{Daniel Smertnig}

\address{Department of Mathematics and Computer Science \\
         University of Central Missouri \\
         Warrensburg, MO 64093, USA}
\email{baeth@ucmo.edu}
\address{Institut f\"ur Mathematik und Wissenschaftliches Rechnen \\
         Karl-Franzens-Universit\"at Graz \\
         Heinrichstra\ss e 36\\
         8010 Graz, Austria}
\email{daniel.smertnig@uni-graz.at}
\thanks{The first author was a Fulbright-NAWI Graz Visiting Professor in the Natural Sciences and supported by the Austrian-American Education Commission.
The second author was supported by the Austrian Science Fund (FWF) projects P26036-N26 and W1230, Doctoral Program ``Discrete Mathematics''.}



\begin{abstract}
  We study the non-uniqueness of factorizations of non zero-divisors into atoms (irreducibles) in noncommutative rings.
  To do so, we extend concepts from the commutative theory of non-unique factorizations to a noncommutative setting.
  Several notions of \emph{factorizations} as well as \emph{distances} between them are introduced. In addition, arithmetical invariants characterizing the non-uniqueness of factorizations such as the \emph{catenary degree}, the \emph{$\omega$-invariant}, and the \emph{tame degree}, are extended from commutative to noncommutative settings.
  We introduce the concept of a cancellative semigroup being \emph{permutably factorial}, and characterize this property by means of corresponding catenary and tame degrees.
  Also, we give necessary and sufficient conditions for there to be a weak transfer homomorphism from a cancellative semigroup to its reduced abelianization.
  Applying the abstract machinery we develop, we determine various catenary degrees for classical maximal orders in central simple algebras over global fields by using a natural transfer homomorphism to a monoid of zero-sum sequences over a ray class group. We also determine catenary degrees and the permutable tame degree for the semigroup of non zero-divisors of the ring of $n \times n$ upper triangular matrices over a commutative domain using a weak transfer homomorphism to a commutative semigroup.
\end{abstract}

\maketitle

\section{Introduction}

The study of factorizations in commutative rings and semigroups has a long and rich history.
Beginning with attempts to understand the factorizations of elements in rings of algebraic integers into irreducibles, this field has grown to include the investigation of non-unique factorizations in Mori domains, Krull domains and Krull monoids, including the study of direct-sum decompositions of modules (see \cite{Baeth-Geroldinger}).
These investigations have used tools from multiplicative ideal theory, algebraic and analytic number theory, combinatorics, and additive group theory.
A thorough overview of the various aspects of commutative factorization theory can be found in \cite{Anderson,Baeth-Wiegand,Chapman,Fontana-Houston-Lucas,Geroldinger09,GHK06}.

On the other hand, the study of unique and non-unique factorization in noncommutative rings and semigroups has received limited attention.  In fact, for many years the study of factorizations in noncommutative settings had been restricted to characterizing and studying noncommutative rings with properties analogous to that of commutative unique factorization domains or to studying factorizations of certain (symmetric) polynomials over noncommutative (e.g. matrix) rings (see \cite{GRW01, HR95, LL04, LO04, GRSW05, GGRW05, LLO08, DL07, R10, L12}).
From the beginning it was clear that each (noncommutative) PID intrinsically has certain unique factorization properties (see, for instance, \cite[Chapter 3.4]{Jacobson}, \cite[Chapter VI.9]{Deuring} and \cite[page 230]{Reiner}).
More recently, such phenomena have been studied; for semifirs and in particular $2$-firs by P.~M.~Cohn \cite{Cohn85,Cohn06}, for the ring of Hurwitz and Lipschitz quaternions by Conway and Smith \cite{Conway-Smith} and by H.~Cohn and Kumar \cite{Cohn-Kumar}, for quaternion orders by Estes and Nipp \cite{Estes-Nipp,Estes}, and in a more general setting by Brungs \cite{Brungs}. Somewhat different notions of unique factorization domains and unique factorization rings were introduced by Chatters and Jordan \cite{Chatters,Chatters-Jordan,Jordan}, and have found applications in \cite{Jespers-Wang, Launois-Lenagan-Rigal, Goodearl-Yakimov}.

Recently, techniques from the factorization theory of commutative rings and monoids have been used to investigate non-unique factorizations in a noncommutative setting.
For example, in \cite{Baeth-Ponomarenko} factorizations within some natural subsemigroups of matrices with integer coefficients are considered, and in \cite{Bachman-Baeth-Gossell} factorizations within the subsemigroup of non zero-divisors of the ring of $n \times n$ upper triangular matrices $T_n(D)^\bullet$ over an arbitrary atomic commutative domain $D$ are studied.
In \cite{Geroldinger}, noncommutative Krull monoids are investigated. Through the study of the divisorial two-sided ideals of $S$ that closely parallels the techniques that have been used fruitfully for commutative Krull monoids, it is shown that in the normalizing case ($aS=Sa$ holds for all $a$ in the Krull monoid $S$) many results from the commutative setting generalize.
In \cite{Smertnig} this approach is, by means of divisorial one-sided ideal theory, extended to a class of semigroups that includes commutative and normalizing Krull monoids as special cases.
In particular, this is applied to investigate factorizations in the semigroup of non zero-divisors of classical maximal orders in central simple algebras over global fields.
In this way, results on some basic invariants of non-unique factorization theory, namely sets of lengths, are obtained.

In \cite{Baeth-Ponomarenko}, \cite{Bachman-Baeth-Gossell}, \cite{Geroldinger}, and \cite{Smertnig}, the focus on noncommutative factorizations was solely on the sets of lengths of factorizations of a given element; that is, sets of the form
\[
  \sL(a)=\{\,n \in \bN_0 : a=u_1\cdots u_n \text{ with each $u_i$ an atom in $S$}\,\}
\]
and associated invariants.
Much of this work was done through the use of various generalizations of transfer homomorphisms to the noncommutative settings, each of which preserves sets of lengths.
While sets of lengths are amongst the most classical of arithmetical invariants describing the non-uniqueness of factorizations, their usefulness is limited by the fact that they can only measure how far a ring or semigroup is away from \emph{half-factoriality}, that is, the property that all factorizations of a non-unit element into atoms have the same length.

The purpose of this paper is to study more refined invariants that describe the non-uniqueness of factorizations in a noncommutative setting, with considerations not just of sets of lengths, but also of distinct factorizations of elements and divisibility properties.
Our main objects of interest are certain classes of rings, but, as is done in the commutative case, we develop everything in the setting of cancellative semigroups (and often more generally in the setting of cancellative small categories), for essentially three reasons:
First, we wish to emphasize that the theory of factorizations is a purely multiplicatively one; secondly, many of the auxiliary objects that appear in studying the factorization theory of rings (e.g. the monoid of zero sum sequences) are not themselves rings, yet we need to be able to apply the language of factorization theory to these objects; and thirdly, sometimes the object one is interested in studying itself is not a ring, but a semigroup.
Moreover, throughout we will restrict to the cancellative case because even in the commutative setting the introduction of non-cancellative elements significantly increases the complexity of studying factorizations.
For rings, this means that we will consider factorizations within the semigroup of non zero-divisors.

While it is completely clear how sets of lengths should be defined in the noncommutative setting, any attempt to introduce more refined invariants such as the catenary degree or the tame degree in a noncommutative setting immediately leads one to the following question.
When are two representations of a non-unit as products of atoms to be considered the same, and when are they distinct?
In the commutative setting, one typically considers factorizations up to permutation and associativity, but this seems less fitting for many natural noncommutative objects.
Further, in order to be able to describe how distinct different factorizations of an element are, one needs to define a reasonable \emph{distance} between two factorizations.

The choice of a notion of a distance and that of a factorization are closely linked, but there does not seem to be a canonical choice that is entirely satisfactory.
For example, based on investigations by P.~M.~Cohn and Brungs in \cite{Brungs, Cohn85, Cohn06}, one can introduce two different notions of factorizations and corresponding distances, both coinciding with the usual one when considered in the commutative setting.
A third notion, that of permutable factorizations, turns out to be particularly well suited to other examples, for instance the semigroup of non zero-divisors of the ring of $n \times n$ upper triangular matrices over a commutative atomic domain.
For this reason, in \cref{sec:factorizations}, we first recall a rigorous notion of \emph{rigid factorizations}.
Based on this, we introduce an axiomatic notion of a distance $\sd$, and derive from it the notion of \emph{$\sd$-factorizations}.

Each such distance gives rise to a corresponding catenary degree and monotone catenary degree which we define and study in \cref{sec:catenary}.
As in the commutative setting, the (monotone) catenary degree associated to a distance $\sd$ provides a measure of how far away a cancellative small category is from being $\sd$-factorial.
In \cref{cf,strong-weak-distance} we show that catenary degrees can be studied using (weak) transfer homomorphisms.

In \cref{sec:divisibility} we approach the study of factorization from the viewpoint of divisibility, introducing \emph{almost prime-like elements} and \emph{prime-like elements} that generalize prime elements from the commutative setting. We then introduce corresponding \emph{tame degrees} and \emph{$\omega$-invariants} based on permutable factorizations that measure how far a given element is from being almost prime-like.
With these notions we are able to give characterizations of \emph{permutable factoriality} in \cref{p-factorial-acd,pf-omega,perm-tame}.

In \cref{sec:abelianization} we consider the notion of weak transfer homomorphisms as introduced in \cite{Bachman-Baeth-Gossell} and give criteria for when there is such a weak transfer homomorphism from a cancellative semigroup to its reduced abelianization (if the abelianization is itself cancellative).
Any weak transfer homomorphism preserves sets of lengths, and by constructing weak transfer homomorphisms from noncommutative cancellative semigroups to commutative cancellative semigroups, we illustrate that sometimes it is possible to reduce the study of sets of lengths in a noncommutative ring or semigroup to a corresponding commutative semigroup, where sets of lengths may have been investigated before.
Of course, such noncommutative semigroups then necessarily have systems of sets of lengths which also occur as systems of sets of lengths in the commutative setting.
At the end of the section we revisit some known examples of weak transfer homomorphisms: In particular, we study various distances for $T_n(D)^\bullet$, and determine the corresponding catenary degrees, tame degrees and $\omega_p$-invariants in \cref{cor:tri-fact}.
Also, in \cref{prop:krull-finite-omega-invariant}, we show that for a normalizing Krull monoid $S$ (as studied in \cite{Geroldinger}), $\omega_p(S,a)$ is always finite as is the case in the commutative setting.

Finally, in \cref{sec:maxord}, we investigate catenary degrees in saturated subcategories of arithmetical groupoids and arithmetical maximal orders in quotient semigroups (as studied in \cite{Smertnig}).
This treatment also includes normalizing Krull monoids considered in \cite{Geroldinger}.
Under suitable conditions, these subcategories, respectively maximal orders, possess a transfer homomorphism to a monoid of zero-sum sequences over a subset of an abelian group.
The factorization theory of monoids of zero-sum sequences over finite abelian groups has been intensively studied (see \cite{Geroldinger09,Grynkiewicz}), due to its applications to commutative Krull monoids arising from rings of algebraic integers and holomorphy rings in function fields over finite fields.
It is therefore desirable to show that catenary degrees in our setting can be studied by means of this transfer homomorphism, as the known results from the commutative setting then immediately carry over.
Indeed, under the expected conditions, we are able to obtain satisfactory results about the catenary degree (cf. \cref{thm:cat-transfer,cor:cat-transfer-semigroup})
that mirror results about commutative Krull monoids. These results, in fact, do not depend very strongly on the particular distance chosen.
We then apply these results to classical maximal orders in central simple algebras over global fields (as long as we have the additional property that every stable free left ideal is free), and obtain \cref{thm:cat-transfer-csa,equfact}, showing that the catenary degree in this case is controlled by the catenary degree of a monoid of zero-sum sequences over a certain ray class group.
Thus, for instance, the results on catenary degrees in commutative Krull monoids obtained in \cite{GGS11} hold in our noncommutative setting.
We summarize some of the consequences in \cref{cor:cat-transfer}.

Throughout, we illustrate the limits of extending the commutative theory to the noncommutative setting by way of simple examples of semigroups given by a presentation via the generators and relations.
While such semigroups will only serve as isolated examples for us, we note that the study of the interplay of arithmetical invariants and presentations of a commutative semigroup was initiated by P.\,A.~Garc\'ia S\'anchez and further investigated by various authors (see \cite{Blanco-GarciaSanchez-Geroldinger,Chapman-GarciaSanchez-Llena-Ponomarenko-Rosales,Philipp}).
We have concentrated the discussion of the main objects of our interest at the end of \cref{sec:abelianization}, where we discuss $T_n(D)^\bullet$, matrix rings over PIDs, and almost commutative semigroups (which include normalizing Krull monoids), and \cref{sec:maxord}, where we discuss arithmetical maximal orders (which, again, include matrix rings over PIDs and normalizing Krull monoids), and in particular classical maximal orders in central simple algebras over global fields.
The relationship between arithmetical maximal orders, Krull monoids, Krull rings, and UF-monoids in the sense of P.~M.~Cohn is discussed in the preliminaries.

\section{Preliminaries}\label{sec:preliminaries}

\subsection*{Notation} We denote by $\bN = \{ 1,2,3, \ldots \}$ the set of natural numbers, and by $\bN_0 = \bN \cup \{ 0 \}$ the set of non-negative integers.
If $a$,~$b \in \bR$, we write $[a,b] = \{\, x \in \bZ : a \le x \le b \,\}$ for the discrete interval from $a$ to $b$.
For $n \in \bN_0$, we write $\fS_n$ for the group of permutations on $[1,n]$ (with $\fS_0 = \fS_1$ the trivial group).
We write $\sC_n$ for a cyclic group of order $n \in \bN$.

We will often introduce an arithmetical invariant as a supremum of a subset $X \subset \bN_0 \cup \{ \infty \}$ and,
by convention, we set $\sup \emptyset = 0$.

\subsection*{Semigroups and small categories}
By a semigroup we always mean a semigroup with a neutral element. A homomorphism of semigroups is always assumed to preserve the neutral element,
and an empty product in a semigroup is defined to be equal to the neutral element.
All rings are assumed to have an identity element, and all ring homomorphisms preserve the identity.
A \emph{domain} is a ring in which zero is the only zero-divisor, and a \emph{principal ideal domain} (PID) is a domain in which every left ideal is generated by a single element and every right ideal is generated by a single element.

In studying the divisorial one-sided ideal theory of noncommutative semigroups (as will be necessary in \cref{sec:maxord}), we are naturally led to consider not only semigroups, but the more general notion of a small category.
Hence we will introduce the necessary notions in this setting.

A small category is a category for which both the class of objects and the class of morphisms are sets.
To a semigroup $S$ we may associate a category with a single object with set of morphisms $S$ and composition of morphisms given by the operation of the semigroup.
Conversely, to a small category with a single object we can associate the semigroup of endomorphisms on that object.
In this way we obtain an equivalence of the category of semigroups and the category of small categories with a single object.
We view small categories as generalizations of semigroups with a partial operation, and set up our notation for small categories in a way that facilitates this point of view: In particular, we emphasize the role of the morphisms, while deemphasizing the role of the objects.

Let $H$ be a small category.
We identify the set of objects of $H$ with the corresponding identity morphisms, and denote the set of identity morphisms of the objects of $H$ by $H_0$.
To be consistent with the language used for semigroups, we shall refer to morphisms of the category $H$ simply as elements of $H$, writing $a \in H$ for a morphism of $H$, and call functors between small categories homomorphisms.
For each $a \in H$ we denote by $s(a) \in H_0$ its source (domain), and by $t(a) \in H_0$ its target (codomain).
Writing composition left to right, contrary to the usual convention for categories, but in line with the conventions for groupoids, we write $ab$ for the composition of $a$ and $b$ in $H$ if $t(a) = s(b)$.
The set of all isomorphisms (which we shall call \emph{units}) in $H$ will be denoted by $H^\times$.
We say that $H$ is \emph{reduced} if $H^\times=H_0$, that is, the only isomorphisms are the identity morphisms, and we say that $H$ is a \emph{groupoid} if $H = H^\times$.
For $e$,~$f \in H_0$ we set $H(e,f) = \{\, a \in H : s(a) = e \text{ and } t(a) = f \,\}$, and further $H(e)=H(e,e)$, $H(e,\cdot) = \bigcup_{f' \in H_0} H(e,f')$, and $H(\cdot,f) = \bigcup_{e' \in H_0} H(e',f)$.
For $A$,~$B \subset H$ we define $AB = \{\, ab \in H : a \in A, b \in B, t(a)=s(b) \,\}$, and if $a \in H$, we define $aB = \{a\}B$ and $Ba = B\{a\}$.
We say that $b \in H$ \emph{left divides} $a \in H$, and write $b \mid_l a$, if $a \in bH$.
Two elements $a$, $b \in H$ are \emph{left coprime} if, for all $c \in H$, $c \mid_l a$ and $c \mid_l b$ implies $c \in H^\times$.
We define $b \mid_r a$ and the notion \emph{right coprime} analogously.

A congruence relation $\sim$ on a category $H$ is, for each pair $e$, $f \in H_0$, a reflexive, symmetric and transitive relation on $H(e,f)$ (we tacitly denote all of these relations by $\sim$ again) satisfying the condition that for all $x$,~$x'$, $y$, $y' \in H$ with $s(x)=s(x')$, $t(x)=t(x')=s(y)=s(y')$, $t(y)=t(y')$ and $x \sim x'$, $y \sim y'$, we have $xy \sim x'y'$.
Given a congruence relation $\sim$ on $H$, we may define a quotient category $H/\!\sim$ with $(H/\!\sim)_0 = H_0$ and $(H/\!\sim)(e,f) = H(e,f)/\!\sim$ for all $e$,~$f \in H_0$.

We shall often not explicitly specify the source and target of elements, but tacitly assume that the necessary conditions for certain products to be defined are fulfilled: For example, if we write ``Let $a$,~$b \in H$ such that $ab=\ldots$'', then we shall implicitly assume that $a$ and $b$ are such that $t(a)=s(b)$.

We say that $H$ is \emph{normalizing} if $Ha= aH$ for all $a \in H$.
In this case, for all $a \in H$, we have $s(a)=t(a)$. Indeed, $a=s(a) a= b s(a)$ for some $b \in H$ with $t(b)=s(a)$, and thus $b=a$ and $s(a)=t(a)$.
Thus $H(e,f) = \emptyset$ whenever $e$,~$f \in H_0$ are distinct, and $H$ is a union of disjoint semigroups.
Clearly, every commutative semigroup is normalizing.

If $H$ is a small category and $a \in H$, then $a$ is \emph{cancellative} if for all $b$,~$c \in H$, $ab=ac$ implies $b=c$ and $ba=ca$ implies $b=c$ (that is, $a$ is both a monomorphism and an epimorphism).
The category $H$ itself is called \emph{cancellative} if each $a \in H$ is cancellative.
We write $H^\bullet$ for the cancellative subcategory of cancellative elements of $H$.

Let $H$ be a cancellative small category.
If $a$,~$b \in H$ with $ab = t(b) = s(a)$, then $bab = b$, and hence $ba = s(b) = t(a)$.
Similarly, $aba=a$ implies $ab = s(a) = t(b)$.
Thus every left (right) invertible element is invertible.
If $m \in \bN$ and $a_1$, $\ldots\,$,~$a_m \in H$ are such that  $a_1 \cdots a_m \in H^\times$, then $a_i \in H^\times$ for each $i \in [1,m]$.

We call two elements $a$,~$b \in H$ \emph{associated}, and write $a \simeq b$, if there exist $\varepsilon$,~$\eta \in H^\times$ such that $b = \varepsilon a \eta$.
Clearly $\simeq$ is an equivalence relation and we denote the equivalence class of $a \in H$ by $[a]_\simeq$.
In general, associativity may not be a congruence relation. In the case of small categories this is partially due to the fact that our notion of a congruence relation is very restrictive. However, we will only care about this relation in the case of semigroups, and will not introduce a more general notion of congruences for small categories.
If $H^\times a = a H^\times$ for all $a \in H$, then, $H^\times(e,f) = \emptyset$ for all $e$, $f \in H_0$ which are distinct.
Moreover, in this case, $\simeq$ is a congruence relation and $\red H$ is again cancellative.
These conditions are satisfied if $H$ is normalizing.
Indeed, suppose $H$ is normalizing, $a \in H$, and $\varepsilon \in H^\times$ with $t(\varepsilon)=s(a)$.
We have already observed that in this case $t(\varepsilon)=s(\varepsilon)=t(\varepsilon^{-1})$.
Therefore, since $H$ is normalizing, there exist $b$,~$c \in H$ such that $\varepsilon a = ab$ and $\varepsilon^{-1}a = ac$.
Then $a = (\varepsilon \varepsilon^{-1})a = abc$, and by cancellativity $bc=s(b)$.
Thus $b$,~$c \in H^\times$, and hence $H^\times a \subset a H^\times$.
The other inclusion follows similarly.

If $H$ is a small category such that $\simeq$ is a congruence relation, we define the \emph{reduced small category} associated to $H$ as $\red H = H/\!\simeq$.
Note that $\red H$ is indeed reduced with $\red H^\times = \{\, [e]_\simeq : e \in H_0 \,\}$.
If $\pi\colon H \to \red H$ denotes the canonical homomorphism, then $\pi^{-1}(\red H^\times) = H^\times$.
If $S$ is a semigroup, we will call $\red S$ the \emph{reduced semigroup} associated to $S$.

Let $Q$ be a quiver, that is, a directed graph which may contain multiple arrows between each pair of vertices as well as loops.
If $a$ is an arrow of $Q$, we write $s(a)$ for its starting vertex and $t(a)$ for its target vertex.
A \emph{path} from a vertex $e$ of $Q$ to a vertex $f$ of $Q$ is a tuple $(e,a_1,\ldots,a_k,f)$ with $k \in \bN_0$ and $a_1$, $\ldots\,$,~$a_k$ arrows of $Q$ such that either $k > 0$ and $e=s(a_1)$, $t(a_i)=s(a_{i+1})$ for all $i \in [1,k-1]$, and $t(a_k)=f$, or $k=0$ and $e=f$.
To a quiver $Q$ we associate the \emph{path category} $\fpath(Q)$ with objects the vertices of $Q$ and morphisms from a vertex $e$ to a vertex $f$ consisting of all paths from $e$ to $f$ in $Q$.
The composition is given by the natural concatenation of paths.
This construction yields a morphism of quivers $j\colon Q \to \fpath(Q)$ and the pair $(\fpath(Q), j)$ is characterized by the universal property that any morphism of quivers $f\colon Q \to H$ to a small category $H$ factors through $j$ in a unique way.

If $X$ is a set, we may associate to $X$ the quiver consisting of a single vertex and the set of loops $X$ on that vertex.
In this special case we recover the notion of a \emph{free monoid}, the elements of which we may view as words on the alphabet $X$, and which we shall denote by $\fmon(X)$.
As is usual, we shall write elements of $\fmon(X)$ as formal products on the alphabet $X$, instead of adopting the tuple notation that we use for path categories.
We write $\langle X \mid R \rangle$ for the semigroup with generators $X$ and relations $R$, that is, $\langle X \mid R \rangle$ is the quotient of $\cF^*(X)$ by the congruence relation generated by $\{\, (u,v) \in \cF^*(X) \times \cF^*(X) : u=v \in R \,\}$.

By $\cF(X)$ we denote the (multiplicatively written) \emph{free abelian monoid} with basis $X$.

\subsection*{Basic notions of factorization theory}
Let $H$ be a cancellative small category.
An element $u \in H \setminus H^\times$ is an \emph{atom} (or \emph{irreducible}) if $u = ab$ with $a$,~$b \in H$ implies either $a \in H^\times$ or $b \in H^\times$.
We denote by $\cA(H)$ the quiver of all atoms of $H$, that is, the quiver with vertex set $H_0$ and arrows consisting of atoms of $H$.
When the additional structure of the quiver is not necessary (in particular in the case that $H$ is a semigroup), we will view $\cA(H)$ simply as the set of atoms.
We say that $H$ is \emph{atomic} if every non-unit element of $H$ can be expressed as a finite product of atoms of $H$.
A sufficient condition for a cancellative small category $H$ to be atomic is that it satisfies the ascending chain condition both on principal left ideals and on principal right ideals.
The standard proof from commutative monoids or domains generalizes to this setting; see for example \cite[Proposition 3.1]{Smertnig}.

\smallskip
Transfer homomorphisms are a key tool in the investigation of non-unique factorizations (see \cite[Section 3.2]{GHK06}).
The notion of a weak transfer homomorphism was introduced in \cite{Bachman-Baeth-Gossell} to be able to study sets of lengths in a wider class of noncommutative semigroups than is possible with transfer homomorphisms.
In either case, given a cancellative small category $H$ one seeks to find an easier-to-study or more well-understood cancellative small category $T$, and a homomorphism from $H$ to $T$, that preserves many properties related to factorizations.
In our applications, the target category $T$ will always be a commutative cancellative semigroup.

\begin{definition}\label{th}
  Let $H$ and $T$ be cancellative small categories.
  \begin{enumerate}
    \item\label{th:th}
    A homomorphism $\phi\colon H \to T$ is called a \emph{transfer homomorphism} if it has the following properties:
    \begin{enumerate}[label=\textup{(\textbf{T\arabic*})},ref=\textup{(T\arabic*)}]
      \item\label{th:units} $T=T^\times \phi(H)T^{\times}$ and $\phi^{-1}(T^{\times})=H^{\times}$.
      \item\label{th:lift} If $a \in H$, $b_1$, $b_2 \in T$ and $\phi(a)=b_1b_2$, then there exist $a_1$,~$a_2 \in H$ and $\varepsilon \in T^\times$ such that $a = a_1a_2$, $\phi(a_1) = b_1 \varepsilon^{-1}$, and $\phi(a_2) = \varepsilon b_2$.
    \end{enumerate}

  \item\label{th:wth}
    Suppose $T$ is atomic.
    A homomorphism $\phi\colon H \rightarrow T$ is called a \emph{weak transfer homomorphism} if it has the following properties:
    \begin{enumerate}[label=\textup{(\textbf{WT\arabic*})},ref=\textup{(WT\arabic*)}]
      \item[\textbf{(T1)}] $T=T^\times \phi(H)T^{\times}$ and $\phi^{-1}(T^{\times})=H^{\times}$.
        \setcounter{enumii}{1}
      \item\label{wth:lift} If $a \in H$, $n \in \bN$, $v_1$, $\ldots\,$,~$v_n \in \cA(T)$ and $\phi(a)=v_1\cdots v_n$, then there exist $u_1$, $\ldots\,$,~$u_n \in \cA(H)$ and a permutation $\sigma \in \fS_n$ such that $a=u_1\cdots u_n$ and $\phi(u_i) \simeq v_{\sigma(i)}$ for each $i \in [1,n]$.
    \end{enumerate}
  \end{enumerate}
\end{definition}

It is easy to see that if $H$ and $T$ are cancellative small categories and $\phi\colon H \to T$ is a transfer homomorphism, or $T$ is atomic and $\phi\colon H \to T$ is a weak transfer homomorphism, then an element $u \in H$ is an atom of $H$ if and only if $\phi(u)$ is an atom of $T$.
If $H$ and $T$ are cancellative small categories and $\phi\colon H \to T$ is a transfer homomorphism, then $H$ is atomic if and only if $T$ is atomic.
If $T$ is atomic and $\phi\colon H \to T$ is a weak transfer homomorphism, then $H$ is also atomic.

If $\simeq$ is a congruence relation on a cancellative small category $H$ and $\red H$ is cancellative, it is easy to check that the canonical homomorphism $H \to \red H$ is a transfer homomorphism.
A composition of two transfer homomorphisms is again a transfer homomorphism, and the same holds for weak transfer homomorphisms.
In particular, if $\phi\colon H \to T$ is a (weak) transfer homomorphism, $\simeq$ is a congruence relation on $T$, and $\red T$ is cancellative, then the induced homomorphism $\phi\colon H \to \red T$ is also a (weak) transfer homomorphism.

The following example shows that in order to obtain a notion that preserves factorization theoretical invariants it is indeed necessary to require that $T$ is atomic in the definition of a weak transfer homomorphism.

\begin{example}
  Let $P$ be a countable set, say $P = \{\, p_n : n \in \bN_0 \,\}$, and let $S=\cF(P)$ be the free abelian monoid with basis $P$.
  Let $\sim$ be the congruence relation on $S$ generated by $\{\, p_n = p_{n+1}^2 : n \in \bN_0 \,\}$, and let $T = S/\!\!\sim$ be the quotient semigroup with canonical homomorphism $\pi\colon S \to T$.
  We claim that $T$ is cancellative. By the universal property of the free abelian monoid, there exists a semigroup homomorphism $\varphi\colon S \to (\bQ,+)$ such that $\varphi(p_n) = 2^{-n}$ for all $n \in \bN_0$, and $\varphi$ factors through $\pi$ to give a homomorphism $T \to (\bQ,+)$ that maps $[p_n]_\sim$ to $2^{-n}$.
  It follows that, for all $n$,~$k$,~$l \in \bN_0$, $p_n^k \sim p_n^l$ if and only if $k=l$.
  Let $a$,~$b$,~$c \in S$ be such that $ac \sim bc$.
  By the defining relations of $T$, there exists an $n \in \bN_0$ and $m$,~$k$,~$l \in \bN_0$ such that $c \sim p_n^m$, $a \sim p_n^k$, and $b \sim p_n^l$.
  Therefore $p_n^{m+k} \sim p_n^{m+l}$ and hence $k=l$ and $a=b$.
  Thus $T$ is cancellative.
  Since $S$ and $T$ are both reduced, the homomorphism $\pi$ satisfies \labelcref{th:units}, and since $T$ obviously contains no atoms \labelcref{wth:lift} is trivially satisfied. However, atoms of $S$ are not mapped to atoms of $T$ and, in fact, $S$ is factorial while $T$ is not even atomic.
\end{example}

It follows that if $T$ is atomic, then any transfer homomorphism $\phi\colon H \to T$ is also a weak transfer homomorphism.
However, the converse is not true in general.
The following lemma better illustrates the difference between transfer homomorphisms and weak transfer homomorphisms.
We omit the proof as the first two claims follow by straightforward induction and the defining properties, and the last claim is an immediate consequence of the second.

\begin{lemma}
  Let $H$ and $T$ cancellative small categories and let $T$ be atomic.
  \begin{enumerate}
    \item Let $\phi\colon H \to T$ be a homomorphism satisfying \labelcref{th:units}.
          Then $\phi$ is a transfer homomorphism if and only if the following property holds:
          If $a \in H$, $n \in \bN$, $v_1$, $\ldots\,$,~$v_n \in \cA(T)$ and $\phi(a)=v_1\cdots v_n$, then there exist $u_1$, $\ldots\,$,~$u_n \in \cA(H)$ and $\varepsilon_1=s(v_1)$,~$\varepsilon_2$, $\ldots\,$,~$\varepsilon_{n}$,~$\varepsilon_{n+1}=t(v_n) \in T^\times$ such that $a=u_1\cdots u_n$, and $\phi(u_i) = \varepsilon_i v_{i} \varepsilon_{i+1}^{-1}$ for each $i \in [1,n]$.

    \item Suppose that $T$ is a commutative semigroup, and let $\phi\colon H \to T$ be a homomorphism satisfying \labelcref{th:units}.
        The following statements are equivalent.
        \begin{enumerateequiv}
          \item $\phi$ is a weak transfer homomorphism.
          \item If $a \in H$, $n \in \bN_{\ge 2}$, $v_1$, $\ldots\,$,~$v_n \in \cA(T)$ and $\phi(a)=v_1\cdots v_n$, then there exist $i \in [1,n]$, $a_0 \in H$ and $u \in \cA(H)$ such that $a=a_0 u$, $\phi(a_0) \simeq v_1\cdots v_{i-1}v_{i+1}\cdots v_n$, and $\phi(u) \simeq v_i$.
        \end{enumerateequiv}
        Furthermore, the following statements are equivalent.
        \begin{enumerateequiv}
          \item $\phi$ is a transfer homomorphism.
          \item If $a \in H$, $b_1$,~$b_2 \in T$ and $\phi(a)=b_1 b_2$, then there exist $a_1$,~$a_2 \in H$ such that $a=a_1 a_2$, $\phi(a_1)\simeq b_1$, and $\phi(a_2) \simeq b_2$.
          \item If $a \in H$, $n \in \bN$, $v_1$, $\ldots\,$,~$v_n \in \cA(T)$ and $\phi(a)=v_1\cdots v_n$, then there exist $u_1$,~$\ldots\,$, $u_n \in \cA(H)$ such that $a=u_1\cdots u_n$ and $\phi(u_i)\simeq v_i$ for each $i \in [1,n]$.
        \end{enumerateequiv}

    \item Suppose $H$ and $T$ are commutative semigroups, and let $\phi\colon H \to T$ be a homomorphism.
          Then $\phi$ is a transfer homomorphism if and only if it is a weak transfer homomorphism.
  \end{enumerate}
\end{lemma}

\begin{remark}
  There are examples of atomic semigroups for which there exists a weak transfer homomorphism to some commutative atomic semigroup, but for which there does not exist a transfer homomorphism to any commutative semigroup.
  Indeed, if $D$ is any commutative atomic domain with atoms $u$ and $v$ such that $u^2=v^m$ for some $m>2$, and $S=T_2(D)^\bullet$ denotes the cancellative semigroup of all $2\times 2$ upper triangular matrices with entries in $D$ having nonzero determinant, then there is no transfer homomorphism from $S$ to any commutative semigroup.
  However, there is a weak transfer homomorphism from $S$ to $(\red{D^\bullet})^2$.
  See \cite[Example 4.5]{Bachman-Baeth-Gossell} for details.
\end{remark}

In some cases, (weak) transfer homomorphisms do not transfer certain information from $T$ to $H$.
It will therefore occasionally be useful to impose the following strong extra condition.

\begin{definition}\label{aainjwt}
  Let $H$ and $T$ be cancellative small categories and let $\phi\colon H \to T$ be a homomorphism.
  We say that $\phi$ is \emph{isoatomic} provided that $\phi(u) \simeq \phi(v)$ implies $u \simeq v$ for all $u$,~$v \in \cA(H)$.
\end{definition}

The following lemma illustrates just how strong the isoatomic condition is when the domain of a weak transfer homomorphism is assumed to be commutative.
A more general version of this lemma will be given in \cref{assocwth}.

\begin{lemma}\label{aaiwthcommutative}
  Let $S$ and $T$ be commutative atomic cancellative semigroups, and let $\phi\colon S \to T$ be an isoatomic (weak) transfer homomorphism.
  Then $\phi$ induces an isomorphism $\red S\cong \red T$.
\end{lemma}

\begin{proof}
  By definition, $T=\phi(S) T^\times$, so the induced semigroup homomorphism $\red\phi\colon \red S \to \red T$ is surjective.
  Suppose that $\phi(a) \simeq \phi(b)$ for some $a$,~$b \in S$.
  Since $T$ is atomic, there exist $n \in \bN_0$ and atoms $w_1$, $\ldots\,$,~$w_n$ in $T$ such that $\phi(a) \simeq \phi(b) \simeq w_1\cdots w_n$.
  Since $\phi$ is a weak transfer homomorphism and $S$ is commutative, $\phi$ is a transfer homomorphism. Thus there are atoms $u_1$, $\ldots\,$, $u_n$ and $v_1$, $\ldots\,$,~$v_n$ in $S$ such that $a \simeq u_1\cdots u_n$, $b \simeq v_1\cdots v_n$, and $\phi(u_i) \simeq \phi(v_i) \simeq w_i$ for each $i \in [1,n]$.
  Since $\phi$ is isoatomic, $u_i \simeq v_i$ for each $i \in [1, n]$ and thus $a \simeq b$.
  We have therefore shown that $\phi(a)\simeq \phi(b)$ implies $a \simeq b$ and hence the induced homomorphism $\red\phi\colon \red S \to \red T$ is an isomorphism.
\end{proof}

Sets of lengths and invariants derived from them belong to the most basic invariants used in studying non-unique factorizations.
We refer the reader to \cite[Chapter 4]{GHK06} for a thorough introduction to the study of sets of lengths in the commutative setting.
If $H$ is a cancellative small category and $a \in H \setminus H^\times$, then
\[
  \sL(a) = \sL_H(a) = \{\, n \in \bN : \text{there exist $u_1$,~$\ldots\,$,~$u_n \in \cA(H)$ with $a=u_1\cdots u_n$} \,\} \subset \bN
\]
is called the \emph{set of lengths of $a$}.
We set $\sL(\varepsilon) = \{ 0 \}$ for all $\varepsilon \in H^\times$.
We call $\cL(H) = \{\, \sL_H(a) : a \in H \,\}$ the \emph{system of sets of lengths of $H$}.

Let $\emptyset \ne L \subset \bZ$.
A positive integer $d \in \bN$ is a \emph{distance of $L$} if there exists an $l \in L$ such that $L \cap [l, l+d] = \{l, l+d\}$.
We denote by $\Delta(L)$ the set of all distances of $L$.
The \emph{set of distances of $H$} is defined as
\[
  \Delta(H) = \bigcup_{L \in \cL(H)} \Delta(L).
\]
We say that $H$ is \emph{half-factorial} if $\card{\sL(a)} = 1$ for all $a \in H$, equivalently, $H$ is atomic and $\Delta(H) = \emptyset$.
The \emph{elasticity} of a set $L \subset \bN$ is $\rho(L) = \sup\left\{\, \frac{m}{n} : m, n \in L \,\right\}$, and $\rho(\{0\})=0$.
The \emph{elasticity} of an element $a \in H$ is $\rho_H(a)= \rho(\sL_H(a))$ and the \emph{elasticity} of $H$ is $\rho(H)=\sup\{\,\rho_H(a) : a \in H\,\}$.
Note that $H$ is half-factorial if and only if $H$ is atomic and $\rho(H)=1$.

(Weak) transfer homomorphisms are a key tool in studying sets of lengths, due to the following straightforward but important result (a proof in the semigroup case for transfer homomorphisms can be found in \cite[Proposition 6.4]{Geroldinger} and for weak transfer homomorphisms in \cite[Theorem 3.2]{Bachman-Baeth-Gossell}).

\begin{lemma}
  Let $H$ and $T$ be cancellative small categories.
  Let $\phi\colon H \to T$ be a transfer homomorphism, or let $T$ be atomic and $\phi\colon H \to T$ a weak transfer homomorphism.
  Then $\sL_H(a) = \sL_T(\phi(a))$ for all $a \in H$, and in particular $\cL(H) = \cL(T)$.
\end{lemma}

In the noncommutative setting, (weak) transfer homomorphisms to appropriate commutative semigroups have already been used to study sets of lengths in \cite{Baeth-Ponomarenko,Bachman-Baeth-Gossell,Geroldinger,Smertnig}.

\subsection*{Arithmetical maximal orders and Krull monoids} \label{arith-max-ord}
In \cref{sec:maxord} we investigate catenary degrees in arithmetical maximal orders in quotient semigroups, a class of semigroups that was first studied by Asano and Murata in \cite{Asano-Murata}.
Developing our machinery in this abstract setting allows us to simultaneously treat normalizing and commutative Krull monoids as well as bounded Krull rings in the sense of Chamarie (see \cite{Chamarie,Marubayashi-VanOystaeyen}), and in particular the classical maximal orders in central simple algebras over global fields (see \cite{Reiner}) to which we ultimately apply our abstract results.
Therefore we recall the following, referring to \cite{Smertnig} for more details.

Let $Q$ be a quotient semigroup (that is, a semigroup in which every cancellative element is invertible).
A subsemigroup $S \subset Q$ is an \emph{order in Q} if for all $q \in Q$, there exist $a$,~$b \in S$ and $c$,~$d \in S \cap Q^\times$ such that $q = ac^{-1}=d^{-1}b$.
Two orders $S$ and $S'$ in $Q$ are \emph{equivalent}, denoted by $S \sim S'$, if there exist $a$,~$b$,~$c$,~$d \in Q^\times$ such that $aSb \subset S'$ and $cS'd \subset S$.
This is an equivalence relation on the orders in $Q$.
A \emph{maximal order} is an order in $Q$ that is maximal in its equivalence class (with respect to set inclusion).
A subset $I \subset Q$ is called a \emph{fractional left $S$-ideal} if $SI \subset I$, and there exist $x$,~$y \in Q^\times$ such that $x \in I$ and $Iy \subset S$.
It is called a \emph{left $S$-ideal} if moreover $I \subset S$.
(Fractional) right $S$-ideals are defined analogously, and we call $I$ a \emph{(fractional) $S$-ideal} if it is both, a (fractional) left and right $S$-ideal.
For a fractional left (respectively right) $S$-ideal $I$, we set $I^{-1} = \{\, q \in Q \mid IqI \subset I \,\}$, and this is a fractional right (respectively left) $S$-ideal.
We define $I_v = (I^{-1})^{-1}$ and call $I$ \emph{divisorial} if $I= I_v$.
The divisorial fractional left $S$-ideals form a lattice with respect to set inclusion, where $I \vee J = (I\cup J)_v$ and $I \wedge J = I \cap J$, and so do the divisorial fractional right $S$-ideals.
An order $S$ is \emph{bounded} if every fractional left $S$-ideal contains a fractional $S$-ideal, and every fractional right $S$-ideal contains a fractional $S$-ideal.

\begin{definition}[{\cite[Definition 5.18]{Smertnig}}]\label{amo}
  Let $S$ be a maximal order in a quotient semigroup $Q$.
  We say that $S$ is an \emph{arithmetical maximal order} if it has the following properties:
  \begin{enumerate}[label=\textup{(\textbf{A\arabic*})},ref=(\textup{A\arabic*})]
    \item\label{amo:acc} $S$ satisfies both the ACC (ascending chain condition) on divisorial left $S$-ideals and the ACC on divisorial right $S$-ideals.
    \item\label{amo:bdd} $S$ is bounded.
    \item\label{amo:mod} The lattice of divisorial fractional left $S$-ideals is modular, and the lattice of divisorial fractional right $S$-ideals is modular.
  \end{enumerate}
\end{definition}

We note that if $S$ is an arithmetical maximal order and $S'$ is a maximal order in $Q$ that is equivalent to $S$, then $S'$ is also an arithmetical maximal order.
Analogous ring-theoretic definitions are made for a ring $R$ which is an order in a quotient ring $Q$.

We now summarize the connections between arithmetical maximal orders and more familiar notions.

\begin{enumerate}
\item Let $S$ be an order in a group $Q$.
If $S$ is an arithmetical maximal order, then $S$ is a Krull monoid in the sense of \cite{Geroldinger} (that is, $S$ is a cancellative semigroup that is left and right Ore, is a maximal order in its quotient group, and satisfies the ACC on divisorial $S$-ideals).
If, in addition, $S$ is normalizing, then $S$ is an arithmetical maximal order if and only if it is a Krull monoid.
In particular, a commutative cancellative semigroup is an arithmetical maximal order (in its quotient group) if and only if it is a commutative Krull monoid (that is, a commutative cancellative semigroup which is completely integrally closed and satisfies the ACC on divisorial ideals).

\item Let $S$ be a normalizing cancellative semigroup.
Then $S$ is a UF-monoid in the sense of \cite[Chapter 3.1]{Cohn85} if and only if $\red S$ is a free abelian monoid,
and $S$ is a Krull monoid if and only if $\red S$ is a commutative Krull monoid.
It follows that $S$ is a UF-monoid if and only if it is a Krull monoid with trivial divisor class group (that is, every divisorial $S$-ideal is principal).

\item
  Let $R$ be a prime PI ring.
  Then $R$ is a Krull ring if and only if $R^\bullet$ is a Krull monoid.
  Equivalently, $R$ is a Krull ring if and only if it is a maximal order in its quotient ring and satisfies the ACC on divisorial $R$-ideals (equivalently, on divisorial left and right $R$-ideals).
  If $R$ is a Krull ring, then the semigroup $(R,\cdot)$ is an arithmetical maximal order.
  (This remains true in more general settings; see \cite{Chamarie}.)

\item If $R$ is a bounded Dedekind prime ring, then $(R,\cdot)$ is an arithmetical maximal order.

\item Let $K$ be a global field, and denote by $S$ the set of non-archimedean places of $K$.
For $v \in S$ denote by $\cO_v \subset K$ the discrete valuation ring of $v$.
A holomorphy ring in $K$ is a subring $\cO \subset K$ such that $\cO = \bigcap_{v \in S \setminus S_0} \cO_v$ with $S_0 \subset S$ finite and $S_0 \ne \emptyset$ in the function field case.
A central simple algebra $A$ over $K$ is a finite-dimensional $K$-algebra with center $K$ that is simple as a ring.
A classical ($\cO$-)order in $A$ is a subring $R \subset A$ such that $\cO \subset R$, $KR = A$, and $R$ is finitely generated as $\cO$-module (see \cite{McConnell-Robson,Reiner}).
A classical maximal ($\cO$-)order in $A$ is a classical $\cO$-order that is maximal with respect to set inclusion amongst classical $\cO$-orders.
If $R$ is a classical maximal order in $A$, then $R$ is a Dedekind prime ring, a PI ring, and in particular a Krull ring.
Investigating such classical maximal orders is the focus and main motivation of \cref{sec:maxord}.

Particular examples of classical maximal orders are, for instance, $M_n(\cO)$ where $\cO$ is a ring of algebraic integers and $n \in \bN$, as well as classical maximal orders in quaternion algebras over number fields, such as the ring of Hurwitz quaternions $\bZ[i,j,k,\frac{1+i+j+k}{2}]$ with $k=ij=-ji$ and $i^2=j^2=-1$.
\end{enumerate}

\smallskip
\label{monoid-zss}
Monoids of zero-sum sequences are examples of commutative Krull monoids and play an important role in studying non-unique factorizations in commutative Krull monoids.
Indeed, every commutative Krull monoid possesses a transfer homomorphism to a monoid of zero-sum sequences over a subset of its divisor class group (see \cite[Chapter 3.4]{GHK06}).
Under certain conditions, this continues to hold true for arithmetical maximal orders, and it is this transfer homomorphism that was exploited in \cite{Smertnig} to study sets of lengths in this setting, and that we will use in \cref{sec:maxord} to study catenary degrees.
We therefore recall the definition of monoids of zero-sum sequences.

Let $G$ be an additive abelian group, $G_P \subset G$ a subset and let $\cF(G_P)$ be the (multiplicatively written) free abelian monoid with basis $G_P$.
Following the tradition of combinatorial number theory, elements $S \in \cF(G_P)$ are called \emph{sequences over $G_P$}, and are written in the form $S=g_1\cdots g_l$ with $l \in \bN_0$ and $g_1$, $\ldots\,$,~$g_l \in G_P$.
To a sequence $S$ we associate its \emph{length}, $\length{S} = l$, and its \emph{sum}, $\sigma(S) = g_1 + \cdots + g_l \in G$.
We call the submonoid
\[
  \cB(G_P) = \{\, S \in \cF(G_P) : \sigma(S) = 0_G \,\}
\]
of $\cF(G_P)$ the \emph{monoid of zero-sum sequences (over $G_P$)}.
The minimal zero-sum sequences are the atoms of $\cB(G_P)$, and the \emph{Davenport constant} is $\sD(G_P) = \sup\{\, \length{S} : S \in \cA(\cB(G_P)) \,\}$, that is, the supremum of the lengths of minimal zero-sum sequences.

If $S$ is a normalizing Krull monoid, $G$ is its divisor class group, and $G_P$ is the set of classes containing prime divisors, then there exists a transfer homomorphism $S \to \cB(G_P)$ (see \cite[Theorems 6.5 and 4.13]{Geroldinger}).
For this reason, monoids of zero-sum sequences, and in particular the Davenport constant, have been the focus of intensive study in combinatorial and additive number theory (see, for instance, \cite{Geroldinger09,Grynkiewicz}).
Many factorization theoretic invariants of $\cB(G_P)$ can be bounded (or even expressed) in terms $\sD(G_P)$.

If $\sD(G_P)$ is finite, the \emph{Structure Theorem for Sets of Lengths} (\cite[Definition 3.2.3]{Geroldinger09}) holds for $\cB(G_P)$, which implies that sets of lengths are almost arithmetical multiprogressions, with differences described by the set of distances, $\Delta(\cB(G_P))$. In the case $G=G_P$, $\Delta(\cB(G_P))$ is a finite interval starting at $1$ (if it is non-empty).
Due to the existence of a transfer homomorphism, the same is then true for $S$ itself.
Similarly, if $S$ is not half-factorial, its catenary degree is $\sc(S) = \max\{ 2, \sc(\cB(G_P)) \}$ (see \cite[Definition 1.6.1]{GHK06} or \cref{sec:catenary} for the definition of the catenary degree, and \cite[Lemma 3.2.6 and Theorem 3.2.8]{GHK06} or \cref{cor:cat-transfer-semigroup} for this result).

\subsection*{Adyan semigroups}
Since we will have need to introduce many atomic semigroups defined via generators and relations in order to illustrate various points, pathological cases, and obstructions to creating a noncommutative analogue of the commutative theory, and do not desire to expose the reader to the tedious details of checking whether or not the semigroup is indeed cancellative, we recall the notion of Adyan semigroups.

Let $\langle X \mid R\rangle$ be a presentation of a semigroup $S$ with a finite set of generators $X$ and finite set of relations $R$ of the form $u=v$ with $u$ and $v$ non-trivial elements in $\cF^*(X)$.
The \emph{left graph of the presentation} is the graph $G(V,E)$ with vertex set $V=X$ and with an edge $\{a,b\} \in E$ if and only if there is a relation $u=v$ in $R$ where $a$ is the left-most letter in $u$ and $b$ is the left-most letter in $v$.
One similarly defines the \emph{right graph of the presentation}.
A semigroup $S$ is said to be \emph{Adyan} if it has a presentation such that the left and right graphs of the presentation are acyclic; i.e., if they are forests.
For the examples of semigroups in this paper that are defined via generators and relations, it can easily be checked that they are Adyan.
Thus it is important to note the following result which allows one to easily verify that these examples are indeed cancellative (see also \cite[Theorem 4.6]{Remmers} for another proof and a more general result that allows for infinite sets $X$ and $R$).

\begin{proposition}[\cite{Adyan}]\label{adyan}
  Let $S$ be an Adyan semigroup.
  Then $S$ embeds into a group and is therefore cancellative.
\end{proposition}

In particular, if $R$ consists of a single relation $u=v$ (which will often be the case in this manuscript), then $S$ is cancellative if the first letters of $u$ and $v$ are distinct, and the last letters of $u$ and $v$ are distinct.

\section{Distances and Factorizations}\label{sec:factorizations}

In this section we introduce rigorous notions of factorizations and distances between factorizations.
We begin by briefly recalling the concepts of factorizations as well as the usual distance from the commutative setting.
A more detailed account can be found in \cite[Section 1.2]{GHK06}.

Let $S$ be a commutative cancellative semigroup.
If $a \in S$ and $a=u_1\cdots u_k = v_1\cdots v_l$ with $k$, $l \in \bN_0$ and $u_1$, $\ldots\,$,~$u_k$, $v_1$, $\ldots\,$,~$v_l \in \cA(S)$ are two representations of $a$ as products of atoms, then one considers these representations to be the same factorization if $k=l$ and there exists a permutation $\sigma \in \fS_k$ such that $u_i \simeq v_{\sigma(i)}$ for all $i \in [1,k]$.
A fully rigorous notion of factorizations is obtained as follows:
Let $\red S$ be the associated reduced semigroup of $S$.
The \emph{factorization monoid of $S$}, denoted by $\sZ(S)$, is the free abelian monoid $\cF(\cA(\red S))$.
There is a canonical homomorphism $\pi\colon \sZ(S) \to \red S$ mapping a formal product $u_1\cdots u_k$ in $\sZ(S)$ to the product $\prod_{i=1}^k u_i$ in $\red S$.
For $a \in S$, the set $\sZ(a) = \sZ_S(a) = \pi^{-1}(aS^\times)$ is the \emph{set of factorizations of $a$}.

If $X$ is a set and $F = \cF(X)$ is the free abelian monoid on $X$, there is a natural notion of a distance function $\sd_F\colon F \times F \to \bN_0$, defined as follows:
If $x$,~$y \in F$, we may write $x$ and $y$ as
\[
  x=u_1\cdots u_k v_1\cdots v_m \quad\text{and}\quad y=u_1\cdots u_k w_1\cdots w_n
\]
with $k$,~$m$, $n \in \bN_0$ and $u_1$, $\ldots\,$,~$u_k$, $v_1$, $\ldots\,$,~$v_m$, $w_1$, $\ldots\,$,~$w_n \in X$ such that
\[
  \{ v_1, \ldots, v_m \} \cap \{ w_1, \ldots, w_n \} = \emptyset.
\]
We then set $\sd_F(x,y) = \max\{ m, n\}$.
This is a metric on $F$ having the additional property that it is invariant under translations, that is, $\sd_F(xz,yz) = \sd_F(x,y)$ for all $x$,~$y$, $z \in F$.
Moreover $\babs{\length{z} - \length{z}} \le \sd_F(z,z') \le \max\{\length{z}, \length{z'} \}$ for all $z$,~$z' \in \sZ(S)$.

Since $\sZ(S)$ is simply the free abelian monoid on $\cA(\red S)$, in this way a notion of a distance between factorizations is obtained.
It is this distance function that has been a central tool in the investigation of non-unique factorizations in the commutative setting.
For instance, the catenary degree $\sc(S)$ and the tame degree $\st(S)$ are defined in terms of $\sd_{\sZ(S)}$.

For a cancellative small category $H$ we cannot directly imitate the approach to defining factorizations that is used in the commutative setting (taking the path category instead of the free abelian monoid) because associativity may not be a congruence relation on $H$.
Instead we shall take the path category on $\cA(H)$, and afterwards impose a congruence relation to deal with the potential presence of units.
This gives rise to the notion of a \emph{rigid factorization} in cancellative small categories.
In the commutative setting, rigid factorizations differ from the usual notion in the commutative sense in that the order of factors matters.
In general, there does not seem to be an entirely natural choice of distance between rigid factorizations that also coincides with the usual distance for factorizations in the commutative case.
Therefore we introduce an axiomatic notion of a distance.
Each distance $\sd$ gives rise to the notion of \emph{$\sd$-factorizations}, possibly coarser than that of rigid factorizations.
On the other hand, different distances may give rise to the same notion of a \emph{factorization}.

We focus on what appear to be two reasonably defined distances:
The first of these notions, the \emph{permutable distance} $\sd_p$, allows for permutations of irreducible factors of an element, and hence gives rise to the notion of \emph{permutable factorizations} of an element.
In a commutative semigroup, $\sd_p$ coincides with the usual distance defined between any two factorizations of an element.
The second notion, the \emph{rigid distance}, instead corresponds more naturally to the notion of rigid factorizations, and hence does not coincide with the usual distance in the commutative setting.

In addition, for the semigroup of non zero-divisors of a ring, we shall also introduce distances based on \emph{similarity} and \emph{subsimilarity} of atoms, denoted by $\sd_\dsim$ and $\sd_\dsubsim$.
For the semigroup of non zero-divisors of a commutative ring, these also coincide with the usual distance of the commutative setting, and have an advantage over the permutable distance in that they correctly reflect the structure of noncommutative rings in the sense that PIDs are $\sd_\dsim$- and $\sd_\dsubsim$-factorial, while they are not necessarily permutably factorial.

We also show that the coarsest possible distance, $\sd_{\length{.}}$, is based on lengths alone, and that a factorization in that distance corresponds to a length, whence sets of lengths naturally reappear as sets of factorizations in this coarse distance.

\vspace{1.5mm}
\begin{center}
  \emph{Throughout this section, let $H$ be a cancellative small category.}
\end{center}
\vspace{1.5mm}

We now recall the notion of a rigid factorization as defined in \cite[Section 3]{Smertnig}.
Let $\fpath(\cA(H))$ denote the path category on the quiver of atoms of $H$. We define
\[
  H^\times \times_r \fpath(\cA(H)) = \{\, (\varepsilon, y) \in H^\times \times \fpath(\cA(H)) : t(\varepsilon) = s(y) \,\},
\]
and define an associative partial operation on $H^\times \times_r \fpath(\cA(H))$ as follows:
If $(\varepsilon,y)$,~$(\varepsilon', y') \in H^\times \times_r \fpath(\cA(H))$ with $\varepsilon$,~$\varepsilon' \in H^\times$,
\[
  y = (e, u_1, \ldots, u_k, f) \in \fpath(\cA(H)) \;\text{ and }\; y' = (e', v_1, \ldots, v_l, f') \in \fpath(\cA(H)),
\]
then the operation is defined if $t(y) = s(\varepsilon')$, and
\[
  (\varepsilon,y) \cdot (\varepsilon',y') = (\varepsilon, (e, u_1,\ldots, u_k\varepsilon', v_1, \ldots, v_l, f')) \qquad\text{if $k > 0$, }
\]
while $(\varepsilon,y)\cdot(\varepsilon',y') = (\varepsilon\varepsilon', y')$ if $k = 0$.
In this way, $H^\times \times_r \fpath(\cA(H))$ is again a cancellative small category (with identities $\{\, (e, (e,e)) : e \in H_0 \,\}$ that we identify with $H_0$, so that $s(\varepsilon,y) = s(\varepsilon)$ and $t(\varepsilon,y) = t(y)$).
We define a congruence relation $\sim$ on $H^\times \times_r \fpath(\cA(H))$ as follows:
If $(\varepsilon,y)$,~$(\varepsilon',y') \in H^\times \times_r \fpath(\cA(H))$ with $y$,~$y'$ as before, then
$(\varepsilon,y) \sim (\varepsilon',y')$ if $k = l$, $\varepsilon u_1\cdots u_k = \varepsilon' v_1 \cdots v_l \in H$ and either $k=0$ or there exist $\delta_2$,~$\ldots\,$,~$\delta_k \in H^\times$ and $\delta_{k+1} = t(u_k)$ such that
\[
  \varepsilon' v_1 = \varepsilon u_1 \delta_2^{-1} \quad\text{and}\quad v_i = \delta_i u_i \delta_{i+1}^{-1} \quad\text{for all $i \in [2,k]$}.
\]

\begin{definition}
  The \emph{category of rigid factorizations of $H$} is defined as
  \[
    \sZ^*(H) = \big(H^\times \times_r \fpath(\cA(H))\big) / \!\sim.
  \]

For $z \in \sZ^*(H)$ with $z = [(\varepsilon, (e, u_1, \ldots, u_k, f))]_\sim$, we write $z=\rf[\varepsilon]{u_1,\cdots,u_k}$ and denote the partial operation on $\sZ^*(H)$ also  by $\rfop$. The \emph{length} of the rigid factorization $z$ is $\length{z} = k$ and there is a homomorphism $\pi=\pi_H\colon \sZ^*(H) \to H$, induced by multiplication in $H$, explicitly $\pi(z) = \varepsilon u_1 \cdots u_k \in H$. For $a \in H$, we define $\sZ^*(a) = \sZ_H^*(a) = \pi^{-1}(\{a\})$ to be the \emph{set of rigid factorization of $a$}.
\end{definition}

Factoring out by the relation $\sim$ is motivated by the fact that if $e \in H_0$ and $u$,~$v \in \cA(H)$ with $t(u)=e=s(v)$, and $\varepsilon \in H^\times$ is such that $t(\varepsilon)=e$, then we always have $uv = (u\varepsilon^{-1})(\varepsilon v)$, but we do not wish to consider these to be distinct factorizations of $uv$.
Working with $H^\times \times_r \fpath(\cA(H))$ instead of $\fpath(\cA(H))$ ensures that, despite factoring out by $\sim$, every unit of $H$ has a rigid factorization (of length $0$).
Hence the homomorphism $\pi\colon \sZ^*(H) \to H$ is surjective if and only if $H$ is atomic.
In this way we often avoid having to treat units as special cases.

Each atom $u \in \cA(H)$ has a unique rigid factorization, as does each unit of $H$.
Moreover, it is easy to see that these unique rigid factorizations of atoms and units of $S$ are precisely the atoms and units of $\sZ^*(H)$, and thus we have bijections
\[
  \pi\mid_{\cA(\sZ^*(H))}\colon \cA(\sZ^*(H)) \to \cA(H) \quad\text{and}\quad \pi\mid_{\sZ^*(H)^\times}: \sZ^*(H)^\times \to H^\times.
\]
In particular, $H^\times$ embeds into $\sZ^*(H)$ as a subcategory by means of $\varepsilon \mapsto \varepsilon=[(\varepsilon, (t(\varepsilon),t(\varepsilon)))]_\sim$.
Thus we may view $\cA(H)$ and $H^\times$ as subsets of $\sZ^*(H)$.

One can verify, directly from the definition of $\sZ^*(H)$, that if $x$, $y$, $x'$, $y' \in \sZ^*(H)$ with $\rf{x,y} = \rf{x',y'}$ and $\length{x} = \length{x'}$, then there exists $\varepsilon \in H^\times$ such that $x' = \rf{x,\varepsilon^{-1}}$ and $y'= \rf{\varepsilon,y}$.
Thus any representation of a rigid factorization as a product of other rigid factorization is, up to trivial insertions of units, uniquely determined by the lengths of the factors.
Below, we will define the notion of rigid factoriality, and by the property just stated, $\sZ^*(H)$ will turn out to be rigidly factorial.

If $H$ is reduced, then we simply have $\sZ^*(H) \cong \fmon(\cA(H))$, and in this case we identify these two objects.
In particular, if $S$ is a commutative reduced cancellative semigroup, then $\sZ^*(S)$ is the free monoid on $\cA(S)$, while the usual factorization monoid from the commutative setting is the free abelian monoid on $\cA(S)$.
Thus rigid factorizations differ from the usual ones in that the order of atoms matters for rigid factorizations.

We shall sometimes write something akin to ``Let $z=\rf[\varepsilon]{u_1,\cdots,u_k} \in \sZ_H^*(a)$ be a rigid factorization \ldots'', and we will tacitly assume that we are implicitly choosing $k \in \bN_0$, $\varepsilon \in H^\times$, and $u_1$,~$\ldots\,$,~$u_k \in \cA(H)$ representing the given factorization.

With rigid factorizations defined, we are now able to introduce distances between them.

\begin{definition}\label{distance}
  A \emph{global distance on $H$} is a map $\sd \colon \sZ^*(H) \times \sZ^*(H) \to \bN_0$ satisfying the following properties.
  \begin{enumerate}[label=\textup{(\textbf{D\arabic*})},ref=\textup{(D\arabic*)}]
    \item\label{d:ref} $\sd(z,z) = 0$ for all $z \in \sZ^*(H)$.
    \item\label{d:sym} $\sd(z,z') = \sd(z',z)$ for all $z$,~$z' \in \sZ^*(H)$.
    \item\label{d:tri} $\sd(z,z') \le \sd(z,z'') + \sd(z'',z')$ for all $z$,~$z'$,~$z'' \in \sZ^*(H)$.
    \item\label{d:mul} For all $z$,~$z' \in \sZ^*(H)$ with $s(z)=s(z')$ and $x \in \sZ^*(H)$ with $t(x) = s(z)$ it holds that $\sd(x\rfop z, x \rfop z') = \sd(z,z')$, and for all $z$,~$z' \in \sZ^*(H)$ with $t(z)=t(z')$ and $y \in \sZ^*(H)$ with $s(y) = t(z)$ it holds that $\sd(z\rfop y, z'\rfop y) = \sd(z,z')$.
    \item\label{d:len} $\babs{\length{z} - \length{z'}} \le \sd(z,z') \le \max\big\{ \length{z}, \length{z'}, 1 \big\}$ for all $z$,~$z' \in \sZ^*(H)$.
  \end{enumerate}
  Let $L = \{\, (z, z') \in \sZ^*(H) \times \sZ^*(H) : \pi(z) = \pi(z') \,\}$.
  A \emph{distance on $H$} is a map $\sd \colon L \to \bN_0$ satisfying properties \labelcref*{d:ref,d:sym,d:tri,d:mul,d:len} under the additional restrictions on $z$,~$z'$ and $z''$ that $\pi(z)=\pi(z')=\pi(z'')$.
\end{definition}

A distance is only defined between two rigid factorizations $z$ and $z'$ of a fixed element, while a global distance is defined between arbitrary rigid factorizations.
If $\sd$ is a global distance, then $\sd|_L$ is a distance and we simply write $\sd=\sd|_L$.
The concept of a distance suffices to introduce $\sd$-factorizations and to study catenary degrees.
Distances have an advantage over global distances in that they can be extended to the category of principal ideals (see \cref{prop:dist-bij}).
While the particular distances we introduce will generally be global distances, our abstract results will always be stated for distances, the only exception being \subref{d-basic:rigid}, where it is necessary to assume that the given distance is a global distance.
The description of $\equiv_p$ after \cref{rel0} by means of the permutable distance makes use of the fact that the permutable distance is a global distance.

Let $z$,~$z' \in \sZ^*(H)$.
If $\length{z}=\length{z'}=0$ then $z=\varepsilon$ and $z'=\eta$ with $\varepsilon$,~$\eta \in H^\times$.
If, in addition, $\pi(z)=\pi(z')$, then $\varepsilon=\eta$ and hence $z=z'$.
In this case, \labelcref{d:len} together with \labelcref{d:ref} implies $\sd(z,z') \le \max\{\length{z},\length{z'}\}$.
For a distance, the upper bound in \labelcref{d:len} is therefore equivalently to $\sd(z,z') \le \max\{\length{z}, \length{z'}\}$.

In the literature a great number of distances between words in a free monoid have been introduced, see for example \cite[Chapter 11]{Deza-Deza}. Modifying these to account for the potential presence of units, they prove to be a rich source of possible interesting distances to study on $H$.

We now introduce two general constructions for global distances in the present context.

\begin{construction}\label{con}
\envnewline
\begin{enumerate}
  \item\label{con:edit}
    Let $\Omega$ be a non-empty set of symmetric relations on $\sZ^*(H) \times \sZ^*(H)$, and for each $\cR \in \Omega$ let $c_\cR \in \bN_0$ denote its \emph{cost}, subject to the condition that $c_\cR \ge \babs{\length{z} - \length{z'}}$ for all $z$,~$z' \in \sZ^*(H)$ with $z \cR z'$.
  We call $\cR \in \Omega$ an \emph{edit operation}.
  Let $z$, $z' \in \sZ^*(H)$.
  An \emph{edit sequence} from $z$ to $z'$ consists of a finite sequence of relations $\cR_1$,~$\ldots\,$,~$\cR_m \in \Omega$ (repetition is allowed) and factorizations $z=z_0$,~$z_1$, $\ldots\,$,~$z_{m-1}$,~$z_m=z' \in \sZ^*(H)$ such that $z_{i-1} \cR_i z_i$ for all $i \in [1,m]$.
  (Note that the intermediate factorizations may have different products, and even $s(z_i) \ne s(z_{i-1})$ and $t(z_i) \ne t(z_{i-1})$ is permitted.)
  The \emph{cost} of the sequence is $c_{\cR_1} + \cdots + c_{\cR_m}$, and the \emph{length} of the sequence is $m \in \bN_0$.

  We set $\sd(z,z') \in \bN_0 \cup \{\infty\}$ to be the minimal cost of an edit sequence from $z$ to $z'$.
  If $\sd(z,z') < \infty$ for all $z$,~$z' \in \sZ^*(H)$ (that is, for any two $z$, $z' \in \sZ^*(H)$ there exists an edit sequence from $z$ to $z'$), then $\sd\colon L \to \bN_0$ satisfies properties \labelcref{d:ref,d:sym,d:tri}.
  Moreover, $\sd$ satisfies one inequality of \labelcref{d:len}:
  For any edit sequence of minimal length,
  \[
    \sd(z,z') = \sum_{i=1}^m c_{\cR_i} \ge \sum_{i=1}^m \babs{\length{z_i} - \length{z_{i-1}}} \ge \Bigabs{\sum_{i=1}^m \length{z_i} - \length{z_{i-1}}} = \babs{\length{z} - \length{z'}}.
  \]
  To establish that $\sd$ is a global distance, it remains to check \labelcref{d:mul} and the remaining inequality from \labelcref{d:len}.
  We will use this construction to introduce the rigid distance in \cref{distdef}.

  \item\label{con:fab}
    Let $\sim$ be an equivalence relation on the set of atoms $\cA(H)$ of $H$ such that, for all $u$,~$v \in \cA(H)$, $u \simeq v$ implies $u \sim v$.
    We denote by $[u]_\sim$ the $\sim$-equivalence class of $u \in \cA(H)$, and by $F=\cF(\cA(S)/\!\sim)$ the free abelian monoid on the equivalence classes of $\cA(H)$ under the equivalence relation $\sim$.
    Then there exists a homomorphism $\varphi\colon \sZ^*(H) \to F$ such that $\varphi(z) = [u_1]_\sim \cdots [u_k]_\sim$ for all $z=\rf[\varepsilon]{u_1,\cdots,u_k} \in \sZ^*(H)$ with $k \in \bN_0$, $\varepsilon \in H^\times$, and $u_1$,~$\ldots\,$,~$u_k \in \cA(H)$.

    Let $\sd_F\colon F \times F \to \bN_0$ denote the usual distance on the free abelian monoid $F$.
    We obtain a global distance $\sd$ on $H$ by setting $\sd(z,z') = \sd_F(\varphi(z),\varphi(z'))$ for all $z$,~$z' \in \sZ^*(H)$.
    Thus, if $z = \rf[\varepsilon]{u_1,\cdots,u_k}$, $z' = \rf[\eta]{v_1,\cdots,v_l} \in \sZ^*(H)$, we compare the sequences of $\sim$-equivalence classes of $u_1$,~$\ldots\,$,~$u_k$ and $v_1$,~$\ldots\,$,~$v_l$ up to permutation.
    Explicitly, there exists a (uniquely determined) $n \in [0,\min\{k,l\}]$, subsets $I \subset [1,k]$ and $J \subset [1,l]$ of cardinality $\card{I}=\card{J} = n$, and a bijection $\sigma\colon I \to J$ such that $u_{i} \sim v_{\sigma(i)}$ for all $i \in [1,n]$, while $u_i \not \sim u_j$ for all $i \in [1,k] \setminus I$ and $j \in [1,l] \setminus J$.
    Then $\sd(z,z') = \max\{ k-n, l-n \}$.

    The permutable distance, as well as the similarity and subsimilarity distances, introduced in the following definition, will be constructed in this way.
\end{enumerate}
\end{construction}

Using these constructions, we now introduce the (global) distances we will focus on.

\begin{definition}\label{distdef}
  \envnewline
  \begin{enumerate}
    \item\label{distdef:rigid}
      In the \emph{rigid distance}, denoted by $\sd^*$, we allow the replacement of $m \in \bN_0$ consecutive atoms by $n \in \bN_0$ new ones at cost $\max\{m,n,1\}$.
      Explicitly, for all $m$, $n \in \bN_0$ we define an edit operation $\cR_{m,n}$ as follows:
      If $z$, ~$z' \in \sZ^*(H)$, then $z \cR_{m,n} z'$ if and only if there exist $x$,~$y$, $z_0$, $z_0' \in \sZ^*(H)$ such that $\{ \length{z_0}, \length{z_0'} \} = \{ m, n \}$ and one of
      \begin{align*}
        z = x \rfop z_0 \rfop y \;\quad&\text{ and }\quad z' = x \rfop z_0' \rfop y, \\
        z = z_0 \rfop y \;\quad&\text{ and }\quad z' = z_0' \rfop y, \\
        z = x \rfop z_0 \;\quad&\text{ and }\quad z' = x \rfop z_0', \qquad\text{ or} \\
        z = z_0 \;\quad&\text{ and }\quad z'=z_0'
      \end{align*}
      holds.
      We set the cost of $\cR_{m,n}$ to be $\max\{m, n, 1\}$ and set
      \[
        \Omega = \{\, \cR_{m,n} : m,n \in \bN_0 \,\}.
      \]
      The rigid distance $\sd^*$ is the distance defined by these edit operations, as described in \subref{con:edit}.
      (We verify in \cref{lemma:rd} below that $\sd^*$ is a global distance.)

    \item\label{distdef:permutable}
      The \emph{permutable distance}, denoted by $\sd_p$, is defined by means of \subref{con:fab} by setting $\sim\, =\, \simeq$.

    \item\label{distdef:cohn-brungs}
      Let $R$ be a ring and $S = R^\bullet$ the cancellative semigroup of non zero-divisors of $R$.
      Two elements $a$,~$a' \in R$ are \emph{similar} if $R/Ra \cong R/Ra'$ as left $R$-modules, and they are \emph{subsimilar} if there exist monomorphisms $R/Ra \hookrightarrow R/Ra'$ and $R/Ra' \hookrightarrow R/Ra$.
      These are equivalence relations on $S$.
      If $a \simeq a'$, then $a$ and $a'$ are similar, and hence subsimilar.
      Using \subref{con:fab}, similarity therefore gives rise to the \emph{similarity distance}, denoted by $\sd_\dsim$, and subsimilarity gives rise to the \emph{subsimilarity distance}, denoted by $\sd_\dsubsim$.
  \end{enumerate}
\end{definition}

\begin{remark}\label{some-dist}
  \envnewline
  \begin{enumerate}
  \item
    If $S$ is a commutative reduced cancellative semigroup, then $\sd_p$ coincides with the usual distance.
    To differentiate it from a generic distance we will always write $\sd_p$ for the usual distance in the commutative setting.

  \item
    The rigid distance is derived from the editing (or Levenshtein) distance on a free monoid:
    There, one permits the insertion, deletion and replacement of a single letter in a word at cost $1$.
    The variation from the definition for a free monoid accounts for the nature of the partial operation, in which a one-by-one deletion and insertion of atoms may not be possible, as well as the presence of units.
    If $z$, $z' \in \sZ^*(H)$ with $\sd^*(z,z') = 0$, then $z=z'$ since all edit operations have cost at least $1$.
    Thus the rigid distance is sufficiently fine to be able to distinguish between two distinct rigid factorizations.

    The definition of the edit operations $\cR_{m,n}$ seems repetitive.
    However, since it is possible that $s(z) \ne s(z')$ or $t(z) \ne t(z')$, one cannot assume that all the listed cases are special cases of $z = \rf{x,z_0,y}$ and $z' = \rf{x,z_0',y}$.

    It may seem natural to set the cost of $\cR_{m,n}$ to $\max\{m,n\}$ instead of $\max\{m,n,1\}$.
    However, it would then be permitted to insert or remove units in arbitrary places at cost $0$.
    It would then be possible to have $\sd^*(z,z') =0$ for two distinct factorizations $z$,~$z' \in \sZ^*(H)$, even if $\pi(z)=\pi(z')$.
    Indeed, consider $S = \langle a,\varepsilon \mid \varepsilon^2 = 1, \varepsilon a^2 = a^2 \varepsilon \rangle$.
    We first check that $S$ is cancellative.
    Since the right hand side of the relation $\varepsilon^2 = 1$ is trivial, we cannot employ Adyan's result.
    However, by mapping
    \[
      \varepsilon \mapsto \begin{pmatrix} 0 & 1 \\ 1 & 0 \end{pmatrix}
      \quad\text{and}\quad
      a \mapsto \begin{pmatrix} 0 & 2 \\ 1 & 0 \end{pmatrix},
    \]
    we obtain a homomorphism $S \to M_2(\bZ)^\bullet$.
    Using the fact that every element of $S$ affords a representation of the form $a^{2m+r} (\varepsilon a)^n \varepsilon^s$ with $m$,~$n \in \bN_0$ and $r$,~$s \in \{ 0,1 \}$, one can check directly that this homomorphism is injective.
    Therefore $S$ can be realized as a subsemigroup of $M_2(\bZ)^\bullet$ and is cancellative.
    Now note that $z=\rf{\varepsilon a, a \varepsilon}$ and $z'=\rf{a,a}$ are distinct rigid factorizations of $a^2$.
    Indeed, suppose otherwise.
    Then there exists $\eta \in S^\times$ such that $\varepsilon a = a \eta$.
    However, $S^\times=\{1,\varepsilon\}$, and thus $z$ and $z'$ are distinct.
    However, $(\rf{a,a}) \cR_{0,0} (\rf{\varepsilon a, a})$, and $(\rf{\varepsilon a, a})\cR_{0,0}(\rf{\varepsilon a, a \varepsilon})$ would give an edit sequence of cost $0$.

  \item
    In \cite[Chapter 3]{Cohn85}, P.~M.~Cohn uses the notion of \emph{similarity} to define a concept of unique factorization in noncommutative rings.
    Similarly, in \cite{Brungs}, the slightly weaker notion of \emph{subsimilarity} is used to introduce such a concept.
    If $z$, $z' \in \sZ^*(H)$ with $\pi(z)=\pi(z')$, then $\sd_\dsim(z,z') = 0$ if and only if $z$ and $z'$ are the same factorization in the sense of P.~M.~Cohn, and similarly $\sd_\dsubsim(z,z') = 0$ if and only if $z$ and $z'$ are the same factorization in the sense of Brungs.

    There is also a purely multiplicative characterization of subsimilarity (see \cite[Lemma 2]{Brungs}, but note that Brungs assumes that $R$ is a domain): Two elements $a$,~$a'$ are subsimilar if and only if there exist elements $c$,~$c' \in R$ such that
    \[
      Ra \cap Rc' = Ra'c' \quad\text{and}\quad Ra' \cap Rc = Rac,
    \]
    with $c$ and $c'$ satisfying the additional property that, for all $r \in R$, $rc \in Rac$ implies $r \in Ra$ and $rc' \in Ra'c'$ implies $r \in Ra'$ (this is a weak form of cancellativity for $c$ and $c'$).

    If $R$ is commutative, then the notions of subsimilarity and similarity coincide with associativity, and hence both of these distances coincide with the usual one in $R$:
    If $a$, $a' \in R^\bullet$ are subsimilar, then $Ra = \ann_R(R/Ra) = \ann_R(R/Ra') = Ra'$, and hence $a \simeq a'$.
\end{enumerate}
\end{remark}

The first part of the following lemma shows that any edit sequence consisting of the edit operations which define the rigid distance can be transformed into an edit sequence of equal or lower cost which consists of pairwise disjoint replacements (the idea is related to the use of traces in studying the Levenshtein distance in a free monoid, see \cite{Wagner-Fischer}).
The details of the proof are somewhat technical, but essentially, given any edit sequence, we can merge two subsequent overlapping edit operations into a single edit operation whose cost does not exceed the combined cost of the two operations.
We then use this characterization to establish that the rigid distance is a global distance.

\begin{lemma}\label{lemma:rd}
  \envnewline
  \begin{enumerate}
    \item\label{lemma:rd:char}
    Let $z$,~$z' \in \sZ^*(H)$ and let $N \in \bN_0$.
    Then $\sd^*(z,z') \le N$ if and only if there exist $n \in \bN$, $x_1$,~$\ldots\,$,~$x_{n-1} \in \sZ^*(H)$ and $y_1$, $y_1'$,~$\ldots\,$,~$y_n$, $y_n' \in \sZ^*(H)$ such that
    \begin{align*}
      z  &= \rf{y_1, x_1, y_2, \cdots, y_{n-1}, x_{n-1}, y_n}\text{ and}\\
      z' &= \rf{y_1', x_1, y_2', \cdots, y_{n-1}', x_{n-1}, y_n'}
    \end{align*}
    with $\sum_{i\in I} \max\{\length{y_i}, \length{y_i'}, 1\} \le N$ where $I = \{\, i \in [1,n] : y_i \ne s(y_i)\text{ or } y_i' \ne s(y_i') \,\}$.

  \item\label{lemma:rd:coprime}
    In the representation in \labelcref*{lemma:rd:char} we can, in addition, assume that either
    \begin{enumerate}
      \item for all $i \in [2,n]$ the suffixes $\rf{y_i,x_i,\cdots,x_{n-1},y_n}$ of $z$ and $\rf{y_i',x_i,\cdots,x_{n-1},y_n'}$ of $z'$ are left coprime, or
      \item for all $i \in [1,n-1]$ the prefixes $\rf{y_1, x_1, \cdots, x_{i-1},y_i}$ of $z$ and $\rf{y_1',x_1,\cdots,x_{i-1},y_i'}$ of $z'$ are right coprime.
    \end{enumerate}

  \item\label{lemma:rd:mul}
    Let $z$, $z' \in \sZ^*(H)$ and let $x$,~$y \in \sZ^*(H)$.
    If $t(x) = s(z) = s(z')$, then $\sd^*(\rf{x,z}, \rf{x,z'}) = \sd^*(z,z')$, and if $s(y) = t(z) = t(z')$, then $\sd^*(\rf{z,y},\rf{z',y}) = \sd^*(z,z')$.

  \item\label{lemma:rd:dist} The rigid distance $\sd^*$ is a global distance on $H$.
  \end{enumerate}
\end{lemma}

\begin{proof}
  For $k$,~$l \in \bN_0$, let $\cR_{k,l}$ denote the edit operations from \subref{distdef:rigid}, and let $\Omega$ denote the set consisting of all such edit operations.
  We recall: If $x$,~$y$, $x'$, $y' \in \sZ^*(H)$ with $\rf{x,y} = \rf{x',y'}$ and $\length{x} = \length{x'}$, then there exists $\varepsilon \in H^\times$ such that $x' = \rf{x,\varepsilon^{-1}}$ and $y'= \rf{\varepsilon,y}$.
  We will make use of this property throughout the proof.

  \ref*{lemma:rd:char}
  Suppose first that $z$ and $z'$ are of the described form.
  Then we can clearly construct an edit sequence from $z$ to $z'$ in the operations from $\Omega$ and of cost at most $N$:
  We successively replace $y_i$ by $y_i'$ for all $i \in I$ with cost at most $\max\{ \length{y_i},\length{y_i'},1\}$.
  Thus $\sd^*(z,z') \le N$.

  For the converse, suppose that $\sd^*(z,z') \le N$.
  Fix an edit sequence from $z$ to $z'$ with cost at most $N$ and with length $m \in \bN_0$.
  For each $i \in [1,m]$, let $\cR_i \in \Omega$ and let $z=z_0$,~$\ldots\,$,~$z_m=z' \in \sZ^*(H)$ be such that $z_{i-1} \cR_i z_i$ for all $i \in [1,m]$ and $\sum_{i=1}^m c_{\cR_i} \le N$.
  We proceed by induction on the length $m$ of the edit sequence.
  If $m=0$, then $z=z'$ and we simply set $n=2$, $x_1=z$, $y_1=y_1'=s(z)$ and $y_2=y_2'=t(z)$.
  Now suppose that $m \ge 1$ and that the claim holds for sequences of length $m-1$.
  Note that $N \ge 1$, since all edit operations in $\Omega$ have cost at least $1$.
  Since $z=z_0$,~$\ldots\,$,~$z_{m-1}$ is a sequence from $z$ to $z_{m-1}$ of length $m-1$, the induction hypothesis implies that there exist $r \in \bN$, $\hat x_1$,~$\ldots\,$,~$\hat x_{r-1} \in \sZ^*(H)$ and $\hat y_1$, $\hat y_1'$,~$\ldots\,$,~$\hat y_{r}$, $\hat y_{r}' \in \sZ^*(H)$ such that
  \begin{equation}\label{eq:rd-ih}
  \begin{aligned}
    z  &= \rf{\hat y_1,\hat x_1, \hat y_2,\cdots, \hat y_{r-1}, \hat x_{r-1}, \hat y_{r}}, \\
    z_{m-1} &= \rf{\hat y_1',\hat x_1, \hat y_2', \cdots, \hat y_{r-1}', \hat x_{r-1}, \hat y_{r}'},
  \end{aligned}
  \end{equation}
  and $\sum_{i \in \hat I} \max\{\length{\hat y_i}, \length{\hat y_i'}, 1\} \le N - c_{\cR_m}$ where $\hat I = \{\, i \in [1,r] : \hat y_i \ne s(\hat y_i)\text{ or } \hat y_i' \ne s(\hat y_i') \,\}$.

  Since $z_{m-1} \cR_m z'$, there exist $\hat x$, $\hat y$, $\hat z$, $\hat z' \in \sZ^*(H)$ with $\max\{ \length{\hat z}, \length{\hat z'}, 1 \} = c_{\cR_m}$ and such that one of the following holds:
  \begin{align*}
    z_{m-1} = \hat x \rfop \hat z \rfop \hat y \;\quad&\text{ and }\quad z' = \hat x \rfop \hat z' \rfop \hat y, \\
    z_{m-1} = \hat z \rfop \hat y \;\quad&\text{ and }\quad z' = \hat z' \rfop \hat y, \\
    z_{m-1} = \hat x \rfop \hat z \;\quad&\text{ and }\quad z' = \hat x \rfop \hat z', \qquad\text{ or} \\
    z_{m-1} = \hat z \;\quad&\text{ and }\quad z'=\hat z'.
  \end{align*}
  We first consider the case where $z_{m-1} = \rf{\hat x,\hat z, \hat y}$ and $z' = \rf{\hat x,\hat z',\hat y}$. The other cases will be analogous to special cases of this one.
  To simplify the notation in what follows, for $i$,~$j \in [0,r]$, we set
  \begin{align*}
    P_{i} &= \rf{s(\hat y_1'),\hat y_{1}',\hat x_{1}, \hat y_2', \cdots, \hat y_{i-1}', \hat x_{i-1}, \hat y_i'}\quad\text{ and} \\
    S_{j} &= \rf{\hat y_{j+1}',\hat x_{j+1},\hat y_{j+2}', \cdots, \hat y_{r-1}', \hat x_{r-1}, \hat y_r', t(\hat y_r')}.
  \end{align*}
  These are the prefixes of $z_{m-1}$ ending in $\hat y_i'$, respectively the suffixes of $z_{m-1}$ starting with $\hat y_{j+1}'$, for $i$,~$j \in [0,r]$.
  Note that $P_0=s(\hat y_1')$ and $S_r = t(\hat y_r')$ are empty products.

  Let $k \in [0,r]$ be maximal and $l \in [k,r]$ be minimal such that
  \[
    \length{P_{k}} \le \length{\hat x}\quad\text{and}\quad \length{S_{l}} \le \length{\hat y}.
  \]

  We first deal with some extremal cases.
  If $k=0$ and $l = r$, then we set $n=1$, $y_1 = z$ and $y_1' = z'$.
  Then
  \begin{align*}
    \length{z'} &= \length{\hat x} + \length{ \hat y } + \length{\hat z'} \le \sum_{i \in \{1,r\}} \length{\hat y_i'} + c_{\cR_m}\quad\text{ and}\\
    \length{z} &= \sum_{i=1}^{r-1} \length{\hat x_i} + \sum_{i=1}^r \length{\hat y_i} \le c_{\cR_m} + \sum_{i \in \hat I} \max\{ \length{\hat y_i}, \length{\hat y_i'} \},
  \end{align*}
  and thus $\max\{ \length{z}, \length{z'}, 1 \} \le N$.

  If $k=0$ and $l < r$, then, by enlarging $\hat z$ and $\hat z'$ if necessary by at most $\length{\hat y_l'}$ elements, we may assume $\length{S_l} \le \length{\hat y} \le \length{\rf{\hat x_l,S_l}}$.
  Then there exist $x_1$, $y_{1,1} \in \sZ^*(H)$ such that $\hat x_l = \rf{y_{1,1},x_1}$ and $\hat y = \rf{x_1,S_l}$.
  Setting $y_1 = \rf{\hat y_1, \hat x_1, \cdots, \hat x_{l-1}, \hat y_{l} ,y_{1,1}}$ and $y_1' = \rf{\hat x, \hat z'}$, we have
  \begin{align*}
    z  &= \rf{y_1,x_1,\hat y_{l+1},\hat x_{l+1},\cdots,\hat x_{r-1},\hat y_r}\quad\text{ and} \\
    z' &= \rf{y_1',x_1,\hat y_{l+1}',\hat x_{l+1},\cdots,\hat x_{r-1},\hat y_r'}.
  \end{align*}
  We set $n=r-l+1$,\; $y_i = \hat y_{i+l-1}$ and $y_i'= \hat y_{i+l-1}'$ for all $i \in [2,n]$, and $x_i = \hat x_{i+l-1}$ for all $i \in [2,n-1]$.
Then
  \begin{align*}
    \length{y_1} &= \sum_{i=1}^{l-1} \length{\hat x_i} + \length{y_{1,1}} + \sum_{i=1}^l \length{\hat y_i} \le c_{\cR_m} + \sum_{i \in \hat I \cap [1,l]} \max\{ \length{\hat y_i}, \length{\hat y_i'} \}\quad\text{ and}\\
    \length{y_1'} &= \length{\hat x} + \length{\hat z'} \le \length{\hat y_1'} + c_{\cR_m} + \length{\hat y_l'}.
  \end{align*}
  Thus $\max\{\length{y_1}, \length{y_1'}, 1\} \le N - \sum_{i \in \hat I \cap [l+1,r]} \max\{\length{\hat y_i}, \length{\hat y_i'},1 \}$.

  The case $k>0$ and $l = r$ is analogous to the previous case.
  We can assume from now on that $k$,~$l \in [1,r-1]$.
  Suppose first that $k=l$.
  Comparing the following two representations of $z_{m-1}$:
  \[
    z_{m-1}=\rf{\hat x, \hat z, \hat y} = \rf{\hat y_1',\hat x_1, \hat y_2', \cdots, \hat y_{r-1}', \hat x_{r-1}, \hat y_{r}'},
  \]
  it follows that there exist $x_k$,~$x_{k+1} \in \sZ^*(H)$ such that $\hat x_k = \rf{x_k,\hat z,x_{k+1}}$,\; $\hat x = \rf{P_k,x_k}$, and $\hat y = \rf{x_{k+1},S_k}$.
  Then
  \begin{align*}
    z  &= \rf{\hat y_1, \cdots, \hat y_{k}, x_k, \hat z, x_{k+1}, \hat y_{k+1}, \cdots, \hat y_{r}}\quad\text{ and} \\
    z' &= \rf{\hat y_1', \cdots, \hat y_{k}', x_k, \hat z', x_{k+1}, \hat y_{k+1}',\cdots, \hat y_{r}'}.
  \end{align*}
  Setting $n=r+1$,\; $y_{k+1}=\hat z$,\; $y_{k+1}' = \hat z'$,\; $x_i=\hat x_i$ for all $i \in [1,k-1]$,\; $y_i = \hat y_i$ and $y_i' = \hat y_i'$ for all $i \in [1,k]$, $x_{i} = \hat x_{i-1}$ for all $i \in [k+2,r]$, and $y_i = \hat y_{i-1}$ and $y_i' = \hat y_{i-1}'$ for all $i \in [k+2,r+1]$, the claim follows since $\max\{\length{\hat z},\length{\hat z'}, 1\} \le c_{\cR_m}$.

  Now suppose that $k < l$ with $k$, $l \in [1,r-1]$.
  Enlarging $\hat z$ and $\hat z'$ if necessary by at most $\length{\hat y_{k+1}'} + \length{\hat y_{l}'}$ elements if $k+1 < l$, respectively by at most $\length{\hat y_l'}$ elements if $k+1=l$, we may further assume
  \[
    \length{P_k} \le \length{\hat x} \le \length{\rf{P_{k},\hat x_k}}\quad\text{and}\quad \length{S_l} \le \length{\hat y} \le \length{\rf{\hat x_{l}, S_{l}}}.
  \]
  Then there exist $x_k$,~$x_{k+1}$, $y_{k+1,1}$, $y_{k+1,2} \in \sZ^*(H)$ such that $\hat x_k = \rf{x_k,y_{k+1,1}}$,\; $\hat x_l = \rf{y_{k+1,2},x_{k+1}}$,\; $\hat x = \rf{P_k,x_k}$,\;  $\hat y = \rf{x_{k+1},S_l}$, and $\hat z = \rf{y_{k+1,1},\hat y_{k+1}',\hat x_{k+1},\cdots,\hat x_{l-1}, \hat y_l', y_{k+1,2}}$.
  Thus
  \begin{align*}
    z  &= \rf{\hat y_1, \cdots, \hat y_{k}, x_k, \hat z, x_{k+1}, \hat y_{l+1}, \cdots, \hat y_{r}}\quad\text{ and} \\
    z' &= \rf{\hat y_1', \cdots, \hat y_{k}', x_k, \hat z', x_{k+1}, \hat y_{l+1}',\cdots, \hat y_{r}'}.
  \end{align*}
  We set $n = r - l + k +1$,\; $y_{k+1} = \hat z$,\; $y_{k+1}' = \hat z'$,\; $x_i = \hat x_i$ for all $i \in [1,k-1]$,\; $y_i = \hat y_i$ and $y_i' = \hat y_i'$ for all $i \in [1,k]$,\; $x_i = \hat x_{i-k+l-1}$ for all $i \in [k+2,n-1]$, and $y_i = \hat y_{i-k+l-1}$ and $y_i' = \hat y_{i-k+l-1}'$ for all $i \in [k+2,n]$.
  If $k+1=l$, then $\max\{ \length{\hat z}, \length{\hat z'},1 \} \le c_{\cR_m} + \length{\hat y_l'}$, and if $k+1 <l$, then $\max\{ \length{\hat z}, \length{\hat z'},1 \} \le c_{\cR_m} + \length{\hat y_{k+1}'} + \length{\hat y_l'}$.
  Thus, in any case, with $I = \{\, i \in [1,n] : y_i \ne s(y_i)\text{ or }y_i' \ne s(y_i') \,\}$ we have
  \[
    \sum_{i \in I} \max\{\length{y_i},\length{y_i'},1\} \le \sum_{i \in \hat I} \max\{\length{\hat y_i}, \length{\hat y_i'},1\} + c_{\cR_m}.
  \]
  Hence the claim is verified in the case where $z_{m-1} = \rf{\hat x, \hat z, \hat y}$ and $z' = \rf{\hat x, \hat z', \hat y}$.
  If $z_{m-1} = \rf{\hat z, \hat y}$ and $z'=\rf{\hat z',\hat y}$, then the proof is similar to the case $k=0$ above, noting that it is possible that $s(\hat z) \ne s(\hat z')$ and hence this case is not strictly a special case where $x=t(x)$.
  If $z_{m-1} = \rf{\hat x, \hat z}$ and $z'=\rf{\hat x, \hat z'}$, then the proof is similar to the case $l=r$ above.
  If $z_{m-1} = \hat z$ and $z' = \hat z'$, then the proof is similar to the case $k=0$ and $l=r$.

  \ref*{lemma:rd:coprime}
  We show that the suffixes can be chosen to be left coprime.
  The basic idea here is that a common left factor of a suffix may be moved into the $x_i$ preceding the suffix.
  Let $i \in [2,n]$.
  Suppose that $\rf{y_i,x_i,\cdots,x_{n-1},y_n} = \rf{a,b}$ and $\rf{y_i',x_i,\cdots,x_{n-1},y_n'} = \rf{a,c}$ for some $a$,~$b$, $c \in \sZ^*(H)$.
  We may assume that $\length{a}$ is maximal.
  Then there exist $k$, $l \in [i,n]$ and $y_{k,1}$, $y_{k,2}$, $y_{l,1}'$, $y_{l,2}' \in \sZ^*(H)$ such that $y_k = \rf{y_{k,1}, y_{k,2}}$,\; $y_l' = \rf{y_{l,1}', y_{l,2}'}$,
  \begin{equation}\label{eq:gcd}
  \begin{aligned}
    &\rf{y_i, x_i,y_{i+1}, \cdots,y_{k-1},x_{k-1},y_{k,1}} \!\!\!\! &=&\;\, a,\text{ and}\\
    &\rf{y_i',x_i,y_{i+1}',\cdots,y_{l-1}',x_{l-1},y_{l,1}'} \!\!\!\! &=&\;\, a.
  \end{aligned}
  \end{equation}
  By swapping the roles of $z$ and $z'$ if necessary, we may, without restriction, assume $k \le l$.
  We set
  \begin{align*}
    \hat x_{i-1} &= x_{i-1} \rfop a,
    &\hat y_i     &= \rf{y_{k,2},x_{k},y_{k+1},\cdots,x_{l-1},y_l},
    &\hat y_i'    &= y_{l,2}',
  \end{align*}
  and $\hat x_j=t(\hat y_i)$ for all $j \in [i,l-1]$ as well as $\hat y_j' = \hat y_j = t(\hat y_i)$ for all $j \in [i+1,l]$.
  Then, comparing lengths in \cref{eq:gcd},
  \[
    \sum_{j=k}^{l-1} \length{x_j} = \sum_{j=i}^{k-1} \length{y_i} + \length{y_{k,1}} - \sum_{j=i}^{l-1} \length{y_j'} - \length{y_{l,1}'},
  \]
  and thus
  \[
    \length{\hat y_i} = \sum_{j=k}^{l-1} \length{x_j} + \sum_{j=k+1}^l \length{y_j} + \length{y_{k,2}} = \sum_{j=i}^l \length{y_i} - \sum_{j=i}^{l-1} \length{y_j'} - \length{y_{l,1}'} \le \sum_{j=i}^l \length{y_i}.
  \]
  Clearly $\length{\hat y_i'} \le \length{y_l'}$.
  If $I \cap [i,l] \ne \emptyset$, then
  \[
    \max\{ \length{\hat y_i}, \length{\hat y_i'}, 1 \} \le \sum_{j \in I \cap [i,l]} \max\{ \length{y_j},  \length{y_j'}, 1 \}.
  \]
  Otherwise, $y_j = s(y_j) = y_j'$ for all $j \in [i,l]$.
  Then \cref{eq:gcd} implies $k=l$ and also $y_{k,1} = y_{l,1}'$.
  Modifying $a$ by a unit if necessary, we may take $y_{k,1} = y_{l,1}' = t(a)$, and hence also $y_{k,2}=y_{l,2}'=t(a)$.
  Thus $\hat y_i' = \hat y_i = s(\hat y_i)$ is trivial.

  Note that we only need to modify the representation to the right of $y_{i-1}$.
  Thus, working our way from left to right, we may ensure that the suffixes are left coprime.

  \ref*{lemma:rd:mul}
  We show $\sd^*(\rf{x,z}, \rf{x,z'}) = \sd^*(z,z')$, and begin by showing $\sd^*(\rf{x,z}, \rf{x,z'}) \le \sd^*(z,z')$.
  Suppose $N = \sd^*(z,z')$ and take a representation of $z$ and $z'$ as in \ref*{lemma:rd:char}.
  Since $s(z) = s(z')$, we have $s(y_1) = s(y_1')$.
  Thus we may multiply both representations by $x$ from the left.
  Again by \ref*{lemma:rd:char}, this implies $\sd^*(\rf{x,z}, \rf{x,z'}) \le N$.

  We now show $\sd^*(\rf{x,z}, \rf{x,z'}) \ge \sd^*(z,z')$.
  Let $N = \sd^*(x\rfop z, x \rfop z')$.
  By \labelcref{lemma:rd:char}, there exist $n \in \bN$, $x_1$,~$\ldots\,$,~$x_{n-1} \in \sZ^*(H)$ and $y_1$, $y_1'$,~$\ldots\,$,~$y_n$, $y_n' \in \sZ^*(H)$ such that
    \begin{align*}
      \rf{x,z}  &= \rf{y_1,x_1,y_2, \cdots, x_{n-1}, y_n}\quad\text{ and}\\
      \rf{x,z'} &= \rf{y_1',x_1,y_2', \cdots, x_{n-1}, y_n'}
    \end{align*}
    with $\sum_{i \in I} \max\{\length{y_i}, \length{y_i'},1\} \le N$ where $I = \{\, i \in [1,n] : y_i\ne s(y_i)\text{ or }y_i' \ne s(y_i') \,\}$.
    By \ref*{lemma:rd:coprime} we may further assume that $\rf{y_2,x_2,\cdots,x_{n-1},y_n}$ and $\rf{y_2',x_2,\cdots,x_{n-1}, y_n'}$ are left coprime.
    If $y_1=y_1'=s(y_1)$, then $x_1 = x \rfop \hat x_0$ with $\hat x_0 \in \sZ^*(H)$.
    Cancelling $x$ on the left, we obtain representations of $z$ and $z'$ as in \ref*{lemma:rd:char}, and conclude $\sd^*(z,z') \le \sd^*(x \rfop z, x \rfop z')$.

    Now suppose that $y_1$ or $y_1'$ is non-trivial.
    Due to the coprimality condition, we must have $\rf{y_1,x_1} = \rf{x,a}$ and $\rf{y_1',x_1}=\rf{x,b}$ with $a$,~$b \in \sZ^*(H)$.
    Let $a=\rf{\hat y_1,\hat x_1}$ and $b = \rf{\hat y_1',\hat x_1}$ with $\hat y_1$,~$\hat y_1'$,~$\hat x_1 \in \sZ^*(H)$ and $\length{\hat x_1}$ chosen to be maximal.
    Since $\rf{y_1,x_1}$ and $\rf{y_1',x_1}$ have common right divisor $x_1$, we have at least $\length{\hat x_1} \ge \length{x_1} - (\length{x} - \min\{\length{y_1},\length{y_1'}\})$.
    Thus
    \[
      \length{\hat y_1} = \length{y_1} + \length{x_1} - \length{x} - \length{\hat x_1} \le \length{y_1} - \min\{\length{y_1},\length{y_1'}\} \le \length{y_1},
    \]
    and similarly $\length{\hat y_1'} \le \length{y_1'}$.
    Therefore $\max\{\length{\hat y_1},\length{\hat y_1'},1\} \le \max\{\length{y_1}, \length{y_1'}, 1\}$ and, applying \ref*{lemma:rd:char} to
    \begin{align*}
      \rf{z}  &= \rf{\hat y_1, \hat x_1, y_2, \cdots, x_{n-1}, y_n}\quad\text{ and}\\
      \rf{z'} &= \rf{\hat y_1',\hat x_1, y_2', \cdots, x_{n-1}, y_n'},
    \end{align*}
    the claim follows.

  \ref*{lemma:rd:dist}
  Let $z$, $z' \in \sZ^*(H)$.
  By the general properties of the construction, \labelcref{d:ref,d:sym,d:tri}, and one inequality from \labelcref{d:len} hold.
  It remains to show that $\sd^*(z,z') \le \max\{ \length{z}, \length{z'}, 1  \}$, and that \labelcref{d:mul} holds.
  We have $z \cR_{\length{z}, \length{z'}} z'$, and this operation has cost $\max\{\length{z}, \length{z'}, 1\}$.
  Property \labelcref{d:mul} follows from \ref*{lemma:rd:mul}.
\end{proof}

If $\sd$ and $\sd'$ are global distances on $H$ with $\sd(z,z') \le \sd'(z,z')$ for all $z$,~$z' \in \sZ^*(H)$, we shall say that $\sd'$ is \emph{finer} than $\sd$ and $\sd$ is \emph{coarser} than $\sd'$.
If $\sd$ and $\sd'$ are distances on $H$ with $\sd(z,z') \le \sd'(z,z')$ for all $z$,~$z' \in \sZ^*(H)$ with $\pi(z) = \pi(z')$, we shall say that $\sd'$ is \emph{finer} than $\sd$ and $\sd$ is \emph{coarser} than $\sd'$.

We note some basic properties of distances.

\begin{lemma}\label{d-basic}
  Let $\sd$ be a distance on $H$.
  \begin{enumerate}
    \item\label{d-basic:mul}
      For $z_1$, $z_2$, $z_3$, $z_4 \in \sZ^*(H)$ with $\pi(z_1)=\pi(z_3)$, $\pi(z_2)=\pi(z_4)$, $t(z_1)=s(z_2)$, and $t(z_3)=s(z_4)$ we have $\sd(z_1 \rfop z_2, z_3 \rfop z_4) \le \sd(z_1, z_3) + \sd(z_2, z_4)$.
    \item\label{d-basic:cong} The relation $\sim_\sd$ on $\sZ^*(H)$, defined by $z \sim_\sd z'$ if and only if $\pi(z) = \pi(z')$ and $\sd(z,z') = 0$, is a congruence relation.

    \item\label{d-basic:rigid} Any global distance $\sd$ on $H$ is coarser than $\sd^*$.
  \end{enumerate}
\end{lemma}

\begin{proof}
  \ref*{d-basic:mul}
  From the triangle inequality \ref{d:tri} we obtain $\sd(z_1 \rfop z_2, z_3 \rfop z_4) \le \sd(z_1 \rfop z_2, z_3 \rfop z_2) + \sd(z_3 \rfop z_2,\, z_3 \rfop z_4)$.
  The translation invariance \ref{d:mul} implies $\sd(z_1 \rfop z_2, z_3 \rfop z_2) = \sd(z_1, z_3)$ and $\sd(z_3 \rfop z_2, z_3 \rfop z_4) = \sd(z_2, z_4)$, and therefore we have $\sd(z_1 \rfop z_2, z_3 \rfop z_4) \le \sd(z_1, z_3) + \sd(z_2, z_4)$.

  \ref*{d-basic:cong}
  It is immediate that $\sim_\sd$ gives reflexive, symmetric and transitive relations on $\sZ^*(H)(e,f)$ for all $e$,~$f \in H_0$.
  Let $z$,~$z'$, $w$, $w' \in \sZ^*(S)$ with $s(z)=s(z')$, $t(z)=t(z')=s(w)=s(w')$, and $t(w)=t(w')$. Moreover assume that $\pi(z)=\pi(z')$, $\sd(z,z') = 0$, $\pi(w)=\pi(w')$ and $\sd(w,w') = 0$.
  We must show that $z \rfop w \sim_\sd z' \rfop w'$.
  Since $\pi(w)=\pi(w')$, $\pi(z)=\pi(z')$ and $\pi$ is a homomorphism of small categories, we also have $\pi(z \rfop w)=\pi(z' \rfop w')$.
  Moreover, by \ref*{d-basic:mul}, $\sd(z \rfop w,z' \rfop w') \le \sd(z,z') + \sd(w,w') = 0$.

  \ref*{d-basic:rigid}
  Let $\sd$ be a global distance on $H$.
  Let $z$,~$z' \in \sZ^*(H)$ and let $N = \sd^*(z,z')$.
  By the definition of $\sd^*$, there exist $l \in \bN_0$, $m_1$,~$n_1$,~$\ldots\,$, $m_l$,~$n_l \in \bN_0$ and $z=z_0$,~$\ldots\,$, $z_l = z' \in \sZ^*(H)$ such that $z_{i-1} \cR_{m,n} z_i$ for all $i \in [1,l]$ and $c_{\cR_{m_1,n_1}} + \cdots + c_{\cR_{m_l,n_l}} = N$.
  By the definition of $\cR_{m_i,n_i}$ and property \labelcref{d:mul} of $\sd$, we find $\sd(z_{i-1},z_i) \le c_{\cR_{m_i,n_i}}$ for all $i \in [1,l]$.
  From the triangle inequality \labelcref{d:tri} we conclude $\sd(z,z') \le N$.
\end{proof}

By application of \subref{d-basic:rigid}, the rigid distance plays a special role in that it is the finest global distance.
Note that $\sd_{\length{\cdot}}(z,z') = \babs{\length{z} - \length{z'}}$ also defines a global distance on $H$.
By property \labelcref{d:len}, we have $\sd_{\length{\cdot}}(z,z') \le \sd(z,z')$ for any other global distance $\sd$, and similarly $\sd_{\length{\cdot}}(z,z') \le \sd(z,z')$ if $\pi(z)=\pi(z')$ and $\sd$ is a distance.
Thus, $\sd_{\length{\cdot}}$ is the coarsest possible (global) distance.
If $R$ is a ring, then $\sd_p$ is finer than $\sd_{\dsim}$ (since associated elements are similar), and $\sd_{\dsim}$ is finer than $\sd_{\dsubsim}$ (since similar elements are subsimilar), and all three of these global distances are coarser than $\sd^*$ by the previous lemma.

It follows from \cref{d-basic} that every distance $\sd$ gives rise to a notion of factorizations derived from $\sd$ by identifying rigid factorizations $z$ and $z'$ of an element $a \in H$ if $\sd(z,z')=0$.

\begin{definition}
  Let $\sd$ be a distance on $H$ and let $a \in H$.
  \begin{enumerate}
    \item We define $\sZ_\sd(H) = \sZ^*(H) / \!\sim_\sd$ and $\sZ_\sd(a)$ to be the image of $\sZ^*(a)$ under the canonical homomorphism $\sZ^*(H) \to \sZ_\sd(H)$.
      An element of $\sZ_\sd(a)$ is called a \emph{$\sd$-factorization of $a$} and $\sZ_\sd(H)$ is the \emph{category of $\sd$-factorizations}. We say that $H$ is \emph{$\sd$-factorial} if $\card{\sZ_\sd(a)} = 1$ for all $a \in H$.

    \item
      We set $\sZ_{p}(H) = \sZ_{\sd_p}(H)$ and call these factorizations \emph{permutable factorizations}.
      Given $z \in \sZ^*(H)$ we shall write $\pf{z}$ for its image in $\sZ_p(H)$.
      If $H$ is $\sd_p$-factorial, we say instead that $H$ is \emph{permutably factorial}.
  \end{enumerate}
\end{definition}

Let $z$,~$z' \in \sZ^*(H)$ with $\pi(z)=\pi(z')$.
Since $\sd(z,z')$ depends only on the classes of $z$ and $z'$ in $\sZ_\sd(S)$, we may think of $\sd$ as being defined on $\{\, (z, z') \in \sZ_\sd(S) \times \sZ_\sd(S) : \pi(z)=\pi(z') \,\}$ whenever this is convenient.
Since $\sd^*(z,z') = 0$ if and only if $z = z'$,\, $\sZ_{\sd^*}(H) = \sZ^*(H)$ is just the category of rigid factorizations.
If $H$ is $\sd^*$-factorial, we say instead that it is \emph{rigidly factorial}.
Observe that $H$ is $\sd$-factorial if and only if the homomorphism $\sZ_\sd(H) \to H$ induced by $\pi\colon \sZ^*(H) \to H$ is an isomorphism.

\begin{remark}
  Let $z$,~$z' \in \sZ^*(H)$ with $\pi(z)=\pi(z')$.
  \begin{enumerate}
    \item We have $\sd_{\length{\cdot}}(z,z') = 0$ if and only if $\length{z} = \length{z'}$.
          Thus $H$ is $\sd_{\length{\cdot}}$-factorial if and only if it is half-factorial.

    \item We have $\sd_p(z,z') = 0$ if and only if $\length{z} = \length{z'}$ and there exists a permutation of the factors of $z$ such that they are pairwise associated to those of $z'$.
      If $H$ is a commutative cancellative semigroup, and $a \in H$, then $\sZ_p(a)$ coincides with the usual notion $\sZ(a)$ of factorizations of $a$ (but $\sZ_p(H) \not \cong \sZ(S)$ if $S$ is not reduced).

    \item
      If $\sd$ and $\sd'$ are distances on $\sZ(H)$ with $\sd'$ finer than $\sd$ and $H$ is $\sd'$-factorial, then $H$ is $\sd$-factorial.
      In particular, we have the following:
      Let $R$ be a ring.
      If $R^\bullet$ is permutably factorial, then it is $\sd_\dsim$-factorial.
      If $R^\bullet$ is $\sd_\dsim$-factorial, then it is $\sd_\dsubsim$-factorial.
      Finally, if $R$ is commutative, then all three notions coincide, since then $\sd_p=\sd_\dsim=\sd_\dsubsim$.
  \end{enumerate}
\end{remark}

Since we have identified $\cA(H)$ with $\cA(\sZ^*(H))$, and since representations of rigid factorizations as products of atoms are unique up to a trivial insertion of units, it follows immediately that $\sZ^*(H)$ is rigidly factorial, and thus in particular, $\sZ^*(\sZ^*(H)) \cong \sZ^*(H)$.
Similarly, for any distance $\sd$, the sets $\cA(H)$ and $H^\times$ embed into $\sZ_\sd(H)$, and $\sZ_\sd(H)$ is atomic with $\cA(\sZ_d(H)) = \cA(H)$.

The finer a distance $\sd$, the more refined the notion of factorizations that can be derived from $\sd$.
While $\sd^*$ turns out to be a very useful tool, it may not always be practical to study such a fine notion as rigid factorizations.
For instance, a commutative cancellative semigroup is rigidly factorial if and only if it is factorial and possesses, up to associativity, a unique prime element (that is, it is a discrete valuation monoid).
Thus even commutative PIDs are usually not rigidly factorial.
However, every path category is rigidly factorial.

Nonetheless, rigid factorizations have been studied in the following settings:
Generalizing the study of PIDs, the study of 2-firs (see \cite[Chapter 3]{Cohn85}) and, on an ideal-theoretic level, saturated subcategories of arithmetical groupoids (see \cite{Smertnig}).
The study of polynomial decompositions, that is, the study of the factorization properties of the noncommutative semigroup  $(K[X] \setminus K, \circ)$ where $K$ is (usually) a field, also concerns itself with what amounts to rigid factorizations (see \cite{Zieve-Mueller}).

The following lemma shows that (weak) transfer homomorphisms induce homomorphisms on the categories of rigid factorizations.
We omit the straightforward proof.

\begin{lemma}\label{ext-hom}
  Let $H$ and $T$ be cancellative small categories.
  Let $\phi\colon H \to T$ be a transfer homomorphism, or let $T$ be atomic and $\phi\colon H \to T$ a weak transfer homomorphism.
  There exists a unique homomorphism $\phi^*\colon \sZ^*(H) \to \sZ^*(T)$ satisfying
  \[
    \phi^*(u) = \phi(u) \quad\text{and}\quad \phi^*(\varepsilon) = \phi(\varepsilon) \quad\text{ for all $u \in \cA(H)$ and $\varepsilon \in H^\times$}.
  \]
  Moreover, $\phi^*$ induces the following commutative diagram
  \[
    \xymatrix{
      \sZ^*(H) \ar^{\phi^*}[r] \ar_{\pi_H}[d] & \sZ^*(T) \ar^{\pi_T}[d] \ar^{[.]_p}[r] & \sZ_p(T) \ar[d]\\
      H        \ar^{\phi}[r]                  & T \ar@{=}[r]                           & T.
    }
  \]
  Let $\lbar\phi\colon \sZ^*(H) \to \sZ_p(T)$ denote the homomorphism in the top row.
  \begin{enumerate}
    \item\label{ext-hom:th} If $\phi$ is a transfer homomorphism, then
      \[
        \sZ^*(T) = T^\times \phi^*(\sZ^*(H))T^\times \quad\text{and}\quad \sZ_p(T) = T^\times \lbar\phi(\sZ^*(H)) T^\times.
      \]
      In particular, for all $a \in H$, the induced maps $\sZ^*(a) \to \sZ^*(\phi(a))$ and $\sZ_p(a) \to \sZ_p(\phi(a))$ are surjective.
    \item\label{ext-hom:wth} If $T$ is atomic and $\phi$ is a weak transfer homomorphism, then $\sZ_p(T) = T^\times \lbar\phi(\sZ^*(H))T^\times$.
      In particular, for all $a \in H$, the induced map $\sZ_p(a) \to \sZ_p(\phi(a))$ is surjective.
  \end{enumerate}
  In either case, if $\phi$ is isoatomic, then the homomorphism $\phi_p\colon \sZ_p(H) \to \sZ_p(T)$ induced from $\lbar\phi$ is injective.
\end{lemma}

\section{Catenary Degrees}\label{sec:catenary}

\vspace{1.5mm}
\begin{center}
  \emph{Throughout this section, let $H$ be a cancellative small category.}
\end{center}
\vspace{1.5mm}

Each notion of a distance $\sd$ gives rise to a corresponding catenary degree, as well as a monotone catenary degree.
These invariants provide a measure of how far away $H$ is from being $\sd$-factorial.
For basic properties of the catenary degree in the commutative setting, see \cite[Section 1.6]{GHK06}.

After giving the basic definitions, in \cref{cf} we provide a technical result that allows the study of catenary degrees using transfer homomorphisms.
This will be applied in \cref{sec:maxord} to arithmetical maximal orders in quotient semigroups.
In \cref{strong-weak-distance} we prove a transfer result for distances using an isoatomic weak transfer homomorphism.
This will be applied at the end of \cref{sec:abelianization} to the semigroup $T_n(D)^\bullet$ of non zero-divisors of the ring of $n \times n$ upper triangular matrices over a commutative atomic domain.

\begin{definition}
  Let $H$ be atomic, $\sd$ a distance on $H$, and $a \in H$.
  \begin{enumerate}
    \item Let $z$,~$z' \in \sZ^*(a)$ and $N \in \bN_0$.
      A finite sequence of rigid factorizations $z_0$,~$\ldots\,$,~$z_n \in \sZ^*(a)$, where $n \in \bN_0$, is called an \emph{$N$-chain (in distance $\sd$)} between $z$ and $z'$ if
      \[
        z = z_0,\  z' = z_n,\ \text{ and } \ \sd(z_{i-1}, z_{i}) \le N\ \text{ for all $i \in [1,n]$.}
      \]
      It is called a \emph{monotone $N$-chain} if either $\length{z_0} \le \length{z_1} \le \cdots \le \length{z_n}$ or $\length{z_0} \ge \length{z_1} \ge \cdots \ge \length{z_n}$.

    \item The \emph{[monotone] catenary degree (in distance $\sd$) of $a$}, denoted by $\sc_\sd(a)$ [$\sc_{\sd,\cmon}(a)$], is the minimal $N \in \bN_0 \cup \{\infty\}$ such that for any two factorizations $z$,~$z' \in \sZ^*(a)$ there exists a [monotone] $N$-chain between $z$ and $z'$.

    \item The \emph{catenary degree (in distance $\sd$) of $H$} is $\sc_\sd(H) = \sup\{\, \sc_\sd(a) : a \in H \,\} \in \bN_0 \cup \{\infty\}$, and the \emph{monotone catenary degree (in distance $\sd$)} is $\sc_{\sd,\cmon}(H) = \sup\{\, \sc_{\sd,\cmon}(a) : a \in H \,\} \in \bN_0 \cup \{\infty\}$.
  \end{enumerate}
\end{definition}

As in the commutative setting, the monotone catenary degree is usually studied using two auxiliary invariants, the equal catenary degree and the adjacent catenary degree. The \emph{equal catenary degree}, $\sc_{\sd,\ceq}(a)$, is the smallest $N \in \bN_0 \cup \{\infty\}$ such that for any two factorizations $z$,~$z' \in \sZ^*(a)$ with $\length{z}=\length{z'}$, there exists a monotone $N$-chain between $z$ and $z'$ (since $\length{z}=\length{z'}$, this means one in which every factorization is of length $\length{z}$).
We set
\[
  \sc_{\sd,\ceq}(H) = \sup\{\,\sc_{\sd,\ceq}(a) : a \in H \,\} \in \bN_0 \cup \{\infty\}.
\]
For $a \in H$ and $k$,~$l \in \sL(a)$ write,
\[
  d_{k,l}(a) = \min\{\, \sd(z,z') : z,\, z' \in \sZ^*(a),\, \length{z}=k,\,\length{z'}=l \,\}.
\]
We say that $k$ and $l$ are adjacent in $\sL(a)$ if $\sL(a) \cap [k,l] = \{k,l\}$.
The \emph{adjacent catenary degree} of $a \in H$ is defined as
\[
  \sc_{\sd,\cadj}(a) = \sup\{\, d_{k,l}(a) : \text{$k$, $l$ are adjacent in $\sL(a)$} \,\} \in \bN_0 \cup \{\infty\},
\]
with $\sc_{\sd,\cadj}(H) = \sup\{\, \sc_{\sd,\cadj}(a) : a \in H \,\} \in \bN_0 \cup \{\infty\}$.
It is immediate from the definitions that
\[
  \sc_{\sd}(a) \le \sc_{\sd,\cmon}(a) =  \sup\{ \sc_{\sd,\ceq}(a), \sc_{\sd,\cadj}(a) \},
\]
and hence $$\sc_\sd(H) \le \sc_{\sd,\cmon}(H) =  \sup\{ \sc_{\sd,\ceq}(H), \sc_{\sd,\cadj}(H) \}.$$

We denote the catenary degrees associated to $\sd^*$, $\sd_p$, $\sd_{\dsim}$, and $\sd_{\dsubsim}$ by $\sc^* = \sc_{\sd^*}$, $\sc_p = \sc_{\sd_p}$, $\sc_\dsim = \sc_{\sd_\dsim}$, and $\sc_\dsubsim = \sc_{\sd_\dsubsim}$, and use analogous conventions for the monotone, equal and adjacent catenary degrees.

The following lemma parallels \cite[Lemma 1.6.2]{GHK06}.
\cref{remark} shows that in \labelcref{dl:difflen,dl:delta} this is the best we can do in a general noncommutative setting, despite the fact that stronger bounds are available for the usual distance in the commutative setting.

\begin{lemma}\label{dl}
  Let $H$ be atomic and let $\sd$ be a distance on $H$. Let $a \in H$.
  \begin{enumerate}
    \item\label{dl:basic} We have $\sc_\sd(a) \le \sc_{\sd,\cmon}(a) \leq \sup \sL(a)$, and $\sc_\sd(a)=0$ if and only if $\sc_{\sd,\cmon}(a)=0$ if and only if $|\sZ_\sd(a)|=1$.
          In particular, $H$ is $\sd$-factorial if and only if $\sc_\sd(H) =\sc_{\sd,\cmon}(H)= 0$.
    \item\label{dl:difflen} If $z$,~$z' \in \sZ^*(a)$, then $\babs{\length z - \length{z'}} \leq \sd(z,z')$.
    \item\label{dl:delta} If $\Delta(\sL(a)) \ne \emptyset$, then $\sup \Delta(\sL(a))\leq \sc_\sd(s)$. In particular, $\sup \Delta(H) \le \sc_\sd(H)$.
    \item\label{dl:hf}
      If $\sc_\sd(H)=0$, then $H$ is half-factorial.
      If $c_\sd(a)\leq 1$, then $\sL(a)$ is an interval.
  \end{enumerate}
\end{lemma}

\begin{proof}
  \ref*{dl:basic}
  If $a \in H^\times$, then $\card{\sZ_d(a)} = 1$, and $\sc_\sd(a)=\sc_{\sd,\cmon}(a) = \sup \sL(a) = 0$.
  Suppose that $a \in H$ is a non-unit and let $z$,~$z' \in \sZ^*(a)$.
  Then $\sd(z,z') \le \max\{ \length{z}, \length{z'},1 \} \le \sup \sL(a)$.
  Thus $\sc_\sd(a) \le \sc_{\sd,\cmon}(a) \leq \sup \sL(a)$.
  If $\card{\sZ_\sd(a)} = 1$, then clearly $\sc_\sd(a) = \sc_{\sd,\cmon}(a) = 0$.
  Conversely, suppose that $\sc_\sd(a) = 0$.
  Then there exist $n \in \bN_0$ and $z_0$,~$\ldots\,$,~$z_n \in \sZ^*(a)$ such that $z=z_0$, $z'=z_n$ and $\sd(z_{i-1},z_i) = 0$ for all $i \in [1,n]$.
  Therefore $\sd(z,z') = 0$ by the triangle inequality.
  Since $z$ and $z'$ are both rigid factorizations of $a$, we have $z \sim_\sd z'$.
  Thus $\card{\sZ_\sd(a)} = 1$ and $\sc_{\sd,\cmon}(a)=0$.

  \ref*{dl:difflen} This is simply property \labelcref{d:len}.

  \ref*{dl:delta} Let $d \in \Delta(\sL(a))$.
  Then there exist $z$,~$z' \in \sZ^*(a)$ such that $d = \length{z'} - \length{z}$ and there exists no $z'' \in \sZ^*(a)$ with $\length{z} < \length{z''} < \length{z'}$.
  By definition of the catenary degree, there exists a $\sc_\sd(a)$-chain in distance $\sd$ between $z$ and $z'$.
  By \ref*{dl:difflen}, this implies $d \le \sc_\sd(a)$.

  \ref*{dl:hf} This is clear by \ref*{dl:delta}.
\end{proof}

\begin{remark}\label{remark}
  The bounds in the previous lemma are weaker than their commutative counterparts.
  In particular, for a commutative cancellative semigroup $S$, it is true that $\sup \Delta(S) + 2 \le \sc_p(S)$ and hence even $\sc_p(S) \le 2$ implies that $S$ is half-factorial.
  We now point out that \cref{dl} is the best possible for a general result in the noncommutative setting.
  Let $T=\langle a, b, c\mid abc=cb\rangle$.
  Clearly $T$ is reduced with $\cA(T) = \{ a,b,c\}$, and it is easily verified that $T$ is an Adyan semigroup and hence cancellative.
  Let $S$ be a commutative atomic cancellative semigroup.

  \begin{enumerate}
    \item
      For all $s \in S$ we have either $\sc_p(s)=0$ or $\sc_p(s) \geq 2$.
      This fails for the semigroup $T$, as $\sc_p(abc=cb)=1$.

    \item If $s \in S$ and $z$,~$z'$ are two distinct factorizations in $\sZ_p(s)$, then $\sd_p(z,z') \geq \babs{\length z - \length{z'}}+2$.
      This fails for $T$ since $\babs{\length{abc}-\length{cb}}=1=\sd_p(\rf{a,b,c},\rf{c,b})$.
      Moreover, in the cancellative semigroup $\langle a,b \mid aba = b \rangle$, we have $\sd^*(\rf{a,b,a}, \rf{b}) = 2 = \length{\rf{a,b,a}} - \length{\rf{b}}$,
      showing that the inequality in \subref{dl:difflen} is also best possible for the rigid distance.

    \item If $s \in S$ and $\card{\sZ_p(s)}\geq 2$, then $\sc_p(s) \geq \sup \Delta(s) +2$.
      This fails for $T$ since $\sL(abc=cb)=\{2,3\}$ and thus $\Delta(abc)=\{1\}$ as well as $\sc_p(abc) = 1$.
  \end{enumerate}
\end{remark}

\begin{example}
  We illustrate the terminology of this section by means of a classical example.
  Consider $S = (\bC[X]\setminus \bC,\circ)$, that is, decompositions of non-constant polynomials with coefficients in $\bC$.
  An atom of $S$ is called an indecomposable polynomial, and a rigid factorization of $f \in S$ is called a complete decomposition of $f$.
  Ritt's first theorem (\cite[Theorem 2.1]{Zieve-Mueller}) says that any complete decomposition of $f \in S$ can be transformed into any other, by a sequence of transformations in each of which two adjacent indecomposable factors are replaced by two new ones.
  In our present terminology, this can be expressed simply as $\sc^*(S) \le 2$.
  (Note however that much more refined results on polynomial decompositions are known.)
\end{example}

Let $H$ and $T$ be atomic cancellative small categories, and let $\phi\colon H \to T$ be a weak transfer homomorphism.
Let $\phi^*\colon \sZ^*(H) \to \sZ^*(T)$ denote the extension of $\phi$ to the categories of rigid factorizations as given in \cref{ext-hom}.
If $z$,~$z' \in \sZ^*(H)$, then
\[
  \sd^*(z,z') \,\ge\, \sd^*(\phi^*(z), \phi^*(z')) \,\ge\, \sd_p(\phi^*(z), \phi^*(z')).
\]
If $\phi$ is a transfer homomorphism, this together with \subref{ext-hom:th} implies $\sc^*(a) \ge \sc^*(\phi(a)) \ge \sc_p(\phi(a))$ for all $a \in H$, and $\sc^*(H) \ge \sc^*(T) \ge \sc_p(T)$.
Moreover,
\[
  \sd_p(z,z') \ge \sd_p(\phi^*(z), \phi^*(z')).
\]
Thus \subref{ext-hom:wth} implies that $\sc^*(a) \ge \sc_p(a) \ge \sc_p(\phi(a))$ for all $a \in H$, and $\sc^*(H) \ge \sc_p(H) \ge \sc_p(T)$.
Analogous inequalities hold for the monotone and equal catenary degrees.

In the commutative setting, catenary degrees can be studied using transfer homomorphisms.
If $\phi\colon S \to T$ is a transfer homomorphism of commutative atomic cancellative semigroups, one finds that $\sc_p(S) \le \max\{ \sc_p(T), \sc_p(S,\phi) \}$, where $\sc_p(S,\phi)$ is a suitably defined catenary degree in the fibers of $\phi$ (see \cite[Lemma 3.2.6]{GHK06}).
The strength of this method lies in the fact that, for a commutative Krull monoid $S$ and its usual transfer homomorphism $\phi$ to a monoid of zero-sum sequences, it always holds that $\sc_p(T,\phi) \le 2$ (by \cite[Theorem 3.2.8]{GHK06}).
Thus the catenary degree in the image $T$ controls the catenary degree in $S$ to a very large degree: Unless $S$ is half-factorial, it holds that $\sc_p(S) = \sc_p(T)$.

Aiming for similar results, we now introduce the notion of a catenary degree in the permutable fibers.

\begin{definition}[Catenary degree in the permutable fibers]
  Let $H$ and $T$ be atomic cancellative small categories, and let $\sd$ be a distance on $H$.
  Suppose that there exists a transfer homomorphism $\phi\colon H \to T$.
  Denote by $\phi^*\colon \sZ^*(H) \to \sZ^*(T)$ its extension to the categories of rigid factorizations as given in \cref{ext-hom}, and denote by $\lbar\phi\colon \sZ^*(H) \to \sZ_p(T)$ the composition of $\phi^*$ with the canonical homomorphism $\sZ^*(T) \to \sZ_p(T)$.

  Let $a \in H$, and let $z$,~$z' \in \sZ^*(a)$ with $\lbar \phi(z) = \lbar \phi(z')$ (that is, $\sd_p(\phi^*(z), \phi^*(z')) = 0$).
  We say that an $N$-chain $z=z_0$,~$z_1$, $\ldots\,$,~$z_n=z' \in \sZ^*(a)$ \emph{lies in the permutable fiber of $z$} if $\lbar\phi(z_i) = \lbar\phi(z)$ for all $i \in [0,n]$.

  We define $\sc_\sd(a,\phi)$ to be the smallest $N \in \bN_0 \cup \{\infty\}$ such that, for any two $z$, $z' \in \sZ^*(a)$ with $\lbar \phi(z) = \lbar \phi(z')$, there exists an $N$-chain (in distance $\sd$) between $z$ and $z'$, lying in the permutable fiber of $z$.
  Moreover, we define
  \[
    \sc_\sd(H, \phi) = \sup\{\, \sc_\sd(a, \phi) : a \in H \,\} \in \bN_0 \cup \{\infty \}.
  \]
\end{definition}

The first claim of the following proposition provides a weaker analogue of \cite[Proposition 3.2.3.3(c)]{GHK06} for the present setting, while the second statement roughly corresponds to \cite[Lemma 3.2.6.2]{GHK06}.
Note that to prove the first claim we need to make use of the catenary degree in the permutable fibers, quite unlike the commutative variant.
The restriction to $T$ being reduced is made to simplify the proof, but is not essential.
We note that the commutative hypothesis on $T$ is essential in the proof.

\begin{proposition}\label{cf}
  Let $T$ be a commutative reduced cancellative semigroup.
  Let $H$ be atomic, let $\sd$ be a distance on $H$, and suppose $\phi\colon H \to T$ is a transfer homomorphism.
  Denote by $\phi^*\colon \sZ^*(H) \to \sZ^*(T)$ the extension of $\phi$ to the categories of rigid factorizations (as in \cref{ext-hom}).
  Let $a \in H$.
  \begin{enumerate}
    \item\label{cf:lift} Let $z \in \sZ_{H}^*(a)$.
      If $\lbar y \in \sZ^*_T(\phi(a))$, then there exist $y$, $y'$, $z' \in \sZ_H^*(a)$ such that
      \[
        \phi^*(y) = \lbar y,\quad \phi^*(z') \sim_{\sd_p} \phi^*(z),\quad \phi^*(y') \sim_{\sd_p} \lbar y,\quad \sd(z',y') \le \sd_p(\phi^*(z), \lbar y),
      \]
      and there exist $\sc_\sd(a,\phi)$-chains between $z$ and $z'$ lying in the permutable fiber of $z$, and between $y$ and $y'$ lying in the permutable fiber of $y$.

    \item\label{cf:chain}
      Let $z$, $z' \in \sZ_H^*(a)$, $k \in \bN_0$, $\lbar z_1$, $\ldots\,$,~$\lbar z_k \in \sZ_T^*(\phi(a))$ and set $\lbar z_0 = \phi^*(z)$ and $\lbar z_{k+1} = \phi^*(z')$.
      Then there exist rigid factorizations $z=z_0$,~$z_1$, $\ldots\,$,~$z_k$,~$z_{k+1}=z' \in \sZ_H^*(a)$ such that $z_i$ and $z_{i+1}$ are connected by a monotone $\max\{\sc_{\sd}(a,\phi), \sd_{p}(\lbar z_i, \lbar z_{i+1} ) \}$-chain in distance $\sd$ for all $i \in [0,k]$, and $\phi^*(z_i) = \lbar z_i$ for all $i \in [1,k]$.
      In particular,
      \begin{align*}
        \sc_\sd(a) &\le \max\{ \sc_p(\phi(a)),\, \sc_\sd(a, \phi) \},  & \sc_\sd(H) & \le \max\{ \sc_p(T), \sc_\sd(H, \phi) \}, \\
        \sc_{\sd,\cmon}(a) &\le \max\{ \sc_{p,\cmon}(\phi(a)),\, \sc_{\sd}(a, \phi) \},  & \sc_{\sd,\cmon}(H) & \le \max\{ \sc_{p,\cmon}(T), \sc_\sd(H,\phi) \}, \\
        \sc_{\sd,\ceq}(a) &\le \max\{ \sc_{p,\ceq}(\phi(a)),\, \sc_{\sd}(a, \phi) \},  & \sc_{\sd,\ceq}(H) & \le \max\{ \sc_{p,\ceq}(T), \sc_\sd(H,\phi) \}.
      \end{align*}
  \end{enumerate}
\end{proposition}

\begin{proof}
  \ref*{cf:lift}
      Since $T$ is commutative, $\sd_p$ is the usual distance on $T$. Thus there exist $\lbar x$, $\lbar y_0$, $\lbar z_0 \in \sZ^*(T)$ such that $\lbar y \sim_{\sd_p} \rf{\lbar x,\lbar y_0}$,\, $\phi^*(z) \sim_{\sd_p} \rf{\lbar x, \lbar z_0}$ and $\max\{ \length{\lbar y_0}, \length{\lbar z_0} \} = \sd_p(\phi^*(z),\lbar y)$.
      In particular, $\rf{\lbar x,\lbar y_0}$ and $\rf{\lbar x, \lbar z_0}$ are indeed rigid factorization of $\phi(a)$.
      Since $\phi$ is a transfer homomorphism, this implies that there exists a rigid factorization $z' \in \sZ_H^*(a)$ such that $\phi^*(z') = \rf{\lbar x, \lbar z_0}$.
      Then $\phi^*(z') \sim_{\sd_p} \phi^*(z)$ and, by definition of $\sc_\sd(a, \phi)$, there exists a $\sc_\sd(a,\phi)$-chain between $z$ and $z'$ lying in the permutable fiber of $z$.
      Let $y \in \sZ^*(a)$ be an arbitrary preimage of $\lbar y$ under $\phi^*$.

      Now let $z' = \rf[\varepsilon]{u_1,\cdots,u_l}$ with $l \in \bN_0$, $\varepsilon \in H^\times$, $u_1$,~$\ldots\,$,~$u_l \in \cA(H)$, and let $k = \length{x}$.
      Then $\phi^*(\rf[\varepsilon]{u_1,\cdots,u_k}) = \lbar x$.
      Suppose $\lbar y_0 = \rf{\lbar v_1, \cdots, \lbar v_n}$ with $n \in \bN_0$ and $\lbar v_1$,~$\ldots\,$,~$\lbar v_n \in \cA(T)$.
      Then $\phi(a) = \phi(\varepsilon^{-1}a) = \phi(u_1)\cdots\phi(u_l) = \phi(u_1) \cdots \phi(u_k)  \lbar v_1 \cdots \lbar v_n$, and thus $\phi(u_{k+1} \cdots u_l) = \phi(u_{k+1})\cdots \phi(u_l) = \lbar v_1 \cdots \lbar v_n$.
      If $k=l$, then $n=0$ and $\lbar y \sim_{\sd_p} \lbar x \sim_{\sd_p} \phi^*(z)$.
      Setting $y' = z'$, we are done.
      Now suppose that $k < l$.
      Then $n > 0$ and, since $\phi$ is a transfer homomorphism, there exist $v_1$, $\ldots\,$,~$v_n \in \cA(H)$ such that $u_{k+1}\cdots u_l = v_1\cdots v_n$ and $\phi(v_i) = \lbar v_i$ for all $i \in [1,n]$.
      We define $y' = \rf[\varepsilon]{u_1, \cdots, u_k, v_1, \cdots, v_n}$ (note that $s(v_1) = s(u_{k+1}) = t(u_k)$).
      Using property \labelcref{d:mul}, it follows that $\sd(y',z') = \sd(\rf{v_1,\cdots,v_n}, \rf{u_{k+1},\cdots,u_l})$.
      Thus property \labelcref{d:len} implies $\sd(y',z') \le \max\{ n, l-k, 1 \} = \sd_p(\phi^*(z),\lbar y)$.
      Now $y$ and $y'$ lie in the same permutable fiber, and hence are connected by a $\sc_\sd(a,\phi)$-chain in the permutable fiber of $y$.

  \ref*{cf:chain}
      Set $z_0 = z$, and $z_{k+1} = z'$.
      We apply \ref*{cf:lift} inductively to construct $z_1, \ldots, z_k \in \sZ^*_H(a)$ with the desired properties.
      Suppose that we have constructed $z_i$ for some $i \in [0,k-1]$.
      Applying \ref*{cf:lift} to $z_i$ and $\lbar z_{i+1}$, we find $z_i'$,~$z_{i+1}'$, $z_{i+1} \in \sZ_H^*(a)$ such that $\phi^*(z_i) \sim_{\sd_p} \phi^*(z_i')$, $\phi^*(z_{i+1}') \sim_{\sd_p} \phi^*(z_{i+1})$, $\phi^*(z_{i+1}) = \lbar z_{i+1}$ and $\sd(z_{i}', z_{i+1}') \le \sd_p(\lbar z_i, \lbar z_{i+1})$.
      Since $z_i$ and $z_i'$ lie in the same permutable fiber, there exists a $\sc_\sd(a,\phi)$-chain between $z_i$ and $z_i'$ lying in the permutable fiber of $z_i$.
      In particular, all the rigid factorizations in this chain have length $\length{z_i}$.
      Similarly, between $z_{i+1}'$ and $z_{i+1}$ there exists a $\sc_\sd(a,\phi)$-chain in the permutable fiber of $z_{i+1}$.
      Since $\sd(z_{i}', z_{i+1}') \le \sd_p(\lbar z_i, \lbar z_{i+1})$ we can therefore construct a $\max\{ \sc_\sd(a,\phi), \sd_p(\lbar z_{i},\lbar z_{i+1}) \}$-chain between $z_i$ and $z_{i+1}$.
      This chain is obviously monotone, since the only point at which the length can change is from $z_{i}'$ to $z_{i+1}'$.

      The upper bound on $\sc_\sd(a)$ now follows immediately by lifting chains from the image.
      Similarly, the upper bounds for the monotone and equal catenary degree follow by lifting monotone chains, respectively chains of equal length.
\end{proof}

\begin{remark}
  Let $H$ be atomic.
  We have $\sc_p(a) \le \sc^*(a)$ for all $a \in H$.
  Suppose $H$ is a commutative semigroup.
  Then the identity map $\id\colon H \to H$ is a transfer homomorphism.
  Let $a \in H$.
  Two rigid factorizations $z=\rf[\varepsilon]{u_1,\cdots,u_k}$, $z'=\rf[\eta]{v_1,\cdots,v_l} \in \sZ^*(a)$ lie in the same permutable fiber of the identity map if and only if $[z]_p = [z']_p$, that is, $k=l$ and there exists a permutation $\sigma \in \fS_k$ such that $u_i \simeq v_{\sigma(i)}$ for all $i \in [1,k]$.
  By writing $\sigma$ as a product of transpositions, it is therefore easy to construct a $2$-chain in distance $\sd^*$ between $z$ and $z'$ (here we use the commutativity of $H$ to ensure that permuting two atoms does not change the product).
  Therefore $\sc^*(a,\id) \le 2$.
  Applying the previous proposition, in the commutative case we therefore have
  \[
    \sc_p(a) \le \sc^*(a) \le \max\{ 2, \sc_p(a) \}.
  \]
  Trivially, this remains true if we replace the rigid distance by any distance finer than the permutable distance, but coarser than the rigid distance: While two such distances can be quite different, their catenary degrees cannot differ by much.

  However, this does not hold in general.
  Let $n \in \bN$ and let
  \[
    S = \langle a,b \;|\; a^n b^n = b^n a^n \rangle.
  \]
  Then $S$ is cancellative and reduced with $\cA(S) = \{a,b\}$, and it follows immediately that $\sc_p(S) = 0$ while $\sc^*(S) = 2n$.
\end{remark}

The following proposition shows that distances and catenary degrees are preserved by isoatomic weak transfer homomorphisms.

\begin{proposition}\label{strong-weak-distance}
  Let $H$ and $T$ be atomic cancellative small categories, and assume that there exists an isoatomic weak transfer homomorphism $\phi\colon H \rightarrow T$.
  Denote by $\phi_p\colon \sZ_p(H) \rightarrow \sZ_p(T)$ the extension of $\phi$ to permutable factorizations (as in \cref{ext-hom}).
  Then $\sd_p(z,z')=\sd_p(\phi_p(z),\phi_p(z'))$ for any two permutable factorizations $z$ and $z'$ of $a \in H$.
  In particular, $\sc_p(H) = \sc_p(T)$.
\end{proposition}

\begin{proof}
If $a \in H^\times$, the claim is trivially true.
Assume from now on that $a$ is not a unit.
Without loss of generality, write $k=|z|\leq |z'|=l$ with $k \leq l$.
Then we can write $z=\pf{u_1,\cdots,u_k}$ and $z'=\pf{v_1,\cdots,v_l}$ with $u_i$,~$v_j \in \cA(H)$ for $i \in [1,k]$ and $j \in [1,l]$.
Moreover, there exists an $m \in [1,k]$ and permutations $\sigma \in \fS_k$, $\tau \in \fS_l$ such that $u_{\sigma(i)} \simeq v_{\tau(i)}$ for all $i \in [1,m]$ and such that $u_{\sigma(i)} \not\simeq v_{\tau(j)}$ whenever $i \in [m+1,k]$ and $j \in [m+1,l]$.
Note that $\sd_p(z,z')=l-m$.

We have $\phi_p(z)=\pf{\phi(u_1),\cdots,\phi(u_k)}$ and $\phi_p(z')=\pf{\phi(v_1),\cdots,\phi(v_l)}$ with $\phi(u_i)$ and $\phi(v_j) \in \cA(T)$ for $i \in [1,k]$ and $j \in [1,l]$.
There exists $n \in [1,k]$ and permutations $\widehat\sigma \in \fS_k$, $\widehat\tau \in \fS_l$ such that $\phi(u_{\widehat\sigma(i)}) \simeq \phi(v_{\widehat\tau(i)})$ for all $i \in [1,n]$ and such that $\phi(u_{\widehat\sigma(i)}) \not\simeq \phi(v_{\widehat\tau(j)})$ whenever $i \in [n+1,k]$ and $j \in [n+1,l]$.
Note that $\sd_p(\phi_p(z), \phi_p(z')) =l-n$.

Since $\phi$ is isoatomic, $n=m$ and therefore $\sd_p(z,z')=\sd_p(\phi_p(z),\phi_p(z'))$.
\end{proof}

\section{Divisibility}\label{sec:divisibility}

\vspace{1.5mm}
\begin{center}
  \emph{Throughout this section, let $H$ be a cancellative small category.}
\end{center}
\vspace{1.5mm}

Our goal is to generalize the notion of divisibility of one element by another from the commutative setting, to use this notion to better understand the factorizations introduced in \cref{sec:factorizations}, and to give another measure of the non-uniqueness of permutable factorizations by generalizing the tame-degree and $\omega$-invariant from the commutative setting.
To this end, we begin by defining an abstract divisibility relation.

\begin{definition}\label{div}
A relation $\wr$ on $H$ is a \emph{divisibility relation} provided that the following conditions are satisfied.
\begin{enumerate}
  \item\label{div:prod} If $a\wr b$ or $a \wr c$ for any elements $a$,~$b$,~$c \in H$ with $t(b)=s(c)$, then $a \wr bc$.
  \item\label{div:unit} For all $a \in H$ and $\varepsilon$,~$\eta \in H^{\times}$ with $t(\varepsilon)=s(a)$ and $s(\eta)=t(a)$, we have $a \wr \varepsilon a \eta$.
  \item\label{div:assoc} For all $a \in H\setminus H^\times$ and $u \in \cA(H)$, if $a \wr u$, then $a \simeq u$.
  \item\label{div:nontrivial} If $a \wr \varepsilon$ for some $\varepsilon \in H^\times$, then $a \in H^\times$.
\end{enumerate}
\end{definition}

If $H$ is a commutative semigroup, the usual notion of divisibility, $a\mid b$ if $b \in aH$, is a divisibility relation.
In the noncommutative setting, our focus will be on one of the following two relations, each of which clearly satisfies the formal properties of a divisibility relation.
As we are mostly interested in when an atom divides a product, the particular choice of divisibility relation will not make a difference as long as $H$ is atomic (see \subref{cdivind:atom}).

\begin{definition}\label{div2}
  Let $a$ and $b$ be two elements of $H$.
  \begin{enumerate}
    \item\label{div2:leftright} We say that $a$ \emph{left-right divides} $b$, and write $a \mid_{l-r} b$, provided  $b \in HaH$.
    \item\label{div2:perm} We say that \emph{$a$ divides $b$ up to permutation}, and write $a \mid_{p} b$, if there are permutable factorizations $\pf[\varepsilon]{u_1,\cdots,u_k}$ of $a$ and $\pf[\eta]{v_1,\cdots,v_l}$ of $b$ (with $k$,~$l \in \bN_0$, $\varepsilon$,~$\eta \in H^\times$ and $u_1$,~$\ldots\,$,~$u_k$, $v_1$,~$\ldots\,$,~$v_l \in \cA(H)$) such that
  \begin{enumerate}
    \item $k \leq l$, and
    \item there exists an injective map $\sigma\colon [1,k] \to [1,l]$ with $u_{i} \simeq v_{\sigma(i)}$ for all $i \in [1,k]$.
  \end{enumerate}
  \end{enumerate}
\end{definition}

If $H$ is a commutative semigroup, then $a$ left-right divides $b$ if and only if $a$ divides $b$, since $HaH=aH$.
If, moreover, $H$ is atomic, then $a$ divides $b$ up to permutation if and only if $a$ divides $b$.

If $H$ is atomic, we can characterize left-right divisibility in terms of rigid factorizations as follows:
$a \mid_{l-r} b$ if and only if for any (equivalently all) rigid factorization $z \in \sZ^*(a)$ there exist $x$,~$y \in \sZ^*(H)$ such that $x\rfop z\rfop y$ is a rigid factorization of $b$.
Indeed, if $z \in \sZ^*(a)$ and $x$,~$y \in \sZ^*(H)$ are as described, then $b = \pi(x)\pi(z)\pi(y) = \pi(x)a\pi(z)$.
Conversely, if $b=cad$ with $c$,~$d \in H$, let $z \in \sZ^*(a)$, $x \in \sZ^*(c)$ and $y \in \sZ^*(d)$.
Then $x \rfop z \rfop y \in \sZ^*(b)$.

If $H$ is atomic and $a \mid_{l-r} b$, then $a \mid_p b$.
The converse is clearly not true.
However, if $u \in \cA(H)$, then $u \mid_p b$ if and only if $u \mid_{l-r} b$, as we shall see in \subref{cdivind:atom}.
The notions of left, respectively right, divisibility, do not generally give a divisibility relation due to the failure of property \labelcref{div:prod} to hold.

We now study a general divisibility relation $\wr$ on $H$.
However, throughout we have in mind the two specific divisibility relations given in \cref{div2}.
In fact, we will return at the end of this section to a more thorough investigation of $\mid_p$.

\begin{definition}
  Let $\wr$ be a divisibility relation on $H$.
  A non-unit $q \in H$ is an \emph{almost prime-like element (with respect to $\wr$)} if, whenever $q \wr ab$ for some $a$,~$b \in H$ with $t(a)=s(b)$, either $q \wr a$ or $q \wr b$.
\end{definition}

\begin{remark}\label{almostprime-like}
\envnewline
\begin{enumerate}
\item\label{almostprime-like:commutative} It is clear from the definition that the notion of an almost prime-like element in $H$ corresponds to the usual notion of a prime element if $H$ is a commutative semigroup and if either $\wr$ is $\mid_{l-r}$, or $H$ is atomic and $\wr$ is $\mid_p$.

\item\label{almostprime-like:ideals}
  We will compare the notion of almost prime-like elements to more established concepts below in \cref{prime-like-ideals}.

\item\label{almostprime-like:independent} We shall see in \cref{cdivind} that if $H$ is atomic, the notion of almost prime-like elements is in fact independent of the particular divisibility relation chosen.

\item\label{almostprime-like:prime-like} After seeing how almost prime-like elements behave like prime elements in the commutative setting in \cref{primesalvage} and \cref{primesalvage2} and how products of almost prime-like elements do not necessarily give elements with unique permutable factorizations in \cref{weirdprimes}, we strengthen the definition in \ref{prime-like} to prime-like elements.
\end{enumerate}
\end{remark}

We now show that if $H$ is atomic, then, as in the commutative setting, almost prime-like elements are necessarily atoms of $H$.

\begin{lemma}\label{consdivprod}
  Let $\wr$ be a divisibility relation on $H$, and let $q$ be an almost prime-like element of $H$.
  \begin{enumerate}
    \item\label{consdivprod:fin} If $q\wr a_1\cdots a_m$ for some $m \in \bN$ and elements $a_1$,~$\ldots\,$,~$a_m \in H$, then $q\wr a_i$ for some $i \in [1,m]$.
    \item\label{consdivprod:assoc} If $q \wr u_1\cdots u_m$ for some $m \in \bN$ and atoms $u_1$,~$\ldots\,$,~$u_m \in \cA(H)$, then $q \simeq u_i$ for some $i \in [1,m]$.
    \item\label{consdivprod:atom} If $H$ is atomic, then $q$ is an atom.
  \end{enumerate}
\end{lemma}

\begin{proof}
  \ref*{consdivprod:fin}
  We proceed by induction on $m$.
  By definition, if $m \in \{1,2\}$, then $q\wr a_1$ or $q\wr a_2$.
  Now suppose that $m>2$ and that if $q \wr a_1\cdots a_{m-1}$, then $q\wr a_i$ for some $i\in [1,m-1]$.
  Since $q\wr (a_1\cdots a_{m-1})a_m$ and $q$ is almost prime-like, either $q\wr a_1\cdots a_{m-1}$ or $q\wr a_m$.
  If $q\wr a_m$, then we are done.
  Otherwise, the induction hypothesis implies $q\wr a_i$ for some $i\in [1,m-1]$.

  \ref*{consdivprod:assoc}
  Applying \ref*{consdivprod:fin}, we find $q \wr u_i$ for some $i \in [1,m]$, and then property \labelcref{div:assoc} of the divisibility relation implies $q \simeq u_i$.

  \ref*{consdivprod:atom}
  Since $H$ is atomic and $q$ is a non-unit, there exist $m \in \bN$ and atoms $u_1$,~$\ldots\,$,~$u_m \in \cA(H)$ such that $q = u_1\cdots u_m$.
  By property \labelcref{div:unit} of the divisibility relation, we have $q \wr u_1\cdots u_m$, and then \ref*{consdivprod:assoc} implies $q \simeq u_i$ for some $i \in [1,m]$.
  Hence $q$ is an atom.
\end{proof}

The following lemma shows that if $H$ is atomic, the choice of divisibility relation has no bearing on which elements are almost prime-like.

\begin{lemma}\label{cdivind}
  Let $H$ be atomic, and let $\wr$ and $\wr'$ be divisibility relations on $H$.
  \begin{enumerate}
    \item\label{cdivind:atom} If $u \in \cA(H)$ and $a \in H$, then $u \wr a$ if and only if $u \wr' a$.
    \item\label{cdivind:consistent} An element $q \in H$ is almost prime-like with respect to $\wr$ if and only if it is almost prime-like with respect to $\wr'$.
  \end{enumerate}
\end{lemma}

\begin{proof}
  \ref*{cdivind:atom}
  Suppose $u \wr a$.
  By property \labelcref{div:nontrivial} of a divisibility relation, $a$ is not a unit, and hence there exist $k \in \bN$ and atoms $u_1$,~$\ldots\,$,~$u_k$ of $H$ such that $a=u_1\cdots u_k$.
  Thus \subref{consdivprod:assoc} implies $u \simeq u_i$ for some $i \in [1,k]$.
  Then property \labelcref{div:unit} of $\wr'$ implies $u \wr' u_i$, and by property \labelcref{div:prod} we have $u \wr' a$.
  The converse follows by symmetry.

  \ref*{cdivind:consistent}
  Suppose $q$ is almost prime-like with respect to $\wr$, and note that \subref{consdivprod:atom} implies that $q$ is an atom.
  Let $a$,~$b \in H$ such that $q \wr' ab$.
  Then $q \wr ab$ by \ref*{cdivind:atom}, and hence either $q \wr a$ or $q \wr b$.
  Using \ref*{cdivind:atom} again, $q \wr' a$ or $q \wr' b$.
  The converse follows by symmetry.
\end{proof}

Products of prime elements in a commutative cancellative semigroup have unique factorization (cf. \cite[Proposition 1.1.8]{GHK06}).
In fact, if $p_1\cdots p_kc=q_1\cdots q_ld$ with each $p_i$ and $q_j$ prime, with no $p_i$ dividing $d$, and no $q_j$ dividing $c$, then $k=l$ and, up to permutation, $p_i\simeq q_i$ for each $i$. Moreover, $c\simeq d$.
We now provide, for almost prime-like elements in a cancellative small category, a weaker statement than its commutative counterpart.
That this result is necessarily weaker is exhibited in \cref{weirdprimes}.

We first make some notational remarks.
The category of rigid factorizations, $\sZ^*(H)$, is itself an atomic cancellative small category, and hence $\mid_p$ and $\mid_{l-r}$ are defined on $\sZ^*(H)$.
As in \cref{sec:factorizations}, we may identify $\cA(H)$ with $\cA(\sZ^*(H))$ and $H^\times$ with $\sZ^*(H)^\times$.
If $z \in \sZ^*(H)$, we say that $u \in \cA(H)$ \emph{occurs in $z$} if $u \mid_p z$ (equivalently $u \mid_{l-r} z$) in $\sZ^*(H)$.
The same remarks apply to $\sZ_p(H)$ instead of $\sZ^*(H)$ and clearly, if $z \in \sZ^*(H)$, $u$ occurs in $z$ if and only if $u$ occurs in $\pf{z} \in \sZ_p(H)$.

We now show that if an almost prime-like element $q$ occurs in some rigid factorization of an element $a$, then $q$ occurs in every such factorization of $a$.

\begin{proposition}\label{primesalvage}
Let $\wr$ be a divisibility relation on $H$ and let $a \in H$.
Suppose that $z=\rf[\varepsilon]{u_1,\cdots,u_k}$ and $z'=\rf[\eta]{v_1,\cdots,v_l}$ are two rigid factorizations of $a$, where
\[
  \{u_1, \ldots, u_k\}=\{p_1, \ldots, p_m\}\cup \{u_{m+1}', \ldots, u_k'\}
\]
with $p_i$ an almost prime-like atom for all $i\in [1,m]$ and where
\[
  \{v_1, \ldots, v_l\}=\{q_1, \ldots, q_n\}\cup \{v_{n+1}', \ldots, v_l'\}
\]
with $q_j$ an almost prime-like atom for all $j\in [1,n]$.
Further suppose that for all $i \in [1,m]$ and $j \in [n+1,l]$ it holds that $p_i \nwr v_{j}'$, and for all $j \in [1,n]$ and $i \in [m+1,k]$ it holds that $q_j \nwr u_i'$.
Then there is a bijective correspondence between the set of associativity classes of the $p_i$ and the set of associativity classes of the $q_j$.
\end{proposition}

\begin{proof}
Since $\varepsilon u_1\cdots u_k=\eta v_1 \cdots v_l$, we have $p_i \wr v_1\cdots v_l$ for each $i\in [1,m]$.
By \cref{consdivprod}, this implies that for each $i \in [1,m]$, $p_i \wr v_j$ for some $j \in [1,l]$.
Then $p_i \nwr v_j'$ for any $j \in [n+1,l]$ and thus $p_i \wr q_j$ for some $j \in [1,n]$.
As $p_i$ and $q_j$ are both atoms, $p_i \simeq q_j$.
Therefore each almost prime-like element $p_i$ occurring in $z$ is associated to an almost prime-like element $q_j$ occurring in $z'$.
A symmetrical argument shows that each almost prime-like element $q_j$ occurring in $z'$ is associated to an almost prime-like element $p_i$ occurring in $z$.
The result follows.
\end{proof}

\begin{corollary}\label{primesalvage2}
Let $H$ be atomic, and let $q$ be an atom of $H$.
The following statements are equivalent.
\begin{enumerateequiv}
  \item\label{primesalvage2:cdiv} $q$ is an almost prime-like element.
  \item\label{primesalvage2:rf} If $a \in H$, $z \in \sZ^*(a)$ and $q$ occurs in $z$, then $q$ occurs in $z'$ for all $z' \in \sZ^*(a)$.
  \item\label{primesalvage2:pf} If $a \in H$, $z \in \sZ_p(a)$ and $q$ occurs in $z$, then $q$ occurs in $z'$ for all $z' \in \sZ_p(a)$.
\end{enumerateequiv}
\end{corollary}

\begin{proof}
  The equivalence of \ref*{primesalvage2:rf} and \ref*{primesalvage2:pf} is clear by the discussion preceeding \cref{primesalvage}.
  If $q$ is an almost prime-like element of $H$, then \ref*{primesalvage2:rf} holds by \cref{primesalvage}.
  Now suppose that \ref*{primesalvage2:rf} holds, and suppose that $q \mid_p ab$ for some $a$,~$b \in H$.
  If $q \nmid_p a$ and $q \nmid_p b$, then $q$ cannot occur in any rigid factorization of $a$ or $b$.
  Let $z \in \sZ^*(a)$ and $z' \in \sZ^*(b)$.
  Then $z \rfop z'$ is a rigid factorization of $ab$ in which $q$ does not occur, but this contradicts $q \mid_p ab$ by \cref{primesalvage}.
\end{proof}

\begin{example}\label{weirdprimes}
Let $S=\langle a,b,c\mid aba=ba^3bc\rangle$.
Clearly $S$ is an Adyan semigroup and hence cancellative. Considering the single relation defining $S$, we see that any word in $\mathcal F^{\ast}(a,b,c)$ that contains $a$ (respectively $b$) can only be rewritten in such a way that it again contains $a$ (respectively $b$).
Therefore the two atoms $a$ and $b$ are almost prime-like elements in $S$.
Since $aba=ba^3bc$, the atom $c$ is not an almost prime-like element.
Considering the relation $aba=ba^3bc$, we see that the product $aba$ of almost prime-like elements does not have a unique permutable factorization in $S$.
\end{example}

The phenomena exhibited in \cref{primesalvage}, \cref{primesalvage2}, and \cref{weirdprimes} motivate the following definition.
In particular we associate to an almost prime-like element $q$ a multi-valued $q$-adic valuation, that corresponds to the concept of a $p$-adic valuation of a prime $p$ in the commutative setting (cf. \cite[Definition 1.1.9]{GHK06}).

\begin{definition}\label{prime-like}
  Let $\wr$ be a divisibility relation on $H$, and let $q \in H$ be an almost prime-like atom with respect to $\wr$.
  \begin{enumerate}
    \item
      Let $a \in H$.
      We define $\sV_q(a) \subset \bN_0$ as follows:
      A non-negative integer $n \in \bN_0$ is contained in $\sV_q(a)$ if and only if there exist $k \in \bN_0$ and $u_1$,~$\ldots\,$, $u_k \in \cA(H)$ such that $a \simeq u_1\cdots u_k$ and $n = \card{\{\,i \in [1,k] : u_i \simeq q\,\}}$.
      In particular $\sV_q(a) = \{0\}$ if $a \in H^\times$.
      We call $\sV_q(a)$ the \emph{$q$-adic valuation} of $a$.

    \item
      The almost prime-like element $q$ is \emph{prime-like (with respect to $\wr$)} provided that $\card{\sV_q(a)}=1$ for all $a \in H$.
  \end{enumerate}
\end{definition}

Note that unlike in the commutative setting, $\sV_q(a)$ need not be a singleton.
Indeed, if $S$ is as in \cref{weirdprimes}, then $\sV_a(aba)=\{2,3\}$ and $\sV_b(aba)=\{1,2\}$.
By considering the Adyan semigroup $S=\langle a,b \mid aba=ba^3b\rangle$, one sees that even if each atom of an atomic cancellative semigroup $S$ is almost prime-like, $S$ need not be permutably factorial.
However, if $q$ is an almost prime-like atom and $a \in H$ is a non-unit, then $0\in \sV_q(a)$ if and only if $\sV_q(a)=\{0\}$ (by \cref{primesalvage}).

\begin{remark}\label{prime-like-ideals}
  We compare the notion of (almost) prime-like elements in cancellative semigroups and rings to more established concepts.
  We do so by means of left-right divisibility, but recall that the choice of divisibility relation does not matter if the semigroup or ring is atomic (by \cref{cdivind}).
  Let $S$ be a cancellative semigroup.
  A proper semigroup ideal $P \subset S$ is called a \emph{completely prime ideal} if, for all $a$, $b \in S$, $ab \in P$ implies $a \in P$ or $b \in P$.
  It is immediate from the definitions that $p \in S$ is an almost prime-like element if and only if $SpS$ is a completely prime ideal of $S$.

  Now let $R$ be a ring and consider the cancellative semigroup $S = R^\bullet$ of non zero-divisors of $R$.
  A proper ideal $P$ of $R$ is called a \emph{prime ideal} if, for all $a$,~$b \in R$,\; $aRb \subset P$ implies $a \in P$ or $b \in P$, and $P$ is called a \emph{completely prime ideal} if, for all $a$,~$b \in R$,\; $ab \in P$ implies $a \in P$ or $b \in P$.
  If $p \in R$ is an almost prime-like element, then in general $pR$ and $Rp$ need not even be ideals of $R$, while the ideal of $R$ generated by $p$ need not even be proper.
  For instance, let $D$ be a commutative PID and $R=M_n(D)$ with $n \in \bN_{\ge 2}$.
  An element $A \in R^\bullet$ is an atom if and only if $\det(A)$ is a prime element of $D$ if and only if $A$ is a prime-like element of $R^\bullet$ (this follows easily by means of the Smith Normal Form, see also the examples at the end of \cref{sec:abelianization}).
  Thus, an (almost) prime-like element $A \in R^\bullet$ is not contained in any proper ideal of $R$.
  Conversely, if $P$ is a prime ideal of $M_n(D)$ then $P = Rp = pR = {}_R \langle p \rangle_R$ with $p$ a prime element of $D$.
  However, $p$ is not even an atom in $R^\bullet$ since $\det(p)=p^n$.
  Thus elements that generate principal prime ideals (as left ideals, right ideals, or two-sided ideals) need not be almost prime-like.

  However, suppose that $p \in R^\bullet$ is such that $Rp=pR$.
  One verifies directly: If $a \in R^\bullet$ and $x \in R$ with either $xp=a$ or $px=a$, then $x \in R^\bullet$.
  In particular $R^\bullet p = pR^\bullet$.
  Thus the following statements are equivalent for $a \in R^\bullet$:
  \begin{enumerate*}[label=(\alph*)]
    \item $p \mid_{l-r} a$,
    \item $p \mid_{l} a$,
    \item $p \mid_{r} a$,
    \item $a \in pR$,
    \item $a \in Rp$.
  \end{enumerate*}
  This implies that the ideal $Rp$ is a completely prime ideal if and only if $p$ is an almost prime-like element.
  In \cite{Chatters}, an element $p$ in a Noetherian ring $R$ is called a \emph{prime element} of $R$ if $Rp = pR$ and $Rp$ is a height-1 prime ideal of $R$ and completely prime.
  Thus any non zero-divisor prime element $p$ of $R$ is an almost prime-like element of $R^\bullet$.
\end{remark}

As the following lemma shows, the behavior of valuations for prime-like elements is quite similar to the behavior of valuations for prime elements in the commutative setting.

\begin{lemma}\label{additive}
Let $H$ be atomic, and let $q$ be an almost prime-like element of $H$.
Then $q$ is prime-like if and only if $\sV_q(a)+\sV_q(b)=\sV_q(ab)$ for all $a$,~$b \in H$ with $t(a)=s(b)$.
\end{lemma}

\begin{proof}
  By definition, $q$ is prime-like if and only if $\card{\sV_q(a)}=1$ for all $a \in H$.

  Note that for all $a$,~$b \in H$ with $t(a)=s(b)$, we trivially have $\sV_q(a) + \sV_q(b) \subset \sV_q(ab)$.
  Indeed, if $m \in \sV_q(a)$ and $n \in \sV_q(b)$, then there exist rigid factorizations $z=\rf[\varepsilon]{u_1,\cdots,u_k}$ of $a$ and $z'=\rf[\eta]{v_1,\cdots,v_l}$ of $b$ with $k$,~$l \in \bN_0$, $\varepsilon$,~$\eta \in H^\times$, and $u_1$,~$\ldots\,$,~$u_k$, $v_1$,~$\ldots\,$,~$v_l \in \cA(H)$ such that $\card{\{\, i \in [1,k] : u_i \simeq q\,\}} = m$ and $\card{\{\, j \in [1,l] : v_j \simeq q\,\}} = n$.
  Since $z \rfop z'$ is a rigid factorization of $ab$, we have $m+n \in \sV_q(ab)$.

  Suppose first that $\card{\sV_q(a)}=1$ for all $a \in H$.
  Let $m$,~$n \in \bN_0$ with $\sV_q(a) = \{m\}$ and $\sV_q(b) = \{n\}$.
  Since $\{m+n\} = \sV_q(a) + \sV_q(b) \subset \sV_q(ab)$, and the latter set is a singleton, it must be the case $\sV_q(ab) = \{m+n\}$.

  We now prove the converse.
  Let $a \in H$, $k \in \bN_0$, $\varepsilon \in H^\times$ and $u_1$,~$\ldots\,$,~$u_k \in \cA(H)$ be such that $a = \varepsilon u_1\cdots u_k$.
  Then
  \[
    \sV_q(a) = \sV_q(\varepsilon) + \sV_q(u_1) + \cdots + \sV_q(u_k)
  \]
  by hypothesis.
  Since $\varepsilon \in H^\times$, $\sV_q(\varepsilon) = \{0\}$.
  Since each $u_i$ is an atom in $H$, for each $i \in [1,k]$ either $\sV_q(u_i)=\{0\}$ (if $u_i \not\simeq q$) or $\sV_q(u_i)=\{1\}$ (if $u_i \simeq q$).
  As $\sV_q(u_i)$ is a singleton for all $i \in [1,k]$, so is $\sV_q(a)$.
\end{proof}

We have the following immediate corollary to \cref{primesalvage} which generalizes the familiar result from the commutative setting (cf. \cite[Proposition 1.1.8]{GHK06}).

\begin{corollary}\label{acdunique}
Let $\wr$ be a divisibility relation on $H$.
Let $a \in H$ and suppose that $z=\rf[\varepsilon]{u_1,\cdots,u_k}$ and $z'=\rf[\eta]{v_1,\cdots,v_l}$ are two rigid factorizations of $a$, where
\[
  \{u_1, \ldots, u_k\}=\{p_1, \ldots, p_m\}\cup \{u_{m+1}', \ldots, u_k'\}
\]
with $p_i$ a prime-like atom for all $i\in [1,m]$ and where
\[
  \{v_1, \ldots, v_l\}=\{q_1, \ldots, q_n\}\cup \{v_{n+1}', \ldots, v_l'\}
\]
with $q_j$ a prime-like atom for all $j\in [1,n]$.
Further suppose that for all $i \in [1,m]$ and $j \in [n+1,l]$ it holds that $p_i \nwr v_{j}'$, and for all $j \in [1,n]$ and $i \in [m+1,k]$ it holds that $q_j \nwr u_i'$.
Then $m=n$ and there exists a permutation $\sigma \in \fS_m$ such that $p_i \simeq q_{\sigma(i)}$ for all $i \in [1,m]$.
\end{corollary}

\begin{remark}
  \envnewline
  \begin{enumerate}
    \item
      Even products of prime-like elements need not have unique permutable factorizations as is exhibited by the following example.
      Let $S=\langle a,b \mid a^2=ba^2b\rangle$. Then $a$ is prime-like, yet $a^2$ does not have a unique permutable factorization.

    \item
      Let $H$ be a commutative semigroup.
      For the sake of completeness, we note that neither the concept of almost prime-like elements nor that of prime-like elements coincide with that of absolutely irreducible elements --- atoms $u$ such that $u^n$ has a unique permutable factorization for all $n \in \bN$.
      Let $m \in \bN$, $n \in \bN_{\ge 2}$ and $S=\langle a,b \mid ab^ma=b^n\rangle$.
Then $b$ is almost prime-like (in fact prime-like if $m=n$), but is not absolutely irreducible.
Conversely, if $S'=\langle a,b,c \mid ab=bc\rangle$, then $a$ is absolutely irreducible, yet is not almost prime-like.
  \end{enumerate}
\end{remark}

We now study permutable factorizations by means of divisibility relations.
By \cref{cdivind} we may consider only the divisibility relation $\mid_p$.
Recall, from \cref{sec:factorizations}, that $H$ is \emph{permutably factorial} if $\card{\sZ_p(a)}=1$ for all $a \in H$.
Explicitly, $H$ is atomic and for all non-units $a \in H$, whenever $a=u_1\cdots u_k = v_1 \cdots v_l$ with $k$,~$l \in \bN$ and atoms $u_1$,~$\ldots\,$,~$u_k$, $v_1$,~$\ldots\,$,~$v_l \in \cA(H)$, $k=l$ and there exists a permutation $\sigma \in \fS_k$ with $u_i \simeq v_{\sigma(i)}$ for all $i \in [1,k]$.

We shall now see that for an atomic cancellative small category $H$, the conditions
\smallskip
\begin{enumerate}
\item \emph{Every atom is almost prime-like} and
\item \emph{Every atom is prime-like}
\end{enumerate}
provide a measure of how close $H$ is to being permutably factorial.
\smallskip

\begin{proposition}\label{p-factorial-acd}
  $H$ is permutably factorial if and only if $H$ is atomic and every atom of $H$ is prime-like.
\end{proposition}

\begin{proof}
Suppose first that $H$ is permutably factorial.
Then $H$ is atomic.
If $u$ is an atom in $H$ and $u \mid_p ab$ for some $a$,~$b \in H$ with $t(a)=s(b)$, then $u$ occurs in the unique permutable factorization $\pf{z}$ of $ab$ by \subref{consdivprod:assoc}.
However, $\pf{z} = \pf{x} \rfop \pf{y}$ with $\pf{x}$ and $\pf{y}$ being the unique permutable factorization of $a$ and $b$.
Thus, $u$ must occur in either $\pf{x}$ or $\pf{y}$ implying $u \mid_p a$ or $u \mid_p b$, and so $u$ is almost prime-like.
Moreover, since $H$ is permutably factorial, $\card{\sV_q(a)} = 1$ for any element $a\in H$ and any almost prime-like element $q$ of $H$.

If, conversely, $H$ is atomic, and every atom of $H$ is prime-like, then \cref{acdunique} implies that $H$ is permutably factorial.
\end{proof}

Applying \cref{additive}, we immediately obtain the following corollary.

\begin{corollary}\label{alt-perm-fact}
$H$ is permutably factorial if and only if $H$ is atomic and the following two conditions are satisfied.
\begin{enumerate}
\item Every atom in $H$ is almost prime-like.
\item For every almost prime-like element $q \in H$ and for all $a$,~$b \in H$ with $t(a)=s(b)$, we have $\sV_q(ab) = \sV_q(a) + \sV_q(b)$.
\end{enumerate}
\end{corollary}

We now briefly consider generalizations of the $\omega$-invariant and the tame degree, which are well-studied invariants in the commutative setting, and measure how far away an atom is from being a prime element.
Accordingly, our invariants will give a measure of how far away an atom is from being almost prime-like.

\begin{definition}\label{omega} Let $H$ be atomic and let $a$,~$b \in H$.
\begin{enumerate}
  \item\label{omega:nonunits} Define $\omega'_p(a,b)$ to be the smallest $N \in \bN_0\cup \{\infty\}$ with the following property: For all $n \in \bN$ and $a_1$,~$\ldots\,$,~$a_n \in H$ with $a=a_1\cdots a_n$, if $b\mid_p a$, then there exists $k \in [0,N]$ and an injective map $\sigma\colon [1,k] \to [1,n]$ such that the permuted subproduct $a_\sigma = s(a_{\sigma(1)}) a_{\sigma(1)}\cdots a_{\sigma(k)}$ is defined, and $b \mid_p a_\sigma$.
    Set $\omega_p'(H,b) = \sup\{\, \omega_p'(a,b) : a \in H \,\}$ and $\omega_p'(H)=\sup \{\,\omega_p'(H,u) : u \in \cA(H) \,\}$.
  \item\label{omega:atoms} Define $\omega_p(a,b)$ to be the smallest $N \in \bN_0\cup \{\infty\}$ with the following property: For all $n \in \bN$ and atoms $u_1$,~$\ldots\,$,~$u_n$ of $H$ with $a=u_1\cdots u_n$, if $b\mid_p a$, then there exists $k \in [0,N]$ and an injective map $\sigma\colon [1,k] \to [1,n]$ such that the permuted subproduct $a_\sigma = s(u_{\sigma(1)}) u_{\sigma(1)}\cdots u_{\sigma(k)}$ is defined, and $b \mid_p a_\sigma$.
    Set $\omega_p(H,b) = \sup\{\, \omega_p(a,b) : a \in H \,\}$ and $\omega_p(H)=\sup \{\,\omega_p(H,u) : u \in \cA(H) \,\}$.
\end{enumerate}
\end{definition}

Note that $\omega'_p(H,a) = \omega_p(H,a) = 0$ if and only if $a \in H^\times$, and $\omega'_p(H,a)=\omega_p(H,a)=1$ if and only if $a$ is an almost prime-like element.
We always have $\omega_p(H,a) \le \omega_p'(H,a)$, and if $H$ is a commutative semigroup, then $\omega_p(H,a) = \omega_p'(H,a) = \omega(H,a)$, where $\omega(H,a)$ is the usual $\omega$-invariant as defined in the commutative setting (see \cite[Definition 2.8.14]{GHK06}).
The following example illustrates that $\omega_p(H,a) = \omega_p'(H,a)$ does not hold in general for noncommutative semigroups.

\begin{example}\label{omega-differs}
  Let $S = \langle a,b,c,d,e \mid ab=cd,\, cede = ba \rangle$.
  The semigroup is Adyan and hence cancellative. Moreover, $S$ is reduced and atomic with $\cA(S) = \{a,b,c,d,e\}$.
  We claim $\omega_p(S,a) = 2$.
  Clearly $\omega_p(S,a) \ge 2$, since $a \mid_p cd=ab$, but $a$ does not permutably divide any permuted subproduct of $cd$.
  Suppose $k \in \bN$ and $u_1$,~$\ldots\,$,~$u_k \in \cA(S)$ are such that $a \mid_p u_1\cdots u_k$.
  By the defining relations of $S$, either $u_i = a$ for some $i \in [1,k]$, $k \ge 2$ and $u_iu_{i+1}=cd$ for some $i \in [1,k-1]$, or that $k \ge 4$ and $u_iu_{i+1}u_{i+2}u_{i+3}=cede$ for some $i \in [1,k-4]$.
  In any of these cases we can take a subproduct of at most two elements that is divided up to permutation by $a$ (due to $cd=ab$ in the second and third case).

  However, $\omega_p'(S,a) \ge 3$:
  Let $a_1=ce$, $a_2=d$ and $a_3=e$.
  Then $a \mid_p a_1a_2a_3=ba$, but clearly $a \nmid_p a_i$ for any $i \in [1,3]$.
  Moreover $\{\, a_i a_j : i,j \in [1,3], i\ne j \,\} = \{ ced, ce^2, de, dce, ece, ed \}$, none of which is divided up to permutation by $a$.
\end{example}

However, even in the noncommutative setting we have the following.

\begin{proposition}\label{pf-omega}
$H$ is permutably factorial if and only if $H$ is atomic, $\omega_p'(H)=\omega_p(H)\le1$ and every almost prime-like element is prime-like.
\end{proposition}

\begin{proof}
  We have $\omega_p'(H)=\omega_p(H)=0$ if and only if $H$ is a groupoid. We may from now on exclude this trivial case, and assume that $H$ is not a groupoid.
  Suppose first that $H$ is permutably factorial.
  By \cref{p-factorial-acd}, $H$ is atomic and every atom of $H$ is prime-like.
  Thus, for all $u \in \cA(H)$, we have $\omega_p(H,u) = \omega_p'(H,u) = 1$, and therefore $\omega_p'(H)=\omega_p(H)=1$.

  We now show the converse implication.
  If $\omega_p'(H)=\omega_p(H)=1$, then every atom in $H$ is almost prime-like.
  By hypothesis, therefore all atoms of $H$ are prime-like, and thus $H$ is permutably factorial by \cref{p-factorial-acd}.
\end{proof}

Continuing our discussion of the notion of divisibility up to permutation in $\sZ^*(H)$ and $\sZ_p(H)$ preceding \cref{primesalvage}, let $x=\rf[\varepsilon]{w_1,\cdots,w_m} \in \sZ^*(H)$ with $m \in \bN_0$, $\varepsilon \in H^\times$ and $w_1$,~$\ldots\,$,~$w_m \in \cA(H)$.
We observe that, for $z \in \sZ^*(H)$ with $z=\rf[\eta]{u_1,\cdots,u_k}$ where $k \in \bN_0$, $\eta \in H^\times$ and $u_1$,~$\ldots\,$,~$u_k \in \cA(H)$, we have $x \mid_p z$ if and only if there exists an injective map $\sigma\colon [1,m] \to [1,k]$ such that $w_i \simeq u_{\sigma(i)}$ for all $i \in [1,m]$.
Note that this is also equivalent to $\pf{x} \mid_p \pf{z}$ in $\sZ_p(H)$.

\begin{definition}
  Let $H$ be atomic and $a \in H$.
  For a permutable factorization $x \in \sZ_p(H)$, let $\st_p(a,x)$ denote the smallest $N \in \bN_0 \cup \{\infty\}$ with the following property:
  \begin{enumerate}
    \item []
      If there exists any $z_0 \in \sZ_p(a)$ such that $x \mid_p z_0$, and $z \in \sZ_p(a)$ is an arbitrary permutable factorization of $a$, then there exists some $z' \in \sZ_p(a)$ with $x \mid_p z'$ in $\sZ_p(H)$ and with $\sd_p(z,z') \leq N$.
  \end{enumerate}
  For subsets $H' \subset H$ and $Z \subset \sZ_p(H)$, we define
  \[
    \st_p(H', Z) = \sup\{\, \st_p(a, z) : a \in H', z \in Z \,\} \in \bN_0 \cup \{\infty\}.
  \]
  In particular, for $u \in \cA(H)$ we set $\st_p(H,u) = \st_p(H,\{u\})$, and the \emph{(permutable) tame degree of $H$} is defined as $\st_p(H) = \st_p(H, \cA(H)) \in \bN_0 \cup \{\infty\}$.
\end{definition}

If $H$ is a commutative semigroup, then $\st_p(H,u) = \st(H,u)$, where $\st(H,u)$ denotes the usual tame degree as defined in the commutative setting (see \cite[Definition 1.6.4]{GHK06}).

We are now able to give a characterization of permutable factoriality in terms of the tame degree.
Compare with \cite[Theorem 1.6.6]{GHK06} for the commutative analogue.

\begin{proposition}\label{perm-tame}
Let $H$ be atomic.
The following statements are equivalent.
\begin{enumerate}
  \item\label{perm-tame:fact} $H$ is permutably factorial.
  \item\label{perm-tame:val} $\st_p(H)=0$ and every almost prime-like element is prime-like.
  \item\label{perm-tame:abscon} $\st_p(H,\sZ_p(H))=0$ and every almost prime-like element is prime-like.
\end{enumerate}
\end{proposition}

\begin{proof}
  It is obvious from the definitions that \ref*{perm-tame:abscon} implies \ref*{perm-tame:val}.
  If $H$ is permutably factorial, then every atom of $H$ is prime-like by \cref{p-factorial-acd}, and by definition of permutable factoriality, $\card{\sZ_p(a)} = 1$ for all $a \in H$. Then $\st_p(H,\sZ_p(H)) = 0$ follows trivially and thus \ref*{perm-tame:fact} implies \ref*{perm-tame:abscon}.

  We now show that \ref*{perm-tame:val} implies \ref*{perm-tame:fact}.
  If $\st_p(H)=0$, then if $u \in \cA(H)$ and $u \mid_p a$ for some $a \in H$, $u \mid_p z$ for all $z \in \sZ_p(a)$ by definition of the tame degree.
  Therefore $u$ is almost prime-like by \cref{primesalvage2}.
Since $u$ is prime-like by hypothesis, \cref{p-factorial-acd} implies that $H$ is permutably factorial.
\end{proof}

The following example illustrates that, despite providing some insight into factorizations in the noncommutative setting, this noncommutative tame degree does not carry nearly as much information as in the commutative case.

\begin{example}
  Let $S=\langle a,b \mid aba=bab\rangle$. Then $\st_p(S,a)=\st_p(S,b)=0$ and hence $\st_p(S)=0$.
  However, $S$ is not permutably factorial and thus the additional hypotheses given in \labelcref{perm-tame:val} and \labelcref{perm-tame:abscon} of \cref{perm-tame} cannot be removed.
\end{example}

We now show that any isoatomic weak transfer homomorphism preserves the values of $\st(-)$ and, under some additional restrictions, of $\omega_p(-)$, $\omega_p'(-)$.
We will see specific applications of this in \cref{cor:tri-fact} and \cref{prop:krull-finite-omega-invariant}.

\begin{proposition}\label{awth-tame}
  Let $T$ be an atomic cancellative small category and let $\phi\colon H \to T$ be an isoatomic weak transfer homomorphism.
  Let $\phi_p\colon \sZ_p(H) \to \sZ_p(T)$ denote the extension of $\phi$ to permutable factorizations as in \cref{ext-hom}.
  \begin{enumerate}
    \item\label{awth-tame:divpf}
      For $z$,~$z' \in \sZ_p(H)$ we have $z \mid_p z'$ if and only if $\phi_p(z) \mid_p \phi(z')$.
    \item\label{awth-tame:div}
      For $a$, $b \in H$ we have $a \mid_p b$ if and only if $\phi(a) \mid_p \phi(b)$.
    \item\label{awth-tame:omega}
      Suppose $H$ and $T$ are semigroups.
      For $a \in H$ we have $\omega_p(H,a) \le \omega_p(T,\phi(a))$ and in particular $\omega_p(H) \le \omega_p(T)$.
      If $T$ is commutative, then $\omega_p(H,a) = \omega_p(T,\phi(a))$ and $\omega_p(H) = \omega_p(T)$.
    \item\label{awth-tame:omega'}
      Suppose $H$ and $T$ are semigroups.
      For $a \in H$ we have $\omega_p'(H,a) \le \omega_p'(T,\phi(a))$ and in particular $\omega_p'(H) \le \omega_p'(T)$.
      If $T$ is commutative, then $\omega_p'(H,a) = \omega_p'(T,\phi(a))$ and $\omega_p'(H) = \omega_p'(T)$.
    \item\label{awth-tame:tame} For $a \in H$ and $x \in \sZ_p(H)$ we have $\st_p(a,x) = \st_p(\phi(a),\phi_p(x))$ and in particular $\st_p(H) = \st_p(T)$.
  \end{enumerate}
\end{proposition}

\begin{proof}
  \ref*{awth-tame:divpf}
  Let $k$,~$l \in \bN_0$, $u_1$,~$\ldots\,$,~$u_k$, $v_1$,~$\ldots\,$,~$v_l \in \cA(H)$, and $\varepsilon$, $\eta \in H^\times$ be such that $z=\pf[\varepsilon]{u_1,\cdots,u_k}$ and $z'=\pf[\eta]{v_1,\cdots,v_l}$.
  Suppose first that $z \mid_p z'$.
  Then $k \le l$ and there exists an injective map $\sigma\colon [1,k] \to [1,l]$ such that $u_i \simeq v_{\sigma(i)}$ for all $i \in [1,k]$.
  Then it is clear that $\phi(u_i) \simeq \phi(v_{\sigma(i)})$ for all $i \in [1,k]$.
  Since $\phi_p(z) = \pf[\phi(\varepsilon)]{\phi(u_1),\cdots,\phi(u_k)}$ and $\phi_p(z') = \pf[\phi(\eta)]{\phi(v_1),\cdots,\phi(v_l)}$, we have $\phi_p(z) \mid_p \phi_p(z')$.

  Now suppose that $\phi(z) \mid_p \phi(z')$.
  Again $k \le l$ and there exists an injective map $\sigma\colon [1,k] \to [1,l]$ such that $\phi(u_i) \simeq \phi(v_{\sigma(i)})$ in $T$.
  Since $\phi$ is isoatomic, $u_i \simeq v_{\sigma(i)}$ and hence $z \mid_p z'$.

  \ref*{awth-tame:div}
  Note that $a \mid_p b$ if and only if there exist $z \in \sZ_p(a)$ and $z' \in \sZ_p(b)$ such that $z \mid_p z'$.
  Since $\phi_p$ restricted to $\sZ_p(a) \to \sZ_p(\phi(a))$, respectively $\sZ_p(b) \to \sZ_p(\phi(b))$, is surjective by \subref{ext-hom:wth}, the claim follows from \ref*{awth-tame:divpf}.

  \ref*{awth-tame:omega}
  Let $H$ and $T$ be semigroups.
  Suppose first that $a \mid_p u_1\cdots u_n$ for $n \in \bN_0$ and atoms $u_1$,~$\ldots\,$,~$u_n$ of $H$.
  Then $\phi(a) \mid_p \phi(u_1)\cdots \phi(u_n)$ by \ref*{awth-tame:div}.
  Thus there exists $k \in [0,n]$ and an injective map $\sigma\colon [1,k] \to [1,n]$ such that
  \[
    \phi(a) \mid_p \phi(u_{\sigma(1)})\cdots \phi(u_{\sigma(k)})=\phi(u_{\sigma(1)}\cdots u_{\sigma(k)}),
  \]
  where the product $u_{\sigma(1)}\cdots u_{\sigma(k)}$ is defined since $H$ is a semigroup.
  Again by \ref*{awth-tame:div}, $a \mid_p u_{\sigma(1)}\cdots u_{\sigma(k)}$, showing $\omega(H,a) \le \omega(T,\phi(a))$.

  Let $T$ be commutative.
  Now suppose that $\phi(a) \mid_p v_1\cdots v_n$ for $n \in \bN_0$ and atoms $v_1$,~$\ldots$,~$v_n$ of $T$.
  Since $\phi$ is a weak transfer homomorphism, there exist atoms $u_1$,~$\ldots\,$,~$u_n$ of $H$ such that $\phi(u_i) \simeq v_i$ for all $i \in [1,n]$.
  Since $H$ is a semigroup, the product $u_1\cdots u_n$ is defined and, since $T$ is a commutative semigroup, $\phi(u_1\cdots u_n) \simeq v_1\cdots v_n$.
  Therefore $\phi(a) \mid_p \phi(u_1\cdots u_n)$, and we have $a \mid_p u_1\cdots u_n$ by \ref*{awth-tame:div}.
  Thus, there exists $k \in [0,n]$ and an injective map $\sigma\colon [1,k] \to [1,n]$ with $a \mid_p u_{\sigma(1)}\cdots u_{\sigma(k)}$.
   Then $\phi(a) \mid_p \phi(u_{\sigma(1)})\cdots\phi(u_{\sigma(k)})$ by \ref*{awth-tame:div} again, and hence $\phi(a) \mid_p v_{\sigma(1)}\cdots v_{\sigma(k)}$ (for this we use the commutativity of $T$ again).
  Thus $\omega_p(T,\phi(a)) \le \omega_p(H,a)$.
  Since $T = T^\times\phi(H)T^\times$, it follows that $\omega_p(H)= \omega_p(T)$.

  \ref*{awth-tame:omega'} The proof of \ref*{awth-tame:omega'} is analogous to that of \ref*{awth-tame:omega}.

  \ref*{awth-tame:tame}
  We first show $\st_p(a, x) \le \st_p(\phi(a),\phi_p(x))$.
  Let $z \in \sZ_p(a)$ and suppose there exists $z_0 \in \sZ_p(a)$ such that $x \mid_p z_0$.
  Then $\phi_p(x) \mid_p \phi_p(z_0)$ by \ref*{awth-tame:divpf}, and $\phi_p(z_0) \in \sZ_p(\phi(a))$.
  Thus there exists $\lbar z' \in \sZ_p(\phi(a))$ with $\phi_p(x) \mid_p \lbar z'$ and $\sd_p(\lbar z',\phi_p(z)) \le \st_p(\phi(a), \phi_p(x))$.
  Since the restriction of $\phi_p$ to $\sZ_p(a) \to \sZ_p(\phi(a))$ is surjective, there exists $z' \in \sZ_p(a)$ such that $\phi_p(z') = \lbar z'$.
  Since $\phi$ is isoatomic, \cref{strong-weak-distance} implies $\sd_p(z',z) = \sd_p(\phi_p(z'),\phi_p(z)) \le \st_p(\phi(a),\phi_p(x))$.
  Moreover, \ref*{awth-tame:divpf} implies $x \mid_p z'$ and thus we have $\st_p(a,x) \le \st_p(\phi(a),\phi_p(x))$.

  We now show $\st_p(\phi(a),\phi_p(x)) \le \st_p(a,x)$.
  Suppose that $\lbar z \in \sZ_p(\phi(a))$ and there exists $\lbar z_0 \in \sZ_p(\phi(a))$ such that $\phi_p(x) \mid_p \lbar z_0$.
  Again, there exist $z$,~$z_0 \in \sZ_p(a)$ such that $\phi_p(z) = \lbar z$ and $\phi_p(z_0) = \lbar z_0$.
  Then $x \mid_p z_0$ and thus there exists $z' \in \sZ_p(a)$ such that $\sd_p(z',z) \le \st_p(a,x)$ and $x \mid_p z'$.
  Hence $\sd_p(\phi(z'),\phi(z)) = \sd_p(z',z) \le \st_p(a,x)$ and $\phi_p(x) \mid_p \phi_p(z')$, proving the claim.
  Since $T = T^\times \phi(H) T^\times$, it follows that $\st_p(H) = \st_p(T)$.
\end{proof}

\begin{corollary}
  Let $H$ be a cancellative semigroup possessing an isoatomic weak transfer homomorphism to a commutative atomic cancellative semigroup.
  Then
  \begin{enumerate}
    \item $\rho(b) \le \sup \sL(b) \le \omega_p(H,b)$ for all $b \in H$,
    \item $\omega_p(H,u) \le \st_p(H,u)$ if $u \in \cA(H)$ is not an almost prime-like element,
    \item $\rho(H) \le \omega_p(H) \le \st_p(H)$ unless $\st_p(H)=0$, and
    \item $\sc_p(H) \le \st_p(H)$.
  \end{enumerate}
\end{corollary}

\begin{proof}
  These inequalities hold whenever $H$ is a commutative semigroup, and hence, by the previous proposition, they also hold if $H$ possesses an isoatomic weak transfer homomorphism to a commutative atomic cancellative semigroup.
\end{proof}

The next examples show that the inequalities in the previous corollary fail to hold in general.

\begin{example}
  \envnewline
  \begin{enumerate}
    \item Let $S=\langle a,b,c \mid ba^{n-1} = a^{n-1}c \rangle$ for some $n \in \bN_{\ge 2}$.
      Since $a$ is almost prime-like, we have $\st_p(a)=0$.
      Moreover, $\st_p(b)=\st_p(c)=1$ and thus $\st_p(S)=1$.
      However, $\omega_p(S,a)=1$, $\omega_p(S,b)=\omega_p(S,c)=n$, and hence $\omega_p(S)=n$.

    \item Let $S=\langle a,b \mid ab=ba^{n-1}\rangle$, where $n \in \bN_{\geq 2}$.
      Since $a$ and $b$ are almost prime-like, $\st_p(S,a)=\st_p(S,b)=0$, whence $\st_p(S)=0$.
      Similarly $\omega_p(S,a)=\omega_p(S,b) = 1$ and $\omega(S) = 1$.

      However, it is clear that $\rho(S)=n/2$.
      Moreover,
      \[
        \sZ^*(a^m b) = \{ \rf{a^m,b},\, \rf{a^{m-1},b,a^{n-1}},\, \ldots,\, \rf{b,a^{m(n-1)}} \},
      \]
      and hence $\sL(a^m b) = \{\, m+1+ k(n-2) : k \in [0,m] \,\}$.
      Thus $\sup \sL(a^m b) = m(n-1)+1$ and $\rho(a^m b) = \frac{m(n-1)+1}{m+1}$, while $\omega_p(S,a^m b) = m+1$.
      Finally, $\sc_p(a^m b) = n - 2$, and hence $\sc_p(S) \ge n-2$.
  \end{enumerate}
\end{example}

\section{The abelianization of a noncommutative semigroup}\label{sec:abelianization}

In this section we study when the natural homomorphism $\pi\colon S \to \rab S$ from a cancellative semigroup to its reduced abelianization is a weak transfer homomorphism.
A necessary and sufficient condition is given in \cref{exwt} where we also see that whenever $\pi$ is a weak transfer homomorphism it must be isoatomic.
In \cref{wtup} we show that in this case $\pi$ satisfies a universal property with regards to weak transfer homomorphisms into commutative reduced cancellative semigroups.
Finally, we give applications to the semigroup of non zero-divisors of the ring of $n \times n$ upper triangular matrices over a commutative atomic domain, and the semigroup of non zero-divisors of the ring of $n \times n$ matrices over a PID.\spacefactor=\sfcode`\.{}

\begin{definition}\label{abelianization}
  Let $S$ be a semigroup and let $\ab \equiv$ be the smallest congruence relation on $S$ such that $ab \abrel ba$ for all $a$,~$b \in S$.
  \begin{enumerate}
    \item The \emph{abelianization of $S$} is the pair $(\ab S, \pi)$ consisting of $\ab S = S/\ab \equiv$ together with the canonical homomorphism $\pi \colon S \to \ab S$.
    \item The \emph{reduced abelianization of $S$} is the pair $(\rab S, \pi)$ consisting of $\rab S = \red{(\ab S)}$ together with the canonical homomorphism $\pi\colon S \to \rab S$.
      We denote the corresponding congruence on $S$ by $\rab \equiv$.
  \end{enumerate}
\end{definition}

\begin{remark}
  \envnewline
  \begin{enumerate}
    \item
      Explicitly, the congruence $\ab \equiv$ is given as follows:
      Let $a$,~$b \in S$.
      Then $a \abrel b$ if and only if there exist $m \in \bN$ and, for each $i \in [1,m]$, $k_i \in \bN$ and $c_{i,j} \in S$ for $j \in [1,k_i]$ as well as a permutation $\sigma_i \in \fS_{k_i}$ such that:
      \begin{equation}\label{eq:ab-rel}
        \begin{split}
          a                                                     &= c_{1,1}\cdots c_{1,k_1}, \\
          c_{1,\sigma_1(1)}\cdots c_{1,\sigma_1(k_1)} &= c_{2,1}\cdots c_{2,k_2}, \\
          \vdots \\
          c_{m-1,\sigma_{m-1}(1)}\cdots c_{m-1,\sigma_{m-1}(k_{m-1})} &= c_{m,1}\cdots c_{m,k_m}, \\
          c_{m,\sigma_m(1)}\cdots c_{m,\sigma_m(k_m)} &= b.
        \end{split}
      \end{equation}

    \item
      The abelianization $(\ab S, \pi)$ satisfies the following universal property:
      Let $T$ be any commutative semigroup, and let $\phi\colon S \to T$ be a semigroup homomorphism.
      Then there exists a unique homomorphism $\lbar{\phi}\colon \ab S \to T$ such that $\phi = \lbar{\phi} \circ \pi$.

    \item
      The reduced abelianization $(\rab S, \pi)$ satisfies the following universal property:
      Let $T$ be any commutative reduced semigroup, and let $\phi\colon S \to T$ be a semigroup homomorphism.
      Then there exists a unique homomorphism $\lbar{\phi}\colon \rab S \to T$ such that $\phi = \lbar{\phi} \circ \pi$.

      If associativity is a congruence relation on $S$, then $\ab{(\red S)}$ together with a canonical homomorphism $\pi'\colon S \to \ab{(\red S)}$ is defined.
      Again we see that every homomorphism $S \to T$ to a commutative reduced semigroup factors through $\pi'$ in a unique way.
      Therefore $(\ab{(\red S)}, \pi')$ satisfies the same universal property as $(\rab S,\pi)$ and hence the two semigroups must be canonically isomorphic.
      We identify $\rab S$ and $\ab{(\red S)}$ by means of this isomorphism.
  \end{enumerate}
\end{remark}

\begin{lemma}\label{ab-prop}
  Let $S$ be a cancellative semigroup, $\ab S$ its abelianization, and denote by $\pi\colon S \to \ab S$ the canonical homomorphism.
  Suppose that $\ab S$ is also cancellative.
  \begin{enumerate}
    \item\label{ab-prop:units} We have $\pi^{-1}({\ab S}^\times) = S^\times$.
    \item\label{ab-prop:atoms} Let $a \in S$.
      Then $a \in \cA(S)$ if and only if $\pi(a) \in \cA(\ab S)$.
    \item\label{ab-prop:assoc} If $u \in \cA(S)$, then $[u]_{\simeq} = \pi^{-1}([\pi(u)]_\simeq)$.
  \end{enumerate}
\end{lemma}

\begin{proof}
  \ref*{ab-prop:units}
  Clearly $S^\times \subset \pi^{-1}(\ab S^\times)$.
  Now let $a \in S$ be such that $\pi(a) \in \ab S^\times$.
  Then there exists $b \in S$ such that $\pi(a)\pi(b) = 1$, and hence $ab \abrel 1$.
  Using notation as in \cref{eq:ab-rel}, with $a$ replaced by $ab$ and $b$ replaced by $1$, we see that $c_{m,\sigma_m(1)}\cdots c_{m,\sigma_m(k_m)} = 1$ for some $c_{i,j}$ in $S$.
  Hence, for all $j \in [1,k_m]$, we have $c_{m,\sigma_m(j)} \in S^\times$.
  Continuing inductively, $c_{i,j} \in S^\times$ for all $i \in [1,m]$ and $j \in [1,k_i]$, and hence $ab \in S^\times$.
  Therefore $a \in S^\times$.

  \ref*{ab-prop:atoms}
  Suppose $a \in S \setminus S^\times$ is not an atom.
  Then $a=bc$ with $b$,~$c \in S\setminus S^\times$ and \ref*{ab-prop:units} implies $\pi(b),\pi(c) \in \ab S \setminus {\ab S}^\times$.
  Hence $\pi(a) = \pi(b) \pi(c)$ is not an atom.
  Conversely, suppose $\pi(a) \in S \setminus \cA(S)$.
  If $\pi(a) \in {\ab S}^\times$, then $a \in S^\times$, and hence we may assume that $\pi(a)$ is not a unit.
  Since $\pi(a)$ is not an atom, there exist $\lbar b, \lbar c \in {\ab S} \setminus {\ab S}^\times$ such that $\pi(a) = \lbar b \lbar c$.
  Since $\pi$ is surjective, there exist $b$,~$c \in S \setminus S^\times$ such that $\pi(b)=\lbar b$ and $\pi(c) = \lbar c$.
  Thus $a \abrel bc$.
  Using the notation of \cref{eq:ab-rel} to write out this relation, it follows inductively that for all $i \in [1,m]$, $c_{i,1}\cdots c_{i,k_i}$ is not an atom.
  Therefore $a$ is not an atom.

  \ref*{ab-prop:assoc}
  If $a \in S$ with $u \simeq a$, then $\pi(u) \simeq \pi(a)$.
  For the converse direction suppose $\pi(u) \simeq \pi(a)$.
  It suffices to show $u \simeq a$.
  Let $m \in \bN$, and $k_i \in \bN$, $c_{i,j} \in S$ for all $i \in [1,m]$, $j \in [1,k_i]$ be as in \cref{eq:ab-rel}.
  Since $u$ is an atom and $u = c_{1,1}\cdots c_{1,k_1}$, there exists some $j_1 \in [1,k_1]$ such that $c_{1,j_1}$ is an atom, and $c_{1,j} \in S^\times$ for all $j \in [1,k_1] \setminus \{j_1\}$.
  In particular, $u \simeq c_{1,j_1}$.
  Inductively it follows that for all $i \in [2,m]$, there exists $j_i \in [1,k_i]$ such that $c_{i,j_i} \in \cA(S)$ and $c_{i,j} \in S^\times$ for all $j \in [1,k_i] \setminus \{j_i\}$, and therefore $c_{i-1,j_{i-1}} \simeq c_{i,j_i}$.
  It follows that $u \simeq c_{m,j_m} \simeq a$.
\end{proof}

\begin{remark}
  \envnewline
  \begin{enumerate}
  \item
    Replacing $\ab S$ by $\rab S$ in \cref{ab-prop}, and assuming that $\rab S$ is cancellative,  we obtain the corresponding statements of the lemma for $\rab S$.
  \item
    Suppose $S$ is in fact a group.
    Then, by \subref{ab-prop:units}, $\ab S$ is a group.
    Using the universal property of the abelianization it follows that $\ab S$ satisfies the universal property of the abelianization of $S$ as a group.
    Thus, in this case, $\ab S$ is the just the usual abelianization of a group.
  \end{enumerate}
\end{remark}

\begin{definition}\label{rel0}
  Let $S$ be a semigroup.
  We define a relation $\equiv_p$ on $S$ as follows:
  For $a$,~$b \in S$ we set $a \equiv_p b$ if and only if there exist $m \in \bN_0$ and $a_1$, $\ldots\,$,~$a_m$, $b_1$, $\ldots\,$,~$b_m \in S$ such that $a \simeq a_1\cdots a_m$, $b \simeq b_1\cdots b_m$ and there exists a permutation $\sigma \in \fS_m$ such that $a_{i} \simeq b_{\sigma(i)}$ for all $i \in [1,m]$.
\end{definition}

If $S$ is atomic, the $a_1$, $\ldots\,$,~$a_m$ and $b_1$, $\ldots\,$,~$b_m$ in the definition of $\equiv_p$ can equivalently be taken to be atoms.
In this case, the definition may equivalently be stated as:
$a \equiv_p b$ if and only if there exist rigid factorizations $z$ of $a$ and $z'$ of $b$ such that $\sd_p(z,z') = 0$.
The relation $\equiv_p$ is obviously reflexive and symmetric, but may not be transitive.
If $a \simeq b$, then clearly $a \equiv_p b$.

\begin{lemma}\label{cong0}
  Let $S$ be a semigroup.
  The following statements are equivalent.
  \begin{enumerate}
    \item\label{cong0:transitive} $\equiv_p$ is transitive.
    \item\label{cong0:congruence} $\equiv_p$ is a congruence relation.
    \item\label{cong0:rab} $\equiv_p \;=\; \rab \equiv$.
  \end{enumerate}
\end{lemma}

\begin{proof}
  \ref*{cong0:transitive}${}\Rightarrow{}$\ref*{cong0:congruence}:
  The relation is symmetric and reflexive, and since we assume transitivity it is therefore an equivalence relation.
  Thus we must show that for all $a$,~$a'$,~$b$,~$b' \in S$, if $a \equiv_p a'$ and $b \equiv_p b'$, then $ab \equiv_p a'b'$.
  Since $a \equiv_p a'$, there exist $m \in \bN_0$, a permutation $\sigma \in \fS_m$, and elements $a_1$, $\ldots\,$~$a_m$, $a_1'$, $\ldots\,$,~$a_m' \in S$ such that $a \simeq a_1\cdots a_m$, $a' \simeq a_1'\cdots a_m'$ and $a_i \simeq a_{\sigma(i)}'$ for all $i \in [1,m]$.
  Similarly, there exist $n \in \bN_0$, a permutation $\tau \in \fS_n$, and elements $b_1$, $\ldots\,$~$b_n$, $b_1'$, $\ldots\,$,~$b_n' \in S$ such that $b \simeq b_1\cdots b_n$, $b' \simeq b_1'\cdots b_n'$ and $b_i \simeq b_{\tau(i)}'$ for all $i \in [1,n]$.

  If $m=0$, then $a$,~$a' \in S^\times$ and thus $ab \simeq b$ and $a'b' \simeq b'$.
  Therefore $ab \equiv_p a'b'$.
  We argue analogously if $n = 0$, and may now assume $m$,~$n > 0$.

  Without loss of generality, we replace $a_1$, $a_m$, $a_1'$, $a_m'$, $b_1$, $b_n$, $b_1'$, and $b_n'$ by associates such that $a=a_1\cdots a_m$, $a'=a_1'\cdots a_m'$, $b=b_1\cdots b_n$, and $b'=b_1'\cdots b_n'$.
  Then $ab=(a_1\cdots a_m)(b_1\cdots b_n)$ and $a'b'=(a_1'\cdots a_m')(b_1'\cdots b_n')$, each written as a product of $m+n$ atoms of $S$.
  Moreover, applying the permutation $(\sigma, \tau)$, interpreted accordingly as a permutation on $[1,m+n]$, we see that $ab \equiv_p a'b'$ and hence $\equiv_p$ is a congruence relation on $S$.

  \ref*{cong0:congruence}${}\Rightarrow{}$\ref*{cong0:rab}:
  Let $a$,~$b \in S$.
  If $a \equiv_p b$, then $a \rabrel b$.
  From the definition of $\equiv_p$ it follows that $ab \equiv_p ba$.
  Thus $\ab\equiv \;\subset\; \equiv_p \;\subset\; \rab\equiv$ and
  moreover, $S/\!\equiv_p$ is reduced. Since $\rab\equiv$ is the minimal
  congruence containing $\ab\equiv$ with respect to being reduced, and by assumption $\equiv_p$ is indeed a congruence, it follows that $\equiv_p = \rab\equiv$.

  \ref*{cong0:rab}${}\Rightarrow{}$\ref*{cong0:transitive}: Clear, since $\rab \equiv$ is transitive.
\end{proof}

\begin{proposition}\label{exwt}
  Let $S$ be a cancellative semigroup and suppose that $\rab S$ is also cancellative.
  The following statements are equivalent.
  \begin{enumerate}
    \item\label{exwt:explicit} If $a$, $b \in S$ are such that $a \equiv_p b$ and $a \simeq u_1\cdots u_m$ with $m \in \bN_0$ and $u_1$, $\ldots\,$,~$u_m \in \cA(S)$, then there exist $v_1$, $\ldots\,$,~$v_m \in \cA(S)$ and a permutation $\sigma \in \fS_m$ such that $b \simeq v_1\cdots v_m$ and $u_i \simeq v_{\sigma(i)}$ for all $i \in [1,m]$.
    \item\label{exwt:wt} The canonical homomorphism $\pi\colon S \to \rab S$ is a weak transfer homomorphism.
    \item\label{exwt:wtaai} The canonical homomorphism $\pi\colon S \to \rab S$ is an isoatomic weak transfer homomorphism.
  \end{enumerate}
  Moreover, each of \labelcref*{exwt:explicit,exwt:wt,exwt:wtaai} imply the equivalent conditions given in  \cref{cong0}.
\end{proposition}

\begin{proof}
  We first show that \ref*{exwt:explicit} implies \subref{cong0:transitive}.
  Let $a$,~$b$,~$c \in S$ be such that $a \equiv_p b$ and $b \equiv_p c$.
  By the definition of $\equiv_p$ there exist $m \in \bN_0$, $u_1$,~$\ldots\,$,~$u_m$, $v_1$,~$\ldots\,$,~$v_m \in \cA(S)$ and a permutation $\sigma \in \fS_m$ such that $a \simeq u_1\cdots u_m$, $b \simeq v_1\cdots v_m$ and $u_i \simeq v_{\sigma(i)}$ for all $i \in [1,m]$.
  By \ref*{exwt:explicit} and since $b \equiv_p c$, there exist $w_1$,~$\ldots\,$,~$w_m \in \cA(S)$ and a permutation $\tau \in \fS_m$ such that $c \simeq w_1\cdots w_m$ and $v_i \simeq w_{\tau(i)}$ for all $i \in [1,m]$.
  Then $u_i \simeq v_{\sigma(i)} \simeq w_{\tau(\sigma(i))}$ for all $i \in [1,m]$, and hence $a \equiv_p c$.

  \ref*{exwt:explicit}${}\Rightarrow{}$\ref*{exwt:wt}:
  Property \labelcref{th:units} of \subref{th:wth} holds since $\pi$ is surjective and by applying \subref{ab-prop:units}.
  It remains to verify property \labelcref{wth:lift} of a weak transfer homomorphism.
  Let $a \in S$, let $m \in \bN$, and let $v_1$,~$\ldots\,$,~$v_m \in \cA(\rab S)$ such that $\pi(a) = v_1\cdots v_m$.
By the surjectivity of $\pi$, there exist $u_1'$,~$\ldots\,$,~$u_m' \in S$ such that $\pi(u_i') = v_i$ for all $i \in [1,m]$, and by \subref{ab-prop:atoms}  $u_i' \in \cA(S)$ for all $i \in [1,m]$.
  We thus have $\pi(a) = \pi(u_1'\cdots u_m')$, whence $a \rabrel u_1'\cdots u_m'$.
  Since we have already established that \ref*{exwt:explicit} implies the equivalent conditions of \cref{cong0}, $\rab \equiv = \equiv_p$.
  Hence $a \equiv_p u_1'\cdots u_m'$, and thus there exist $u_1$,~$\ldots\,$,~$u_m \in \cA(S)$ and a permutation $\sigma \in \fS_m$ such that $a \simeq u_1\cdots u_m$ with $u_i \simeq u_{\sigma(i)}'$ for all $i \in [1,m]$.
  Then $\pi(a)=\pi(u_1)\cdots \pi(u_m)$ and $\pi(u_i) = \pi(u'_{\sigma(i)}) = v_{\sigma(i)}$.

  \ref*{exwt:wt}${}\Rightarrow{}$\ref*{exwt:wtaai}:
  By \subref{ab-prop:assoc}, $\pi\colon S \to \rab S$ is isoatomic.

  \ref*{exwt:wtaai}${}\Rightarrow{}$\ref*{exwt:explicit}:
  Let $a$, $b \in S$ with $a \equiv_p b$, $m \in \bN_0$, and $u_1$,~$\ldots\,$,~$u_m \in \cA(S)$ with $a \simeq u_1\cdots u_m$.
  By the definition of $\equiv_p$ there exist $n \in \bN_0$ and $u_1'$,~$\ldots\,$,~$u_n'$, $v_1'$,~$\ldots\,$,~$v_n' \in \cA(S)$ as well as a permutation $\tau \in \fS_n$ such that $a \simeq u_1'\cdots u_n'$, $b \simeq v_1'\cdots v_n'$ and $u_i' \simeq v_{\tau(i)}'$ for all $i \in [1,n]$.
Since $a \equiv_p b$ implies $a \rabrel b$, $\pi(u_1)\cdots \pi(u_m) = \pi(a) = \pi(b)$, and, by \subref{ab-prop:atoms}, $\pi(u_i) \in \cA(\rab S)$ for all $i \in [1,m]$.
  Since $\pi$ is a weak transfer homomorphism, there exist $v_1,\ldots,v_m \in \cA(S)$ and a permutation $\sigma \in \fS_m$ such that $b \simeq v_1\cdots v_m$ and $\pi(v_{\sigma(i)}) \simeq \pi(u_i)$ for all $i \in [1,m]$.
By \subref{ab-prop:assoc}, $v_{\sigma(i)} \simeq u_i$ for each $i \in [1,m]$.
\end{proof}

We now illustrate that the equivalent statements of \cref{exwt} can fail to hold for a semigroup $S$ even if there is a transfer homomorphism from $S$ to a commutative reduced cancellative semigroup $T$.

\begin{example}\label{length}
  Let $S = \langle a,b,c,d \mid ab = cd \rangle$. Clearly $S$ is reduced and is an Adyan semigroup, whence $S$ is cancellative.
  Then $\ab S$ is the free abelian monoid $\cF(\{\alpha,\beta,\gamma,\delta\})$ modulo the congruence relation generated by $\alpha \beta = \gamma \delta$, with the canonical homomorphism $\pi\colon S \to \ab S$ being defined by $\pi(a) = \alpha$, $\pi(b)=\beta$, $\pi(c)=\gamma$, $\pi(d)=\delta$.
  We have $\pi(ab) = \alpha\beta = \gamma\delta = \delta\gamma = \pi(dc)$, but the two possible permutable factorizations of $ab$ in $S$ are $\pf{a,b}$ and $\pf{c,d}$, while $dc$ only has the factorization $\pf{d,c}$.
  Therefore the factorization $\pf{\alpha,\beta}$ of $\pi(dc)$ does not lift, and thus $\pi$ is not a weak transfer homomorphism.

  However, from the relation imposed, it is clear that there exists a length function $\ell\colon S \to \bN_0$ mapping each of $a$,~$b$, $c$ and $d$ to $1$.
  The map $\ell$ is a transfer homomorphism from $S$ to the commutative semigroup $(\mathbb N_0,+)$.
  Thus a noncommutative semigroup may possess a (weak) transfer homomorphism to a commutative semigroup even when the canonical map to the reduced abelianization is not a (weak) transfer homomorphism.
\end{example}

The following proposition shows that, if $S$ is a cancellative semigroup with $\rab S$ cancellative, the existence of an isoatomic weak transfer homomorphism from $S$ to a commutative cancellative semigroup however does imply that the canonical homomorphism $S \to \rab S$ is a weak transfer homomorphism, and thus that the equivalent conditions of \cref{exwt} are satisfied.
We note that the transfer homomorphism $\ell$ from \cref{length} is not isoatomic.

\begin{proposition}\label{assocwth}
  Let $S$ be a cancellative semigroup.
  Assume that $T$ is a commutative atomic cancellative semigroup and that there exists an isoatomic weak transfer homomorphism $\phi\colon S \rightarrow T$.
  Then $\rab S$ is cancellative, the canonical homomorphism $\pi\colon S \to \rab S$ is an isoatomic weak transfer homomorphism, and $\phi$ induces an isomorphism $\rab S \cong \red T$.
\end{proposition}

\begin{proof}
  It suffices to show that $\phi$ induces an isomorphism $\rab\phi\colon\rab S \isomto \red T$.
  Then $\rab S$ is cancellative, and since $\phi$ is an isoatomic weak transfer homomorphism and $\rab\phi$ is an isomorphism, $\pi = \rab\phi^{-1} \circ\phi$ is an isoatomic weak transfer homomorphism.

  We may, without loss of generality, assume that $T$ is reduced.
  Since $T$ is commutative, $\phi$ factors through $\pi$, that is, there exists $\rab\phi\colon \rab S \to T$ such that $\rab\phi \circ \pi = \phi$.
  Since $T=T^\times\phi(S)T^\times=\phi(S)$, the induced map $\rab \phi$ is surjective.

  It remains to show that $\rab \phi$ is injective.
  Let $\lbar a$, $\lbar b \in \rab S$ be such that $\rab\phi(\lbar a) = \rab\phi(\lbar b)$, and let $a$,~$b \in S$ be such that $\pi(a) = \lbar a$ and $\pi(b) = \lbar b$.
  We may assume $\lbar a \ne 1$, and hence $a \not \in S^\times$.
  We have $\phi(a) = \phi(b)$ and, by \labelcref{th:units}, $\phi(a) \ne 1$.
  Thus there exist $m \in \bN$ and atoms $w_1$, $\ldots\,$,~$w_m \in \cA(T)$ such that $\phi(a) = w_1\cdots w_m$.
  By \labelcref{wth:lift}, there exist $u_1$, $\ldots\,$,~$u_m \in \cA(S)$, $v_1$, $\ldots\,$,~$v_m \in \cA(S)$ and permutations $\sigma$,~$\tau \in \fS_m$ such that $a = u_1\cdots u_m$, $b = v_1\cdots v_m$ and $w_i = \phi(u_{\sigma(i)}) = \phi(v_{\tau(i)})$ for all $i \in [1,m]$.
  Since $\phi$ is isoatomic, $u_{\sigma(i)} \simeq v_{\tau(i)}$ for all $i \in [1,m]$.
  Since $\rab S$ is commutative and reduced, we have
  \[
    \pi(a) = \pi(u_1)\cdots\pi(u_m) = \pi(u_{\sigma(1)}) \cdots \pi(u_{\sigma(m)}) = \pi(v_{\tau(1)}) \cdots \pi(v_{\tau(m)}) = \pi(v_1)\cdots \pi(v_m ) = \pi(b),
  \]
  that is, $\lbar a = \lbar b$.
\end{proof}

\begin{remark}\label{representatives}
  Let $S$ be an atomic cancellative semigroup and let $R$ be a set of representatives for the associativity classes of $\cA(S)$.
  Suppose that the equivalent conditions of \cref{exwt} are satisfied.
  In this case we can give a construction of $\rab S$ in terms of the free abelian monoid $\cF(R)$.
  We define a relation $\equiv_S$ on $\cF(R)$ as follows:
  We set $1 \equiv_S 1$ and, for $a$, $b \in \cF(R) \setminus \{1\}$, we set $a \equiv_S b$ if and only if there exist $k$,~$l \in \bN$ and $u_1$,~$\ldots\,$,~$u_k$, $v_1$,~$\ldots\,$,~$v_l \in R$ such that $a=u_1 \cdots u_k$, $b=v_1\cdots v_l$ and, for all $i \in [1,k]$ and $j \in [1,l]$, there exist associated atoms $u_i' \simeq u_i$ and $v_j' \simeq v_j$ in $S$, and permutations $\sigma \in \fS_k$ and $\tau \in \fS_l$ such that $u_{\sigma(1)}'\cdots u_{\sigma(k)}' \simeq v_{\tau(1)}'\cdots v_{\tau(l)}'$ in $S$.
  Since we are assuming that the canonical homomorphism to $\rab S$ is a weak transfer homomorphism, it is easy to check that $\equiv_S$ is transitive.
  Trivially, $\equiv_S$ is reflexive and symmetric and thus also a congruence relation.
  For any $x \in \cF(R)$, we write $[x]$ for its image in $\cF(R) / \!\equiv_S$.

  Let $a \in S$ and let $k$,~$l \in \bN_0$ and $u_1'$,~$\ldots\,$,~$u_k'$, $v_1'$,~$\ldots\,$,~$v_l' \in \cA(S)$ be such that $a\simeq u_1'\cdots u_k' \simeq v_1'\cdots v_l'$, and, for all $i \in [1,k]$ and $j \in [1,l]$, let $u_i \in R$ and $v_j \in R$ be such that $u_i \simeq u_i'$ and $v_i \simeq v_j'$.
  From the definition of $\equiv_S$ it is then clear that $u_1\cdots u_k \equiv_S v_1\cdots v_l$, and thus we can define a map
  \[
    \pi\colon S \to \cF(R)/\!\equiv_S,\  a \mapsto [u_1\cdots u_k]  \quad\text{(with units mapping to $[1]$).}
  \]
  It is now straightforward to check that $\pi$ is a homomorphism and that $(\cF(R)/\!\equiv_S, \pi)$ satisfies the universal property of the reduced abelianization.
  Therefore $\cF(R) / \!\equiv_S\; \cong \rab S$.
\end{remark}

The following example illustrates that not every atomic cancellative semigroup possesses a weak transfer homomorphism into a commutative semigroup.

\begin{example}
  Let $S = \langle a,b,c,d,e \mid abc = de \rangle$.
  Then $\sL(abc) = \{ 2,3 \}$, while $\sL(bac) = \{ 3 \}$.
  Hence $S$ does not admit a weak transfer homomorphism into any commutative semigroup.
\end{example}

The next proposition shows that if $\rab S$ is cancellative, then every weak transfer homomorphism to a commutative reduced cancellative semigroup factors through the canonical homomorphism $\pi\colon S \to \rab S$.
If, moreover, $\pi$ is a weak transfer homomorphism, then we obtain as an immediate corollary a universal property that characterizes the weak transfer homomorphism $\pi$.

\begin{proposition}\label{wtup}
  Let $S$ be a cancellative semigroup with $\rab S$ cancellative and let $\pi\colon S \to \rab S$ denote the canonical homomorphism.
  Suppose that there exists a weak transfer homomorphism $\phi\colon S \to T$ to a commutative atomic cancellative semigroup $T$.
  Let $\red \pi\colon T \to \red T$ denote the canonical homomorphism.
  Then there exists a unique transfer homomorphism $\rab\phi\colon \rab S\to \red T$ such that $\red \pi \circ \phi = \rab\phi \circ \pi$.

  In particular, if $\pi$ is a weak transfer homomorphism, it satisfies the following universal property:
  If $\phi\colon S \rightarrow T$ is a weak transfer homomorphism from $S$ to a commutative reduced cancellative semigroup $T$, then there exists a unique transfer homomorphism $\rab\phi \colon \rab S \rightarrow T$ such that $\rab\phi \circ \pi=\phi$, that is, the following diagram commutes.
  \[
  \xymatrix{
      S \ar@{->>}^\phi[r] \ar@{->>}_\pi[d] & T \\
      \rab S \ar@{-->>}_{\exists !\, \rab\phi }[ur]
    }
  \]
\end{proposition}

\begin{proof}
  Replacing $\phi$ by $\red\pi \circ \phi$ if necessary, we may without loss of generality assume that $T$ is reduced.
  Since $T$ is commutative and reduced, $\phi$ factors through $\pi \colon S \to \rab S$, that is, there exists a homomorphism $\rab\phi\colon \rab S \to T$ such that $\phi = \rab\phi \circ \pi$.
  Clearly, $\rab\phi$ is uniquely determined by this relation, and it remains to show that it is a transfer homomorphism.
  Since $T$ is atomic, and $\rab S$ is commutative, it suffices to show that $\rab\phi$ is a weak transfer homomorphism.

  Since $\phi$ is a weak transfer homomorphism, we have $T=\phi(S)T^\times=\phi(S)$ and $\phi^{-1}(T^\times)=S^\times$, with $T^\times=\{1\}$.
  The first property immediately implies $T = \rab\phi(\rab S)$.
  We now show $\rab\phi^{-1}(\{1\}) = \{1\}$.
  Trivially $\rab\phi(1) = 1$.
  For the other inclusion, suppose that $\lbar a \in \rab S$ is such that $\rab\phi(\lbar a) = 1$.
  There exists $a \in S$ such that $\pi(a) = \lbar a$, and hence $\phi(a)=\rab \phi(\lbar a) =1$.
  Thus $a \in S^\times$ and therefore $\lbar a= \pi(a) = 1$ in $\rab S$.
  Hence $\rab\phi$ satisfies \labelcref{th:units}.

  We now check \labelcref{wth:lift}.
  Let $\lbar a \in \rab S$ and suppose $\rab\phi (\lbar a)=w_1\cdots w_m$ with $m \in \bN$, $w_1$, $\ldots\,$,~$w_m \in \cA(\red T)$.
  Let $a \in \pi^{-1}(\{\lbar a\})$.
  Since $\phi$ is a weak transfer homomorphism, there exist atoms $u_1$, $\ldots\,$,~$u_m \in \cA(S)$ and a permutation $\sigma \in \fS_m$ such that $a=u_1\cdots u_m$ and $\phi(u_i)=w_{\sigma(i)}$ for each $i \in [1,m]$.
  Now $\pi(u_i) \in \cA(\rab S)$ and $\rab\phi \circ\pi(u_i)=\phi(u_i)=w_{\sigma(i)}$ for each $i \in [1,m]$.
  Therefore $\rab\phi $ is a weak transfer homomorphism.
\end{proof}

\subsection*{Examples}

We now highlight some examples in which the canonical homomorphism to the reduced abelianization is a weak transfer homomorphism.

\subsubsection*{Rings of Triangular Matrices}\label{triangularmatrices}
  Let $D$ be a commutative atomic domain, let $n \in \bN$, and let $R = T_n(D)$ denote the ring of all $n \times n$ upper triangular matrices with entries in $D$.
  We study $S=T_n(D)^\bullet$, the multiplicative subsemigroup of non zero-divisors of $R$ consisting of those upper triangular matrices having nonzero determinant.
  Sets of lengths in this semigroup were studied extensively in \cite{Bachman-Baeth-Gossell} where the homomorphism
  \[
    \delta\colon
    \begin{cases}
      T_n(D)^\bullet         &\to     (\red{D^\bullet})^n \\
      [a_{i,j}]_{i,j \in [1,n]} &\mapsto (a_{i,i}D^\times)_{i \in [1,n]},
    \end{cases}
  \]
  mapping a matrix to the vector of associativity classes of its diagonal entries, was shown to be a weak transfer homomorphism.

  We now give a lemma which illustrates that $\delta$ is, in fact, isoatomic. Moreover we show that associativity, similarity and subsimilarity coincide for atoms of $S$.

\begin{lemma}\label{triangular}
Let $D$ be a commutative atomic domain and let $\overline{\cA(D)}$ denote a set of representatives for the associativity classes of atoms of $D$.
Let $n \in \bN$ and $R = T_n(D)$.
\begin{enumerate}
  \item\label{triangular:atoms}
    An element $A=[a_{i,j}] \in T_n(D)^\bullet$ is an atom if and only if there exists $m \in [1,n]$ such that $a_{i,i}\in D^{\times}$ for all $i \in [1,n]$ with $i\not=m$ and $a_{m,m} \in \cA(D)$.

  \item\label{triangular:repr}
    If $A$ is an atom of $T_n(D)^\bullet$ with $\det(A)$ associated to $a \in \cA(D)$, then $A$ is associated to the matrix $B=[b_{i,j}] \in T_n(D)^\bullet$ where
  \[
  b_{i,j} =
  \begin{cases}
    1             \quad &\text{if $i=j$ and $(i,j)\not=(m,m)$,} \\
    \overline{a}  \quad &\text{if $i=j=m$,} \\
    0             \quad &\text{if $i\not= j$,}
  \end{cases}
  \]
  and $\overline{a}$ is the representative of $a$ in $\overline{\cA(D)}$.

  \item\label{triangular:ann}
    If $A$ is an atom of $T_n(D)^\bullet$, and $m \in [1,n]$ is such that the $m$-th diagonal entry of $A$ is $a \in \cA(D)$, then
    \begin{equation}\label{eq:ann}
      \ann_R(R/RA) = \{\, [b_{i,j}] \in T_n(D) : b_{i,j} \in aD \text{ for all $i$,~$j \in [1,m]$} \,\}.
    \end{equation}
  \item\label{triangular:sim}
    For $A$, $B \in \cA(T_n(D)^\bullet)$, the following statements are equivalent.
    \begin{enumerateequiv}
    \item\label{simtri:ass} $A$ is associated to $B$.
    \item\label{simtri:sim} $A$ is similar to $B$.
    \item\label{simtri:sub} $A$ is subsimilar to $B$.
    \item\label{simtri:ann} $\ann_R(R/RA) = \ann_R(R/RB)$.
    \end{enumerateequiv}
    In particular, $\sd_p = \sd_\dsubsim = \sd_\dsim$ on $T_n(D)^\bullet$.
\end{enumerate}
\end{lemma}

\begin{proof}
Let $S=T_n(D)^\bullet$.
The claim in \ref*{triangular:atoms} follows from the fact that $\delta$ is a weak transfer homomorphism.

\ref*{triangular:repr} We denote by $I_n$ the $n \times n$ identity matrix and, for all $i$,~$j \in [1,n]$, by $E_{i,j}$ the $n \times n$ matrix with $1$ in the $(i,j)$ position, and $0$ in all other positions.
By \ref*{triangular:atoms} there exists an $m \in [1,n]$ such that $a_{i,i}$ is a unit of $D$ for all $i \in [1,n]$ with $i \not=m$, while $a_{m,m} \in \cA(S)$.
Let $\lbar a$ be the element of $\cA(S)$ that is associated to $a_{m,m}$.
Consider the diagonal matrix
\[
  U = \overline{a}a_{m,m}^{-1} E_{m,m} + \sum_{\substack{i=1\\i \ne m}}^n a_{i,i}^{-1} E_{i,i} \,\in\, T_n(D)^\bullet.
\]
Then, $A'=UA=[a'_{i,j}] \in T_n(D)^\bullet$, with $a'_{m,m} = \lbar a$, and all other diagonal entries equal to $1$.
Since $A$ is associated to $UA$ and associativity is transitive, we may assume for the remainder of this proof that $A=A'$, that is, all but one of the diagonal entries of $A$ are $1$ and that the non-unit diagonal entry $a_{m,m}$ is already in the pre chosen set $\overline{\cA(D)}$.

We now define a sequence of associates of $A$ inductively, successively eliminating rows and columns of $A$.
Set $A^{(0)} = A$.
For all $i \in [1,m-1]$, assuming that $A^{(i-1)}=[a_{k,l}^{(i-1)}]$, let
\[
  C_i = I_n - \sum_{j = i+1}^n a_{i,j}^{(i-1)} E_{i,j}.
\]
Clearly $C_i \in T_n(D)^\times$, and we set $A^{(i)} = A^{(i-1)}C_i$, that is, $A^{(i)}$ is obtained from $A^{(i-1)}$ by eliminating the $i$-th row (except for the diagonal entry).
Setting now $A^{(m)} = A^{(m-1)}$, we inductively define, for all $j \in [m+1,n]$, a matrix $A^{(j)} \in S$:
Assuming that $A^{(j-1)}=[a_{k,l}^{(j-1)}]$, let
\[
  C_j = I_n - \sum_{i = 1}^{j-1} a_{i,j}^{(j-1)} E_{i,j}.
\]
Again $C_j \in T_n(D)^\times$, and we set $A^{(j)} = C_jA^{(j-1)}$, that is, $A^{(j)}$ is obtained from $A^{(j-1)}$ by eliminating the $j$-th column (except for the diagonal entry).
The final matrix $A^{(n)}=[b_{k,l}]$ is therefore diagonal, with $b_{m,m} = a_{m,m}$, and $b_{i,i} =1$ for all $i \in [1,n] \setminus \{m\}$.
Therefore $A$ is associated to a diagonal matrix as desired.

\ref*{triangular:ann}
We first recall a description of the ideals of $T_n(D)$.
For all $i \in [1,n]$ and $j \in [i,n]$ let $I_{i,j}$ be an ideal of $D$, and suppose that $I_{i,j} \subset I_{i-1,j}$ for all $j \in [1,n]$ and $i \in [2,j]$, and that $I_{i,j} \subset I_{i,j+1}$ for all $j \in [1,n-1]$ and $i \in [1,j]$.
An elementary calculation shows that
\[
  I = \big\{\, [a_{i,j}] \in T_n(D) : a_{i,j} \in I_{i,j} \text{ for all $i \in [1,n]$, $j \in [i,n]$} \,\big\}
\]
is an ideal of $R$, and it is easy to check that in fact every ideal of $R$ is of this form.

Let $A \in \cA(S)$.
We will show that \cref{eq:ann} holds.
Since $R/RA \cong R/RA'$ for all $A' \in R$ with $A' \simeq A$, we may assume without restriction (by \ref*{triangular:repr}), that all off-diagonal entries of $A$ are zero, and that there exist $m \in [1,n]$ and $a \in \cA(D)$ such that the $m$-th diagonal entry of $A$ is equal to $a$, while all other diagonal entries are equal to $1$.
It is then easy to check that $RA$ consists of all those $n \times n$ upper triangular matrices for which the entries of the $m$-th column are contained in $aD$.
Therefore all elements of the right-hand side of \cref{eq:ann} annihilate $R/RA$.
Recall that $\ann_R(R/RA)$ is a two-sided ideal of $R$ necessarily contained in $RA$.
By our description of ideals of $R$, the set on the right hand side of \cref{eq:ann} is an ideal of $R$, and moreover the maximal two-sided ideal of $R$ contained in $RA$.
Thus it must be the annihilator of $R/RA$ and we have established the claim.

\ref*{triangular:sim}
  The implications \ref*{simtri:ass}${}\Rightarrow{}$\ref*{simtri:sim}${}\Rightarrow{}$\ref*{simtri:sub}${}\Rightarrow{}$\ref*{simtri:ann} are immediate from the definitions.

\ref*{simtri:ann}${}\Rightarrow{}$\ref*{simtri:ass}:
Let $A$,~$B \in \cA(S)$ be such that $\ann_R(R/RA) = \ann_R(R/RB)$.
There exist $k$,~$l \in [1,n]$ and $a$,~$b \in \cA(D)$ such that the $k$-th diagonal entry of $A$ is $u$, the $l$-th diagonal entry of $B$ is $v$ and all other diagonal entries of $A$ and $B$ are units.
To show $A \simeq B$, we need to show that $k=l$ and $a \simeq b$ in $D$ (using \ref*{triangular:repr}).
From the description of annihilator ideals in \ref*{triangular:ann}, we must have $k=l$. Comparing the entries in the upper left corner, $aD=bD$, that is, $a\simeq b$.
\end{proof}

Since $\delta$ is isoatomic, it follows from \cref{assocwth} and \subref{triangular:repr} that $\rab S \cong (\red{D^\bullet})^n$.

\begin{proposition}\label{cor:tri-fact}
  Let $D$ be a commutative atomic domain, $n \in \bN$, and $S = T_n(D)^\bullet$.
  Then
  \[
    \sc_p(S) = \sc_\dsim(S) = \sc_\dsubsim(S) = \sc_p(D^\bullet)
  \]
  and
  \[
    \sc_{p,\cmon}(S) = \sc_{\dsim,\cmon}(S) = \sc_{\dsubsim,\cmon}(S) = \sc_{p,\cmon}(D^\bullet).
  \]
  In particular, $T_n(D)^\bullet$ is permutably ($\sd_\dsim$-, $\sd_\dsubsim$-) factorial if only if $D$ is factorial.
  Moreover, $\st_p(S) = \st(D^\bullet)$ and $\omega_p(S) = \omega_p(D^\bullet)$.
\end{proposition}

\begin{proof}
  From \cref{strong-weak-distance} it follows that $\sc_p(T_n(D)^\bullet) = \sc_p({(\red{D^\bullet})}^n)$, and \cite[Proposition 1.6.8.1]{GHK06} implies $\sc_p({(\red{D^\bullet})}^n) = \sc_p(\red{D^\bullet})$, the latter of which is trivially equal to $\sc_p(D^\bullet)$.
  Since \subref{triangular:sim} implies that $\sd_p = \sd_\dsim = \sd_\dsubsim$ on $T_n(D)^\bullet$, the remaining equalities for the catenary degrees follow.
  The claims for the monotone catenary degrees are shown in the same way.
  The equality $\st_p(S) = \st(D^\bullet)$ follows from \cref{awth-tame} together with \cite[Proposition 1.6.8.4]{GHK06}, and $\omega_p(S) = \omega(D^\bullet)$ follows similarly.

  Since $D$ is factorial if and only if $\sc_p(D^\bullet) = 0$ and $S$ is permutably factorial if and only if $\sc_p(S) = 0$, the final statement follows.
\end{proof}

\subsubsection*{Matrix Rings}

We now provide an even more straightforward example which illustrates how simple the abelianization of a somewhat complicated-looking noncommutative semigroup $S$ can be.

Let $D$ be a commutative principal ideal domain, let $n \in \bN$, and let $S=M_n(D)^\bullet$ denote the multiplicative semigroup of all $n \times n$ matrices with entries in $D$ having nonzero determinant.
Factorization invariants of this semigroup were studied in \cite{Baeth-Ponomarenko}.
Every $A \in S$ can be put into Smith Normal Form, that is, there exist $U$,~$V \in M_n(D)^\times$ and a diagonal matrix $C=[c_{i,j}] \in M_n(D)^\bullet$ with $c_{i+1,i+1}$ dividing $c_{i,i}$ for all $i \in [1,n-1]$ and such that $A = UCV$.
From this it follows immediately that $\det\colon M_n(D)^\bullet \rightarrow D^\bullet$ is a transfer homomorphism (cf. \cite[Lemma 2.2]{Baeth-Ponomarenko}).
Therefore $A$ is an atom if and only if $c_{1,1} \in \cA(D)$ and $c_{i,i} \in D^\times$ for all $i \in [2,n]$.
Hence the transfer homomorphism is isoatomic, and thus $\rab S \cong \red{D^\bullet}$ by \cref{assocwth}.

As in $T_n(D)$, in $R=M_n(D)$ the notions of associativity, similarity and subsimilarity coincide (if $A \in S$ is an atom with Smith Normal Form $C$ as above, then $\ann_{R}(R/RA)=\ann_R(R/RC)= Rc_{1,1}$, implying as for the ring of $n \times n$ upper triangular matrices, that two atoms with the same annihilator are associated), and thus results analogous to \cref{cor:tri-fact} hold.

\subsubsection*{Almost Commutative Semigroups}

Recall that if $S$ is a cancellative semigroup on which associativity is a congruence relation and for which $\red S$ is cancellative, the canonical homomorphism $\pi\colon S \to \red S$ is always an isoatomic transfer homomorphism.
However, this is only useful to us if $\red S$ is a semigroup that we understand better than $S$ itself.
In this example we consider the case where $\red S$ is commutative.
We say that a semigroup $S$ is \emph{almost commutative} if $ab \simeq ba$ for all $a$,~$b \in S$.

\begin{proposition}
  Let $S$ be a semigroup on which $\simeq$ is a congruence relation.
  \begin{enumerate}
    \item\label{ac:equiv} The following statements are equivalent:
          \begin{enumerateequiv}
          \item\label{aachar:sim} $S$ is almost commutative.
          \item\label{aachar:red} $\red S$ is commutative.
          \end{enumerateequiv}
    \item\label{ac:th} Suppose $S$ is almost commutative, and that $S$ and $\red S$ are cancellative.
        Then $\red S = \rab S$, and the canonical homomorphism $\pi\colon S \to \red S$ is an isoatomic transfer homomorphism to a commutative reduced cancellative semigroup.
  \end{enumerate}
\end{proposition}

\begin{proof}
  \ref*{ac:equiv}
  \ref*{aachar:sim}${}\Rightarrow{}$\ref*{aachar:red}:
  Let $\lbar a$,~$\lbar b \in \red S$ and let $a$,~$b \in S$ be such that $\pi(a) = \lbar a$ and $\pi(b) = \lbar b$.
  Then $\lbar a \lbar b = \pi(ab) = \pi(ba) = \lbar b \lbar a$, where the middle equality holds since $ab \simeq ba$.

  \ref*{aachar:red}${}\Rightarrow{}$\ref*{aachar:sim}:
  Let $a$,~$b \in S$. Then $\pi(ab) = \pi(a)\pi(b) = \pi(b)\pi(a) = \pi(ba)$, and hence $ab \simeq ba$.

  \smallskip
  \ref*{ac:th} The homomorphism $\pi$ is always an isoatomic transfer homomorphism, and hence it suffices to show that $\red S = \rab S$.
  By \ref*{ac:equiv}, $\red S$ is commutative, whence $\red S = \ab{(\red S)}$, but we have identified $\ab{(\red S)} = \rab S$.
\end{proof}

Let $S$ be a normalizing Krull monoid.
Then $\red S$ embeds into a free abelian monoid (as a consequence of \cite[Lemma 4.6 and Theorem 4.13]{Geroldinger}), and hence is itself commutative, and in fact a commutative Krull monoid.
Thus normalizing Krull monoids are almost commutative, with their associated reduced monoids being commutative Krull monoids that have been well-investigated.
We can expect factorization theoretic results about $\red S$ to immediately carry over to $S$.
For instance, we have the following.

\begin{proposition}\label{prop:krull-finite-omega-invariant}
Let $S$ be a normalizing Krull monoid.
Then, for all $a \in S$, $\omega_p(S,a)<\infty$ and $\omega'_p(S,a)<\infty$.
\end{proposition}

\begin{proof}
  Since $\red S$ is a commutative Krull monoid, we have $\omega'_p(\red S,a) = \omega_p(\red S, a) < \infty$ for all $a \in \red S$ by \cite[Theorem 4.2]{GH08}.
  The canonical homomorphism $\pi\colon S \to \red S=\rab S$ from $S$ to its reduced abelianization is an isoatomic transfer homomorphism, and thus the same holds true for $S$ by \cref{awth-tame}.
\end{proof}

We conclude this section by giving a construction of another family of semigroups which are almost commutative.

\begin{example}
Let $T$ be a commutative reduced cancellative semigroup and let $G$ be a group.
Further suppose that there is an action
\[
  \cdot: T \times G \rightarrow G
\]
of $T$ on $G$ such that the following two additional conditions hold:
\begin{enumerate}
\item\label{morphism} for each $t \in T$ and all pairs $g_1$, $g_2 \in G$,\; $t\cdot(g_1g_2)=(t\cdot g_1)(t\cdot g_2)$, and
\item\label{injective} for each $t \in T$ and all pairs $g_1$, $g_2 \in G$,\; $t\cdot g_1=t \cdot g_2$ implies $g_1=g_2$.
\end{enumerate}
(Equivalently, there is a homomorphism $T \to \Mono(G)$ given by $t \mapsto (g \mapsto t\cdot g)$, with $\Mono(G)$ denoting the semigroup of monomorphisms of $G$.)
The \emph{semidirect product} $S=G \rtimes T$ is the semigroup defined on the cartesian product $G \times T$ by the operation
\[
  (g_1,t_1)(g_2,t_2) = (g_1 (t_1 \cdot g_2), t_1 t_2).
\]
The identity element of this semigroup is $(1_G,1_T)$ and, since $T$ is reduced, $S^\times = \{\, (g, 1_T) \in S : g \in G \,\}$.

\noindent \emph{$S$ is cancellative:}
Suppose that $g_1$,~$g_2$,~$g_3 \in G$ and $t_1$,~$t_2$, $t_3 \in T$ are such that
\[
  (g_1,t_1)(g_2,t_2)=(g_1,t_1)(g_3,t_3).
\]
Then $t_1t_2=t_1t_3$, and due to the cancellativity of $T$, $t_2=t_3$.
Moreover $g_1 (t_1 \cdot g_2) = g_1 (t_1 \cdot g_3)$ implies $t_1\cdot g_2 = t_1\cdot g_3$ and hence, by \ref*{injective}, $g_2=g_3$.

Suppose that $g_1$,~$g_2$,~$g_3 \in G$ and $t_1$,~$t_2$, $t_3 \in T$ are such that
\[
  (g_2,t_2)(g_1,t_1)=(g_3,t_3)(g_1,t_1).
\]
Cancellativity of $T$ again implies $t_2=t_3$.
Then $g_2 (t_2 \cdot g_1) = g_3 (t_3 \cdot g_1) = g_3 (t_2 \cdot g_1)$, and hence $g_2=g_3$.

\noindent \emph{$S$ is normalizing:}
 Let $(g,t) \in S$.
 Let $(g_1,t_1) \in (g,t)S$, with $(g_2,t_2) \in S$ such that $(g_1,t_1) = (g,t)(g_2,t_2)$.
 Then
 \[
   (g(t\cdot g_2)(t_2\cdot g)^{-1},t_2)\, (g,t) = (g(t\cdot g_2)(t_2\cdot g)^{-1}(t_2\cdot g),t_2t) = (g (t \cdot g_2), t t_2) = (g,t)(g_2,t_2) = (g_1,t_1),
 \]
 showing that also $(g_1,t_1) \in S(g,t)$.
 (We used the commutativity of $T$ to conclude $t_2t=t t_2$.)
 A symmetrical argument gives the reverse containment and hence $S$ is normalizing.

\noindent \emph{$S$ is almost commutative:}
  Since $S$ is normalizing, associativity is a congruence relation on $S$.
  We claim $\red S \cong T$, and indeed, this follows because, for all $g \in G$ and $t \in T$,
  \[
    S^\times (g,t) S^\times = S^\times (g,t) = \{\, (g',t) : g' \in G \,\},
  \]
  where the first equality holds because $S$ is normalizing.
  Since $\red S \cong T$ is commutative, $S$ is almost commutative.
\end{example}

\section{Maximal orders}\label{sec:maxord}

In this final section we study arithmetical maximal orders and right-saturated subcategories of the integral elements of an arithmetical groupoid.
We refer the reader to the corresponding subsection of \cref{sec:preliminaries} on page~\pageref{arith-max-ord} for the definition of arithmetical maximal orders (\cref{amo}) and a discussion as to how they relate to Krull monoids and classical maximal orders in central simple algebras over global fields.

We show that certain conditions imply the existence of a transfer homomorphism from an arithmetical maximal order to a monoid of zero-sum sequences and that this transfer homomorphism has catenary degree in the permutable fibers at most $2$.
Thus we obtain \cref{thm:cat-transfer,cor:cat-transfer-semigroup} which generalize the corresponding result for commutative Krull monoids.
We then apply these results to classical maximal orders $R$ in central simple algebras over global fields satisfying the condition that every stable free left $R$-ideal is free.
Thereby we obtain \cref{thm:cat-transfer-csa} which gives a description of permutable, $\sd_\dsim$- and $\sd_\dsubsim$-catenary degrees of $R^\bullet$ in terms of the well-studied catenary degree of a monoid of zero-sum sequences over a finite abelian group.
Finally we note several immediate corollaries to this theorem.

We start by recalling the notion of an arithmetical groupoid, which is useful in describing the divisorial one-sided ideal theory of an arithmetical maximal order (see \cite[Section 4]{Smertnig} for more details and proofs).

\begin{definition}\label{lg}
  A lattice-ordered groupoid $(G, \le)$ is a groupoid $G$ together with a relation $\le$ on $G$ such that for all $e, f \in G_0$
  \begin{enumerate}
    \item\label{lg:ll} $(G(e,\cdot),\, \le \mid_{G(e,\cdot)})$ is a lattice (we write $\meet'_e$ and $\join'_e$ for the meet and join),
    \item\label{lg:rl} $(G(\cdot,f),\, \le \mid_{G(\cdot,f)})$ is a lattice (we write $\meet''_f$ and $\join''_f$ for the meet and join),
    \item\label{lg:il} $(G(e,f), \, \le \mid_{G(e,f)})$ is a sublattice of both $G(e,\cdot)$ and $G(\cdot,f)$.
      Explicitly: For all $a,b \in G(e,f)$ it holds that $a \meet'_e b = a \meet''_f b \in G(e,f)$ and $a \join'_e b = a \join''_f b \in G(e,f)$.
  \end{enumerate}

  If $a$,~$b \in G$ and $s(a) = s(b)$ we write $a \meet b = a \meet'_{s(a)} b$ and $a \join b = a \join'_{s(a)} b$.
  If $t(a) = t(b)$ we write $a \meet b = a \meet''_{t(a)} b$ and $a \join b = a \join''_{t(a)} b$.
  By \labelcref*{lg:il} this is unambiguous if both $s(a)=s(b)$ and $t(a)=t(b)$.
  The restriction of $\le$ to any of $G(e,\cdot)$, $G(\cdot,f)$ or $G(e,f)$ will simply be denoted by $\le$ again.

  An element $a$ of a lattice-ordered groupoid is called \emph{integral} if $a \le s(a)$ and $a \le t(a)$, and we write $G_+$ for the subset of all integral elements of $G$.
\end{definition}

\begin{definition}
  A lattice-ordered groupoid $G$ is called an \emph{arithmetical groupoid} if it has the following properties for all $e, f \in G_0$:

  \begin{enumerate}[label=\textup{(\textbf{P\arabic*})},ref=\textup{(P\arabic*)}]
  \item\label{p:int} For $a \in G$, $a \le s(a)$ if and only if $a \le t(a)$.
  \item\label{p:mod} $G(e, \cdot)$ and $G(\cdot, f)$ are modular lattices.
  \item\label{p:ord} If $a \le b$ for $a$,~$b \in G(e, \cdot)$ and $c \in G(\cdot,e)$, then $ca \le cb$. Analogously, if $a$,~$b \in G(\cdot, f)$ and $c \in G(f, \cdot)$, then $ac \le bc$.
  \item\label{p:sup} For every non-empty subset $M \subset G(e,\cdot) \cap G_+$, the supremum $\sup(M) \in G(e,\cdot)$ exists, and similarly for $M \subset G(\cdot, f) \cap G_+$.
    If moreover $M \subset G(e,f)$ then $\sup_{G(e,\cdot)}(M) = \sup_{G(\cdot,f)}(M)$.
  \item\label{p:bdd} $G(e,f)$ contains an integral element.
  \item\label{p:acc} $G(e, \cdot)$ and $G(\cdot, f)$ satisfy the ACC on integral elements.
\end{enumerate}
\end{definition}

Let $G$ be an arithmetical groupoid.
The set of integral elements $G_+$ is a reduced subcategory.
An element $a \in G_+$ is an atom if and only if it is \emph{maximal integral}, that is, it is maximal in $G_+(s(a),\cdot) \setminus \{s(a)\}$ (equivalently, in $G_+(\cdot,t(a)) \setminus \{t(a)\}$) with respect to $\le$.
The category $G_+$ is atomic.
The groupoid $G$ is a group if and only if it is free abelian with basis $\cA(G_+)$ and, in this case, $G_+$ is the free abelian monoid with basis $\cA(G_+)$.

For all $e \in G_0$, the group $G(e)$ is a free abelian group and, if $f \in G_0$, then every $a \in G(e,f)$ induces an order-preserving group isomorphism $G(e) \to G(f)$ defined by $x \mapsto a^{-1}xa$ that is independent of the choice of $a$.
For $e \in G_0$ and $x \in G(e)$ we define $(x) = \{\, a^{-1}xa : a \in G(e,\cdot) \,\}$, and we set
\[
  \bG = \{\, (x) : x \in G(e), e \in G_0 \,\}.
\]
For all $e \in G_0$, the map $G(e) \to \bG$ defined by $x \mapsto (x)$ is a bijection, and therefore induces the structure of a free abelian group on $\bG$.
We note that this structure is independent of the choice of $e$ and we call $\bG$ the \emph{universal vertex group}.
Moreover, with the order induced from any $G(e)$, $\bG$ is an arithmetical groupoid.
The subsemigroup of integral elements, $\bG_+$, is the free abelian monoid on $\cA(\bG_+)$, the maximal integral elements of $\bG_+$.
In particular, $\bG_+$ is factorial with set of prime elements $\cA(\bG_+)$.
For $\cX \in \bG$ we denote its unique preimage in $G(e)$ by $\cX_e$.

Let $u \in \cA(G_+)$ and let $x = \sup\{ x' \in G(s(u)) : x' \le u \} \in G(s(u))$.
We define $\eta(u) \in \bG$ as $\eta(u) = (x)$ and note that $\eta(u) \in \cA(\bG_+)$.
Let $a \in G_+$ and $\rf[s(a)]{u_1,\cdots,u_k} \in \sZ_{G_+}^*(a)$.
A key result on the factorization theory of $G_+$ is the following:
For all $\rf[s(a)]{v_1,\cdots,v_l} \in \sZ_{G_+}^*(a)$ we have $l=k$ and there exists a permutation $\sigma \in \fS_k$ such that $\eta(u_i) = \eta(v_{\sigma(i)})$ for all $i \in [1,k]$ (see \cite[Proposition 4.12]{Smertnig}, noting that $\Phi(u)=\eta(u)$ for $u \in \cA(G_+)$).
Thus we can extend $\eta$ to a homomorphism $G_+ \to \bG$ and further to a surjective homomorphism $\eta\colon G \to \bG$, which we call the \emph{abstract norm}.
We also recall that, given any permutation $\tau \in \fS_k$, there exist $w_1,\ldots,w_k \in \cA(G_+)$ such that $\rf[s(a)]{w_1,\cdots,w_k} \in \sZ_{G_+}^*(a)$ and $\eta(u_i) = \eta(w_{\tau(i)})$ for all $i \in [1,k]$.

\medskip
We now fix the following notation which we will tacitly assume up to and including \cref{thm:cat-transfer}.
Let $G$ be an arithmetical groupoid and let $H \subset G_+$ be a right-saturated subcategory (that is, $HH^{-1} \cap G_+ = H$).
Let $\eta\colon G \to \bG$ be the abstract norm.
By $\quo(\eta(H))$ we denote the quotient group of $\eta(H)$, where we may assume $\quo(\eta(H)) \subset \bG$.
We set $\cgrp = \bG / \quo(\eta(H))$ and, for $\cG \in \bG$, we set $[\cG] = \cG \quo(\eta(H)) \in \cgrp$.
The abelian group $C$ is a kind of class group and we shall use additive notation for it.
Finally, we define $\cgrp_M = \{\, [\eta(u)] \in \cgrp : \text{$u \in \cA(G_+)$} \,\}$ and note the following lemma.
The second statement of the lemma was stated as an assumption in \cite[Theorem 4.15]{Smertnig}, but in fact always holds.

\begin{lemma} \label{ex-cgrp}
  We have $\cgrp_M = \{\, [\cP] : \cP \in \cA(\bG_+) \,\}$.
  Let $e \in G_0$ and $g \in \cgrp_M$.
  Then there exists an element $u \in \cA(G_+)$ such that $s(u) = e$ and $[\eta(u)] = g$.
\end{lemma}

\begin{proof}
  If $u \in \cA(G_+)$, then $\eta(u) \in \cA(\bG_+)$, and therefore $\cgrp_M$ is contained in $\{\, [\cP] : \cP \in \cA(\bG_+) \,\}$.
  We prove the other inclusion.
  Let $\cP \in \cA(\bG_+)$ and $e \in G_0$.
  Since $G_+(e,\cdot)$ satisfies the ACC by \labelcref{p:acc}, the set
  \[
    \{\, u' \in G(e,\cdot) : \cP_e \le u' < e \,\}
  \]
  possesses a maximal element $u \in G_+(e,\cdot)$.
  Then $u$ is maximal integral, that is, $u \in \cA(G_+)$.
  Since $\cP_e$ is maximal integral in $G(e)$, it is necessarily the case that $\eta(u) = \cP$.
  Hence $[\eta(u)] = [\cP]$.
  Thus $\cgrp_M$ has the claimed form.
  The second statement follows similarly, by taking $\cP$ a representative of $g$.
\end{proof}

We write $\cF(\cgrp_M)$ for the free abelian monoid on $\cgrp_M$, and $\cB(\cgrp_M)$ for the monoid of zero-sum sequences over $C_M$ (see \cref{sec:preliminaries}, page~\pageref{monoid-zss}).
Since $\bG_+$ is a free abelian monoid, there exists a homomorphism $\varphi_0\colon \bG_+ \to \cF(\cgrp_M)$ such that $\varphi_0(\cP) = [\cP]$ for all $\cP \in \cA(\bG_+)$.
We denote by $\varphi\colon \eta(H) \to \cB(\cgrp_M)$ the restriction of $\varphi_0$ to $\eta(H)$.
Explicitly, if $\cX \in \eta(H)$, then there exist uniquely determined $k \in \bN_0$ and $\cP_1$, $\ldots\,$,~$\cP_k \in \cA(\bG_+)$ such that $\cX = \cP_1\cdots\cP_k$, and we have $\varphi(\cX) = [\cP_1]\cdots [\cP_k] \in \cB(\cgrp_M)$.

We define $\theta\colon H \to \cB(\cgrp_M)$ as $\theta = \varphi \circ \eta$.
Thus $\theta$ is a homomorphism from $H$ to the monoid of zero-sum sequences over $C_M$ given as follows: If $a \in H$ and $\rf[s(a)]{u_1,\cdots,u_k} \in \sZ^*_{G_+}(a)$, then $\theta(a) = [\eta(u_1)]\cdots[\eta(u_k)]$ (note that the factorization is into maximal integral elements of $G$, and not into atoms of $H$).
We call $\theta$ the \emph{block homomorphism of $H \subset G_+$}. \label{block-hom}

We will also require the following additional hypothesis:
\begin{enumerate}[label=\textup{(\textbf{N})},ref=\textup{(\textbf{N})}]
\item\label{ass:norm} for $a \in G$ with $s(a) \in H_0$, we have $a \in HH^{-1}$ if and only if $\eta(a) \in \quo(\eta(H))$.
\end{enumerate}
In \ref*{ass:norm} we may equivalently require $t(a) \in H_0$ instead of $s(a) \in H_0$ (the equivalence of the two statements follows by considering $a^{-1}$).
By \cite[Theorem 4.15]{Smertnig}, $\theta \colon H \to \cB(C_M)$ is a transfer homomorphism if \labelcref{ass:norm} holds, and it is this transfer homomorphism that we ultimately investigate.

We now provide two easy combinatorial lemmas that will later be useful to obtain lower bounds on the catenary degree.

\begin{lemma}\label{rigid-2}
  Suppose that $\cgrp_M = -\cgrp_M$ and that every class in $\cgrp_M$ contains at least two distinct prime elements of $\bG_+$.
  Then $\sc^*(H) \ge 2$.
\end{lemma}

\begin{proof}
  Let $e \in H_0$ and let $\cP$,~$\cQ \in \bG_+$ be two distinct prime elements.
  By assumption we can choose them such that $[\cP]=[\cQ]$.
  Let $u_1 \in \cA(G_+)$ with $s(u_1)=e$ and $\cP_e \le u_1$, so that $\eta(u_1) = \cP$.
  Since $C_M = -C_M$, \cref{ex-cgrp} implies that there exists an atom $u_2 \in \cA(G_+)$ such that $s(u_2) = t(u_1)$ and $[\eta(u_2)]=-[\cP]$. By hypothesis on $\cgrp_M$, we may further assume $\eta(u_2) \ne \cQ$.
  Our assumption \labelcref{ass:norm} together with $[\eta(u_1u_2)] = \vec 0 \in C$ implies $u_1 u_2 \in H H^{-1}$.
  Moreover, $u_1 u_2 \in G_+$ and $H$ is right-saturated in $G_+$, thus $u_1u_2 \in H$.
  Similarly, we construct $v_1$,~$v_2 \in \cA(G_+)$ such that $s(v_1)=t(u_2)$, $s(v_2)=t(v_1)$, $\eta(v_1) = \cQ$, and $[\eta(v_2)]=-[\cQ]$.
  As before, we have $v_1 v_2 \in H$.
  By \cite[Proposition 4.12.4]{Smertnig}, there exist $u_1'$,~$u_2'$,~$v_1'$,~$v_2' \in \cA(G_+)$ such that $\eta(u_1')=\cP$, $\eta(u_2')=\eta(u_2)$, $\eta(v_1') = \cQ$, $\eta(v_2') = \eta(v_2)$ and $a = u_1u_2v_1v_2 = v_1'v_2'u_1'u_2' \in H$.
  Again $v_1'v_2'$ and $u_1'u_2'$ lie in $H$.
  This gives rise to two rigid factorizations of $a$ in $H$, namely
  \begin{align*}
    \rf{u_1u_2,v_1v_2} \quad&\text{and}\quad\rf{v_1'v_2',u_1'u_2'} \qquad& &\text{if $[\cP]\ne \vec 0$ and $[\cQ] \ne \vec 0$, and}\\
    \rf{u_1,u_2,v_1,v_2}\quad&\text{and}\quad\rf{v_1',v_2',u_1',u_2'} \qquad& &\text{if $[\cP]=[\cQ]=\vec 0$.}
  \end{align*}
  That the stated elements are indeed atoms follows since $\theta\colon H \to \cB(\cgrp_M)$ is a transfer homomorphism.
  In each case, the two factorizations are distinct as rigid factorizations with distance at least two, because the factors have different abstract norms.
\end{proof}

\begin{lemma}\label{nonunique}
   Suppose that $\cgrp = \cgrp_M$ and that $C$ is non-trivial.
   If $\cgrp \not\cong \sC_2$ a cyclic group with two elements, then $H$ is not half-factorial.
   If $\cgrp \cong \sC_2$ and the non-trivial class contains at least two distinct prime elements $\cP$ and $\cQ$ of $\bG_+$, then there exist atoms $a$,~$b$,~$c$,~$d \in \cA(H)$ such that $ab=cd$, $\eta(a) = \cP^2$, $\eta(b)=\cQ^2$ and $\eta(c)=\eta(d)=\cP\cQ$.
\end{lemma}

\begin{proof}
  If $\cgrp \not\cong \sC_2$, then $\cB(\cgrp)$ is not half-factorial, and the existence of the transfer homomorphism from $H$ to $\cB(\cgrp)$ implies that neither is $H$.
  If $\cgrp \cong \sC_2$, let $e \in H_0$.
  Let $u_1$,~$u_2$,~$v_1$,~$v_2 \in \cA(G_+)$ with $s(u_1)=e$, $s(u_2)=t(u_1)$, $s(v_1)=t(u_2)$, $s(v_2)=t(v_1)$, $\eta(u_1)=\eta(u_2)=\cP$, and $\eta(v_1)=\eta(v_2) = \cQ$.
  Assumption \labelcref{ass:norm} implies $a=u_1u_2 \in HH^{-1}$ and $b=v_1v_2 \in HH^{-1}$.
  Since moreover $a$,~$b \in G_+$ and $H$ is right-saturated in $G_+$, we find $a$,~$b \in H$.
  Because $\eta(u_1)$,~$\eta(u_2)$, $\eta(v_1)$, and $\eta(v_2)$ all lie in the non-trivial class of $\cgrp$, we have that $\theta(a)$ and $\theta(b)$ are atoms of $\cB(\cgrp_M)$.
  Since $\theta\colon H \to \cB(\cgrp_M)$ is a transfer homomorphism, and therefore also $a \in \cA(H)$ and $b \in \cA(H)$.
  By \cite[Proposition 4.12.4]{Smertnig}, there exist $u_2'$,~$v_1' \in \cA(G_+)$ such that $u_2v_1=v_1'u_2'$ and $\eta(v_1')=\cQ$, $\eta(u_2')=\cP$.
  As before, $c=u_1v_1' \in \cA(H)$ and further $d=u_2'v_1 \in \cA(H)$.
  By construction the elements $a$,~$b$,~$c$ and $d$ have the claimed properties.
\end{proof}

To better leverage results on the transfer of the catenary degree for commutative Krull monoids, we will show that $\eta$ is, in fact, a transfer homomorphism to a commutative Krull monoid.
This will allow us to view $\theta$ as a composite of two transfer homomorphisms.
The proof is very similar to that of \cite[Theorem 4.15]{Smertnig}, but for the reader's convenience we shall state it in full.

\begin{theorem}\label{eta-transfer}
  Let $G$ be an arithmetical groupoid, let $H$ be a right-saturated subcategory of $G$, and suppose that assumption \labelcref{ass:norm} holds.
  Then $\eta(H) \subset \bG_+$ is saturated, $\eta(H)$ is a commutative Krull monoid, and $\eta\colon H \to \eta(H)$ is a transfer homomorphism.
\end{theorem}

\begin{proof}
  We first show:
  If $a \in H$ and $\eta(a) = \cX \cY$ with $\cX$,~$\cY \in \bG_+$, then there exists $x$,~$y \in G_+$ such that $a=xy$, $\eta(x) = \cX$, and $\eta(y) = \cY$.
  Moreover, if $\cX \in \eta(H)$, then $x \in H$, and if $\cY \in \eta(H)$, then $y \in H$.

  Let $a \in H$ with $\eta(a) = \cX \cY$ for some $\cX$,~$\cY \in \bG_+$.
  Let $k$,~$l \in \bN_0$ and $\cP_1$,~$\ldots\,$,~$\cP_l \in \cA(\bG_+)$ be such that $\cX = \cP_1\cdots \cP_k$ and $\cY = \cP_{k+1}\cdots \cP_{l}$.
  By \cite[Proposition 4.12]{Smertnig}, there exists a rigid factorization $\rf[s(a)]{u_1,\cdots,u_l} \in \sZ_{G_+}(a)$ such that $\eta(u_i) = \cP_i$ for all $i \in [1,l]$.
  Now set $x = s(a)u_1\cdots u_k \in G_+$ and $y = t(u_k) u_{k+1}\cdots u_l \in G_+$.
  Then $a=xy$, $\eta(x)=\cX$, and $\eta(y) = \cY$.

  Now suppose that $\cX \in \eta(H)$.
  Since $s(x) \in H_0$ and $\eta(x) = \cX \in \eta(H)$, assumption \labelcref{ass:norm} implies $x \in H H^{-1} \cap G_+$.
  Because $H$ is right-saturated in $G_+$, this implies $x \in H$.
  Now assume that $\cY \in \eta(H)$.
  Then $t(y) \in H_0$, and $\eta(y) = \cY \in \eta(H)$.
  Applying \labelcref{ass:norm} to $y^{-1}$, it again follows that $y \in HH^{-1} \cap G_+$.
  Hence $y \in H$, and the claim is established.

  We now show that $\eta$ is a transfer homomorphism.
  Since $H$ as well as $\bG_+$ are reduced and $\eta\colon H \to \eta(H)$ is surjective by definition, \labelcref{th:units} holds.
  We have to verify \labelcref{th:lift}.
  Let $a \in H$ and $\eta(a) = \cB \cC$ with $\cB$, $\cC \in \eta(H)$.
  We need to show that there exist $b$,~$c \in H$ such that $a=bc$, $\eta(b)=\cB$ and $\eta(c) = \cC$.
  This follows immediately from the claim we just proved.

  It remains to show that $\eta(H) \subset \bG_+$ is saturated.
  Then, since $\eta(H)$ is a subsemigroup of the free abelian monoid $\bG_+$, $\eta(H)$ is a commutative Krull monoid.
  Let $\cA$, $\cB \in \eta(H)$ and $\cX \in \bG_+$ be such that $\cX \cB = \cA$.
  We need to show $\cX \in \eta(H)$.

  Let $a \in H$ with $\eta(a) = \cA = \cX \cB$.
  Again by the claim, there exist $b \in H$ and $x \in G_+$ such that $a=xb$, $\eta(b) = \cB$ and $\eta(x) = \cX$.
  Then $x = ab^{-1} \in H H^{-1} \cap G_+$.
  Since $H$ is right-saturated in $G_+$, this implies $x \in H$ and hence $\cX = \eta(x) \in \eta(H)$.
\end{proof}

Let $\sd$ be a distance on $H$.
The following is the key result on the catenary degree in the permutable fibers of $\eta$. It will ultimately allow us to transfer results on the permutable catenary degree in $\eta(H)$ (respectively $\cB(\cgrp_M)$) to results on the catenary degree in distance $\sd$ on $H$.

\begin{proposition}\label{prop:perm}
  Let $\sd$ be a distance on $H$.
  \begin{enumerate}
  \item \label{prop:perm:perm}
    Let $a \in H$, $z = \rf[s(a)]{a_1,\cdots,a_m} \in \sZ_H^*(a)$ with $m \in \bN_0$ and $a_1$,~$\ldots\,$,~$a_m \in \cA(H)$, and let $\sigma \in \fS_m$ be a permutation.
    Then there exist $a_1'$,~$\ldots\,$,~$a_m' \in \cA(H)$ such that $a=s(a)a_1'\cdots a_m'$, $\eta(a_i') = \eta(a_{\sigma(i)})$ for all $i \in [1,m]$, and such that there exists a $2$-chain in distance $\sd$ between $z$ and $z'=\rf[s(a)]{a_1',\cdots,a_m'}$.
    Furthermore, every rigid factorization in the chain has length $m$, and the sequence of abstract norms of its atoms is the same as the sequence of abstract norms of the atoms in $z$, up to permutation.

  \item \label{prop:perm:cat} $\sc_\sd(H, \eta) \le 2$.
  \end{enumerate}
\end{proposition}

\begin{proof}
  \ref*{prop:perm:perm}
  We may assume $m \ge 2$, as the claim is trivially true otherwise.
  Since $\fS_m$ is generated by transpositions of the form $(i, i+1)$ for $i \in [1,m-1]$, it suffices to prove the claim where $\sigma$ is such a transposition.
  Moreover, by property \labelcref{d:mul} of a distance, we may even assume that $m=2$ and $\sigma=(1\, 2)$.
  Therefore it suffices to prove: If $a$, $b \in \cA(H)$ with $t(a)=s(b)$, then there exist $a'$, $b' \in \cA(H)$ such that $ab = b'a'$, $\eta(a) = \eta(a')$, and $\eta(b) = \eta(b')$.
  Then \labelcref{d:len} implies $\sd(\rf{a,b}, \rf{b',a'}) \le 2$.

  Let $a = u_1\cdots u_k$ and $b = v_1\cdots v_l$ with $k$, $l \in \bN$ and $u_1$,~$\ldots\,$,~$u_k$, $v_1$,~$\ldots\,$,~$v_l \in \cA(G_+)$.
  By \cite[Proposition 4.12.4]{Smertnig}, there exist $u_1'$,~$\ldots\,$,~$u_k'$, $v_1'$,~$\ldots\,$,~$v_l' \in \cA(G_+)$ such that
  \[
    u_1\cdots u_k v_1 \cdots v_l = v_1'\cdots v_l' u_1' \cdots u_k'
  \]
  with $\eta(u_i') = \eta(u_i)$ for all $i \in [1,k]$ and $\eta(v_i') = \eta(v_i)$ for all $i \in [1,l]$.
  Set $b' = v_1'\cdots v_l'$ and $a' = u_1' \cdots u_k'$.
  Then $\eta(a) = \eta(a')$ and $\eta(b) = \eta(b')$.
  Using assumption \labelcref{ass:norm}, we find $a' \in HH^{-1} \cap G_+ = H$.
  Similarly, $b' \in HH^{-1} \cap G_+ = H$.
  Using that $\eta$ is a transfer homomorphism, $\eta(a)=\eta(a')$, and $\eta(b)=\eta(b')$, it follows that $a'$,~$b' \in \cA(H)$ and hence the claim is shown.

  \ref*{prop:perm:cat}
  Let $a \in H$ and let $z$,~$z' \in \sZ^*_H(a)$ with $\sd_p(\eta(z),\eta(z'))=0$.
  We must show that there exists a $2$-chain of rigid factorizations of $a$ between $z$ and $z'$ lying in the permutable fiber of $z$.
  Since $\sd_p(\eta(z),\eta(z'))=0$, $z=\rf{a_1,\cdots,a_m}$ and $z'=\rf{b_1,\cdots,b_m}$ with $m \in \bN_0$, $a_1$,~$\ldots\,$,~$a_m$, $b_1$,~$\ldots\,$,~$b_m \in \cA(H)$ and such that there exists a permutation $\sigma \in \fS_m$ with $\eta(a_{\sigma(i)})=\eta(b_i)$.
  Applying (1), we may assume without loss of generality that $\sigma=\operatorname{id}$, and hence $\eta(a_i)=\eta(b_i)$ for all $i \in [1,m]$.

  Let $n \in \bN$ and let $u_1$, $\ldots\,$,~$u_n$, $v_1$, $\ldots\,$,~$v_n \in \cA(G_+)$ be such that $a_m=u_1\cdots u_n$ and $b_m=v_1\cdots v_n$ with $\eta(u_i)=\eta(v_i)$ for each $i \in [1,n]$.
  (This choice is always possible due to \cite[Proposition 4.12.4]{Smertnig} and since $\eta(b_m)=\eta(a_m)$.)
  Let $k \in [0,n]$ be minimal such that $u_l=v_l$ for all $l \in [k+1,n]$.
  We proceed by induction on $(m,k)$ using lexicographic order, i.e., $(m',k') < (m,k)$ if and only if $m' < m$ or $m'=m$ and $k' < k$.

  If $m \le 2$, then the claim is trivially true.
  If $m > 2$ but $k=0$, then $a_m=b_m$.
  By the induction hypothesis applied to $(m-1,k')$ for some $k' \in \bN_0$, there exists a $2$-chain from $\rf{a_1,\cdots,a_{m-1}}$ to $\rf{b_1,\cdots,b_{m-1}}$ with the corresponding properties.
  Appending a factor $a_m$ to each of these rigid factorizations gives a $2$-chain from $z$ to $z'$ (using property \ref{d:mul}).

  We therefore need only consider the case $m > 2$ and $k \ge 1$.
  Let $e=t(u_k)=t(v_k)$ and $c = u_k \meet v_k \in G_+(\cdot, e)$.
  By the choice of $k$, $u_k \ne v_k$, and thus $u_k \join v_k = e$ in $G(\cdot,e)$.
  Modularity of the lattice $G(\cdot,e)$ therefore implies that the element $c$ has length two in $G_+$.
  Consequently $c=v_{k-1}'u_k=u_{k-1}'v_k$ for some $v_{k-1}'$,~$u_{k-1}' \in \cA(G_+)$.
  By \cite[Proposition 4.12.1]{Smertnig} together with \cite[Lemma 4.11.4]{Smertnig}, we find that $\eta(v_{k-1}')=\eta(v_k)$ and $\eta(u_{k-1}')=\eta(u_k)$, so the abstract norms of all these four elements are the same.

  Since $a \le u_ku_{k+1}\cdots u_n$ and $a \le v_ku_{k+1}\cdots u_n \in G(\cdot,t(a))$, we have $a \le c u_{k+1}\cdots u_n$ as well.
  Hence there exists $d \in G_+$ with $a=dcu_{k+1}\cdots u_n$.
  Noting that $\eta(c) \mid \eta(a(u_{k+1}\cdots u_n)^{-1})$ and using (1) to rewrite $\rf{a_1,\cdots,a_{m-1}}$ and $\rf{b_1,\cdots,b_{m-1}}$ as necessary, we may assume without restriction that $\eta(c) \mid \eta(a_{m-1}a_m(u_{k+1}\cdots u_n)^{-1})$, while still preserving $\eta(a_i)=\eta(b_i)$ for each $i \in [1,m]$.

  As in the proof of \cref{eta-transfer} we can write $d=d_1d_2$ with $d_1$,~$d_2 \in G_+$, $\eta(d_1)=\eta(a_1\cdots a_{m-2})$ and $\eta(d_2cu_{k+1}\cdots u_n)=\eta(a_{m-1}a_m)$.
  (Here we have used that $\eta(c) \mid \eta(a_{m-1}a_m(u_{k+1}\cdots u_n)^{-1}$.)
  Note that $d_1 \in H$, by \labelcref{ass:norm} and $HH^{-1}\cap G_+=H$.
  Similarly, $d_2cu_{k+1}\cdots u_n \in H$.
  Since $\eta$ is a transfer homomorphism, we can take $\rf{b_1',\cdots,b_{m-2}'}$ to be a rigid factorization of $d_1$ in $H$ with $b_i' \in \cA(H)$ and $\eta(b_i')=\eta(a_i)$ for all $i \in [1,m-2]$.
  Arguing as we just have, we find a rigid factorization $b_{m-1}'*b_m'$ of $d_2cu_{k+1}\cdots u_n$ in $H$ satisfying $b_m' \le u_ku_{k+1}\cdots u_n$ in $G(\cdot,t(a))$ and $\eta(b_i')=\eta(a_i)$ for $i \in \{m-1,m\}$.

  By the induction hypothesis, there exists a $2$-chain from $z$ to $z''=\rf{b_1',\cdots,b_m'}$ lying in the permutable fiber of $z$.
  Since $v_{k-1}'u_k=u_{k-1}'v_k$, and since the abstract norms of all these elements are the same, we have $b_{m-1}'b_m'=b_{m-1}''b_m''$ with $b_{m-1}''$,~$b_m'' \in \cA(H)$ and $b_m'' \le v_ku_{k+1}\cdots u_n$ in $G(\cdot,t(a))$.
  By the induction hypothesis, there exists a $2$-chain from $z'$ to $z'''=\rf{b_1',\cdots,b_{m-2}',b_{m-1}'',b_m''}$, and by property \labelcref{d:mul} of distances, we have $\sd(z'',z''') \le 2$.
  Hence the claim is shown.
\end{proof}

\begin{theorem}\label{thm:cat-transfer}
  Let $G$ be an arithmetical groupoid, $\eta\colon G \to \bG$ its abstract norm, $H \subset G_+$ a right-saturated subcategory, $C = \bG/\quo(\eta(H))$ and $C_M = \{\, [\eta(u)] \in C : u \in \cA(G_+) \}$.
  Let $\theta\colon H \to \cB(C_M)$ be the block homomorphism of $H \subset G_+$ (see page~\pageref{block-hom}) and let $\sd$ be a distance on $H$.
  Assume that \labelcref{ass:norm} holds, that is, for $a \in G$ with $s(a) \in H_0$, we have $a \in HH^{-1}$ if and only if $\eta(a) \in \quo(\eta(H))$.

  Then $\theta$ is a transfer homomorphism and
  \[
    \sc_{\sd}(H,\theta) \le 2.
  \]
  Therefore, for all $a \in H$,
  \begin{align*}
    \sc_\sd(a) &\le \max\{ \sc_p(\theta(a)),\, 2 \},  & \sc_\sd(H) & \le \max\{ \sc_p(\cB(\cgrp_M)), 2 \}, \\
    \sc_{\sd,\cmon}(a) &\le \max\{ \sc_{p,\cmon}(\theta(a)),\, 2 \},  & \sc_{\sd,\cmon}(H) & \le \max\{ \sc_{p,\cmon}(\cB(\cgrp_M)), 2 \}, \\
    \sc_{\sd,\ceq}(a) &\le \max\{ \sc_{p,\ceq}(\theta(a)),\, 2 \},  & \sc_{\sd,\ceq}(H) & \le \max\{ \sc_{p,\ceq}(\cB(\cgrp_M)), 2 \}.
  \end{align*}
\end{theorem}

\begin{proof}
  We need only show that $\theta$ is a transfer homomorphism with $\sc_{\sd}(H,\theta) \le 2$.
  The remaining claims then follow from \subref{cf:chain}.
  By \cref{prop:perm}, $\eta\colon H \to \eta(H)$ is a transfer homomorphism with $\sc_\sd(H,\eta) \le 2$ and $\eta(H)$ is a commutative Krull monoid.
  Now, by \cite[Proposition 3.4.8]{GHK06}, the homomorphism $\varphi\colon \eta(H) \to \cB(\cgrp_M)$ is a transfer homomorphism with $\sc_p(\eta(H), \varphi) \le 2$.
  Since $\theta = \varphi \circ \eta$, the map $\theta$ is a transfer homomorphism with $\sc_\sd(H,\theta) \le 2$.
\end{proof}

We now derive, from the previous abstract result, a result for arithmetical maximal orders by means of their divisorial one-sided ideal theory. (See \cref{sec:preliminaries}, page~\pageref{arith-max-ord} for the definition of an arithmetical maximal order and a short summary of the necessary notions of their divisorial one-sided ideal theory. More details can be found in \cite[Section 5]{Smertnig}.)
Let $Q$ be a quotient semigroup and $S$ an order in $Q$.
We set
\[
  \cH_S = \{\, q^{-1}(Sa)q : q \in Q^\bullet, a \in S^\bullet \,\}
\]
to be the category of principal one-sided $S'$-ideals with $S'$ conjugate to $S$ and having a cancellative generator.
Recall that $\cH_S$ forms a cancellative small category with $(\cH_S)_0 = \{\, q^{-1}Sq : q \in Q^\bullet \,\}$ such that $s(q^{-1}(Sa)q) = q^{-1}Sq$ and $t(q^{-1}(Sa)q) = (aq)^{-1}Saq$ for each $q \in Q^\bullet$ and $a \in S^\bullet$, and such that the multiplication of ideals is induced by the multiplication of elements in $Q^\bullet$:
\[
  q^{-1}(Sa)q \cdot (aq)^{-1}(Sb)aq = q^{-1}(Sba)q.
\]

Despite not having a homomorphism $S^\bullet \to \cH_S$, we have, for all $a \in S^\bullet$ and $q \in Q^\bullet$, a bijection $\phi_{a,q}\colon \sZ^*_S(a) \to \sZ^*_{\cH_S}(q^{-1}(Sa)q)$ given by
\[
  \rf[\varepsilon]{u_1,\cdots,u_k} \mapsto \rf{q^{-1}(Hu_k) q,   q^{-1}u_k^{-1}(H u_{k-1}) u_k q,  \cdots,  q^{-1}u_k^{-1}\cdots u_2^{-1}(Hu_1) u_2\cdots u_k q}.
\]
(See \cite[Proposition 5.20]{Smertnig}, where this is stated in a more restrictive setting, but with a proof that generalizes verbatim.)
Observe that if $q=1$, the right hand side is essentially a multiplicative way of writing the chain of left $H$-ideals $H \supsetneq H u_k \supsetneq H u_{k-1} u_k \supsetneq \cdots \supsetneq H u_1 \cdots u_k$.
In this way, questions about factorization in $S^\bullet$ translate into questions about factorization in $\cH_S$.

To extend distances from $S^\bullet$ to $\cH_S$ we need to impose one additional mild restriction on the distance.
If $n \in Q^\bullet$ normalizes $S$, that is $n^{-1}Sn = S$, it induces an automorphism on $S^\bullet$ given by $a \mapsto n^{-1}an$ for $a \in S^\bullet$. This automorphism in turn induces an automorphism
\[
\psi_n\colon \sZ^*(S) \to \sZ^*(S),\;\; \rf[\varepsilon]{u_1,\cdots,u_k} \mapsto \rf[(n^{-1}\varepsilon n)]{n^{-1}u_1n,\cdots, n^{-1} u_k n},
\]
which induces bijections $\sZ_S^*(a) \to \sZ_S^*(n^{-1}an)$ for all $a \in S^\bullet$.
We say that a distance $\sd$ on $S^\bullet$ is \emph{invariant under conjugation by normalizing elements} if $\sd(z,z') = \sd(\psi_n(z),\psi_n(z'))$ for all $z$,~$z' \in \sZ^*(S)$ with $\pi(z)=\pi(z')$ and for all $n \in Q^\bullet$ that normalize $S$.

Similarly, the elements of $Q^\bullet$ normalizing $S$ act on $\cH_S$ by mapping
\[
  q^{-1}(Sa)q \;\mapsto\; n^{-1}q^{-1}(Sa)qn = n^{-1}q^{-1}n(Sn^{-1}an)n^{-1}qn.
\]
This in turn induces an action of the normalizing elements on $\sZ^*(\cH_S)$, where $n$ acts by
\[
  \Psi_n\colon \sZ^*(\cH_S) \to \sZ^*(\cH_S),\; \rf[S']{I_1,\cdots,I_k} \mapsto \rf[(n^{-1}S'n)]{n^{-1}I_1n, \cdots, n^{-1}I_kn},
\]
and by restriction we obtain bijections $\sZ^*_{\cH_S}(q^{-1}(Sa)q) \to \sZ^*_{\cH_S}(n^{-1}q^{-1}n(Sn^{-1}an)n^{-1}qn)$ for all $q \in Q^\bullet$ and $a \in S^\bullet$.
Again, we say that a distance $\sd$ on $\cH_{S}$ is \emph{invariant under conjugation by normalizing elements} if $\sd(z,z')=\sd(\Psi_n(z),\Psi_n(z'))$ for all $z$,~$z' \in \sZ^*(\cH_S)$ with $\pi(z)=\pi(z')$ and for all $n \in Q^\bullet$ that normalize $S$.

Observe that each distance introduced ($\sd^*$, $\sd_p$, $\sd_\dsim$, and $\sd_\dsubsim$) is in fact invariant under any automorphism of $S$, and that it seems quite reasonable to expect any natural distance to have this property.

We now observe how the family of bijections $\phi_{a,q}\colon \sZ^*_S(a) \to \sZ^*_{\cH_S}(q^{-1}(Sa)q)$ behaves with respect to conjugation by normalizing elements.
Let $z$,~$z' \in \sZ^*(S)$ and $q \in Q^\bullet$.
Set $a = \pi(z)$ and $b=\pi(z')$.
Then $\phi_{ab,q}(z\rfop z') = \phi_{b,q}(z') \rfop \phi_{a,bq}(z)$ and $\phi_{n^{-1}an, n^{-1}qn}(\psi_n(z)) = \Psi_n(\phi_{a,q}(z)) = \phi_{a,qn}(z)$ for each $n \in Q^\bullet$ that normalizes $S$.
Now let $z$,~$z' \in \sZ^*(\cH_S)$ with $t(z)=s(z')$.
We may suppose $\pi(z)=q^{-1}(Sa)q$ and $\pi(z') = (aq)^{-1}(Sb)(aq)$ with $q \in Q^\bullet$ and $a$,~$b \in S^\bullet$.
Then $\phi_{ba,q}^{-1}(z \rfop z') = \phi_{b,aq}^{-1}(z') \rfop \phi_{a,q}^{-1}(z)$ and $\phi_{a,qn}^{-1}(\Psi_n(z)) = \phi_{a,q}^{-1}(z)$ for each $n \in Q^\bullet$ that normalizes $S$.

The proof of the following proposition is now relatively straightforward, but somewhat cumbersome.

\begin{proposition}\label{prop:dist-bij}
  Let $Q$ be a quotient semigroup and let $S$ be an order in $Q$.

  Let $\sd_S$ be a distance on $S^\bullet$ that is invariant under conjugation by normalizing elements.
  Then there exists a unique distance on $\cH_S$, denoted by $\sd_{\cH}$, such that
  \[
    \sd_{\cH}(z,z') = \sd_{S}(\phi_{a,q}^{-1}(z), \phi_{a,q}^{-1}(z'))
  \]
  for all $z$, $z' \in \sZ^*(\cH_S)$ and for all $q \in Q^\bullet$ and $a \in S^\bullet$ with $\pi(z)=\pi(z')=q^{-1}(Sa)q$.

  Conversely, if $\sd_\cH$ is a distance on $\cH_S$ that is invariant under conjugation by normalizing elements, then there exists a unique distance $\sd_S$ on $S^\bullet$ such that
  \[
    \sd_S(z,z') = \sd_{\cH}(\phi_{a,q}(z), \phi_{a,q}(z'))
  \]
  for all $z$,~$z' \in \sZ^*(S)$ with $a=\pi(z)=\pi(z')$ and $q \in Q^\bullet$.

  In particular, there is a bijection between distances on $S^\bullet$ which are invariant under conjugation by normalizing elements and distances on $\cH_S$ which are invariant under conjugation by normalizing elements.
\end{proposition}

\begin{proof}
  First let $\sd_S$ be a distance on $S^\bullet$ that is invariant under conjugation by elements of $Q^\bullet$ that normalize $S$.
  Let $z$, $z' \in \sZ^*(\cH_S)$ with $\pi(z)=\pi(z')$.
  Let $q \in Q^\bullet$ and $a \in S^\bullet$ be such that $\pi(z)=\pi(z') = q^{-1}(Sa)q$.
  Note that, after fixing $q$, the element $a$ is uniquely determined up to left associativity.
  We wish to define $\sd_{\cH}$ by $\sd_{\cH}(z,z') = \sd_{S}(\phi_{a,q}^{-1}(z), \phi_{a,q}^{-1}(z'))$ and need to show that the expression on the right hand side is independent of the choice of $q$ and $a$.

  Suppose $q' \in Q^\bullet$ and $a' \in S^\bullet$ are such that $(q')^{-1}(Sa')q'=q^{-1}(Sa)q$.
  Let $z = \rf[(q^{-1}Sq)]{I_k,\cdots,I_1}$ and $z' = \rf[(q^{-1}Sq)]{J_l,\cdots,J_1}$ with $k$, $l \in \bN_0$ and $I_1$, $\ldots\,$,~$I_k$, $J_1$,~$\ldots\,$,~$J_l \in \cH_S$.
  If $k=0$, then, since $\pi(z)=\pi(z')$, $l=0$. Similarly, $l=0$ implies $k=0$.
  In that case $\sd_{S}(\phi_{a,q}^{-1}(z), \phi_{a,q}^{-1}(z')) = \sd_{S}(\phi_{a',q'}^{-1}(z), \phi_{a',q'}^{-1}(z')) = 0$.
  From now on we may assume $k$,~$l > 0$.
  Following the construction in \cite[Proposition 5.20]{Smertnig}, there exist $u_1$, $\ldots\,$,~$u_k$, $u_1'$,$\ldots\,$,~$u_k'$, $v_1$,$\ldots\,$,~$v_l$, $v_1'$,$\ldots\,$,~$v_l' \in \cA(S^\bullet)$ such that
  \begin{equation}\label{eq:udef}
    I_i = q^{-1} u_k^{-1} \cdots u_{i+1}^{-1} (Su_i) u_{i+1} \cdots u_{k} q = (q')^{-1} (u_k')^{-1} \cdots (u_{i+1}')^{-1} (Su_i') u_{i+1}' \cdots u_{k}' q'
  \end{equation}
  for all $i \in [1,k]$,
  \begin{equation}\label{eq:vdef}
    J_j = q^{-1} v_l^{-1} \cdots v_{j+1}^{-1} (Sv_j) v_{j+1} \cdots v_{l} q = (q')^{-1} (v_l')^{-1} \cdots (v_{j+1}')^{-1} (Sv_j') v_{j+1}' \cdots v_{l}' q'
  \end{equation}
  for all $j \in [1,l]$, and
  \begin{align*}
    \phi_{a,q}^{-1}(z) &= \rf{u_1,\cdots,u_k},   & \phi_{a',q'}^{-1}(z) &= \rf{u_1',\cdots,u_k'}, \\
    \phi_{a,q}^{-1}(z') &= \rf{v_1,\cdots,v_l},   & \phi_{a',q'}^{-1}(z') &= \rf{v_1',\cdots,v_l'}.
  \end{align*}
  We must show $\sd_S(\rf{u_1,\cdots,u_k}, \rf{v_1,\cdots,v_l}) = \sd_S(\rf{u_1',\cdots,u_k'}, \rf{v_1',\cdots,v_l'})$.
  Note that $n = q(q')^{-1} \in Q^\bullet$ normalizes $S$.
  From \cref{eq:udef,eq:vdef} applied with $i=k$ and $j=l$, we deduce that $S u_k' = Sn^{-1}u_kn$ and $Sv_l' = Sn^{-1}v_ln$.
  Hence there exist $\varepsilon_k$, $\eta_l \in S^\times$ such that $u_k' = \varepsilon_k n^{-1}u_kn$ and $v_l' = \eta_l n^{-1}v_l n$.
  Inductively, we find that there exist $\varepsilon_1$, $\ldots\,$,~$\varepsilon_{k-1}$ and $\eta_1$, $\ldots\,$,~$\eta_{l-1}$ such that $u_i' = \varepsilon_i (n^{-1} u_i n) \varepsilon_{i+1}^{-1}$ and $v_j' = \eta_j (n^{-1} v_j n) \eta_{j+1}^{-1}$ for all $i \in [1,k-1]$ and $j \in [1,l-1]$.
  Thus
  \[
    \rf{u_1',\cdots,u_k'} = \rf[\varepsilon_1]{n^{-1}u_1n, \cdots, n^{-1}u_k n} \quad\text{and}\quad \rf{v_1',\cdots,v_l'} = \rf[\eta_1]{n^{-1} v_1 n, \cdots, n^{-1} v_l n}.
  \]
  Since $a' = \varepsilon_1 n^{-1} u_1\cdots u_k n = \varepsilon_1 n^{-1} a n$ and similarly $a' = \eta_1 n^{-1} a n$, we have $\varepsilon_1 = \eta_1$.
  First using property \labelcref{d:mul} of a distance and then the invariance under normalizing elements, we find that
  \[
  \begin{split}
     &\sd_S(\rf[\varepsilon_1]{n^{-1}u_1n, \cdots, n^{-1}u_k n},\; \rf[\varepsilon_1]{n^{-1} v_1 n, \cdots, n^{-1} v_l n})  \\
    &= \sd_S(\rf{n^{-1}u_1n, \cdots, n^{-1}u_kn},\; \rf{n^{-1}v_1n, \cdots, n^{-1}v_ln}) \\
    &= \sd_S(\rf{u_1,\cdots,u_k},\rf{v_1,\cdots,v_l}).
  \end{split}
  \]
  Hence we may define $\sd_{\cH}(z,z') = \sd_{S}(\phi_{a,q}^{-1}(z), \phi_{a,q}^{-1}(z'))$, and it is clear that in this way $\sd_\cH$ is uniquely determined.

  We now check that $\sd_\cH$ is indeed a distance on $\cH$.
  Properties \labelcref{d:ref,d:sym,d:tri,d:len} are immediate from the definition and the fact that $\sd_S$ satisfies each of these properties.
  We now check \labelcref{d:mul}.
  Let $x$, $y$, $z$, $z' \in \sZ^*(\cH_S)$ be such that $\rf{x,z,y}$ and $\rf{x,z',y}$ are defined and $\pi(z)=\pi(z')$.
  We may assume $\pi(x)=q^{-1}(Sc)q$, $\pi(z) = \pi(z') = (cq)^{-1}(Sb)cq$ and $\pi(y) = (bcq)^{-1}(Sa)(bcq)$ with $q \in Q^\bullet$ and $a$,~$b$, $c \in S^\bullet$.
  Observe that
  \begin{align*}
    \phi_{abc,q}^{-1}(\rf{x,z,y})  &= \rf{\phi_{a,bcq}^{-1}(y),\phi_{b,cq}^{-1}(z),\phi_{c,q}^{-1}(x)},\text{ and} \\
    \phi_{abc,q}^{-1}(\rf{x,z',y}) &= \rf{\phi_{a,bcq}^{-1}(y),\phi_{b,cq}^{-1}(z'),\phi_{c,q}^{-1}(x)}.
  \end{align*}
  Therefore, using property \labelcref{d:mul} of $\sd_S$,
  \begin{align*}
    \sd_{\cH}(\rf{x,z,y}, \rf{x,z',y}) &= \sd_S(\rf{\phi_{a,bcq}^{-1}(y),\phi_{b,cq}^{-1}(z),\phi_{c,q}^{-1}(x)},\; \rf{\phi_{a,bcq}^{-1}(y),\phi_{b,cq}^{-1}(z'),\phi_{c,q}^{-1}(x)}) \\
        &= \sd_S(\phi_{b,cq}^{-1}(z), \phi_{b,cq}^{-1}(z')) \\
        &= \sd_{\cH}(z,z').
  \end{align*}

  It remains to check that $\sd_\cH$ is invariant under conjugation by normalizing elements.
  Suppose again that $z$ and $z'$ are in $\sZ^*(\cH_S)$ with $\pi(z)=\pi(z')$, and let $q \in Q^\bullet$ and $a \in S^\bullet$ be such that $\pi(z)=\pi(z')=q^{-1}(Sa)q$.
  Let $n \in Q^\bullet$ be such that it normalizes $S$.
  Then
  \[
    \sd_{\cH}(\Psi_n(z), \Psi_n(z'))
    = \sd_S(\phi_{a,qn}^{-1}(\Psi_n(z)), \phi_{a,qn}^{-1}(\Psi_n(z')))
    = \sd_S(\phi_{a,q}^{-1}(z), \phi_{a,q}^{-1}(z')) = \sd_{\cH}(z,z').
  \]

  We now prove the converse.
  Suppose that $\sd_\cH$ is a distance on $\cH_S$ that is invariant under conjugation by normalizing elements.
  Let $z$,~$z' \in \sZ^*(S)$ with $a=\pi(z)=\pi(z')$ and let $q$,~$q' \in Q^\bullet$.
  Noting that $q^{-1}q'$ normalizes $S$ and that $\phi_{a,q'}(z) = \Psi_{q^{-1}q'}(\phi_{a,q}(z))$ and $\phi_{a,q'}(z') = \Psi_{q^{-1}q'}(\phi_{a,q}(z'))$, $\sd_{\cH}(\phi_{a,q}(z), \phi_{a,q}(z')) = \sd_{\cH}(\phi_{a,q'}(z), \phi_{a,q'}(z'))$ follows immediately from the invariance of $\sd_\cH$ under normalizing elements.
  Thus a unique $\sd_S$ as claimed exists, and we must check that it is a distance on $S^\bullet$.
  Again, properties \labelcref{d:ref,d:sym,d:tri,d:len} follow immediately from the corresponding properties of $\sd_\cH$.
  Let us verify \labelcref{d:mul}, and to this end let $x$,~$y$, $z$,~$z' \in \sZ^*(S)$ with $\pi(z)=\pi(z')$.
  Let $a = \pi(x)$, $b=\pi(y)$, and $c=\pi(z)=\pi(z')$.
  Then
  \begin{align*}
    \sd_S(\rf{x,z,y}, \rf{x,z',y})
    &= \sd_\cH(\phi_{acb,1}(\rf{x,z,y}), \phi_{acb,1}(\rf{x,z',y})) \\
    &= \sd_\cH(\rf{\phi_{b,1}(y), \phi_{c,b}(z), \phi_{a,cb}(x)}, \rf{\phi_{b,1}(y), \phi_{c,b}(z'), \phi_{a,cb}(x)}) \\
    &= \sd_\cH(\phi_{c,b}(z), \phi_{c,b}(z')) \\
    &= \sd_S(z,z').
  \end{align*}
  Finally, we verify that $\sd_S$ is invariant under conjugation by normalizing elements.
  Let $n \in Q^\bullet$ be such that it normalizes $S$, and let $z$, $z' \in \sZ^*(S)$ with $a=\pi(z)=\pi(z')$.
  Then
  \begin{align*}
    \sd_S(\psi_n(z), \psi_n(z'))
    &= \sd_\cH(\phi_{n^{-1}an,1}(\psi_n(z)), \phi_{n^{-1}an,1}(\psi_n(z'))) \\
    &= \sd_\cH(\Psi_n(\phi_{a,1}(z)), \Psi_n(\phi_{a,1}(z'))) \\
    &= \sd_S(z,z'). \qedhere
  \end{align*}
\end{proof}

\begin{remark}
  Note that the extension of the rigid distance from $S^\bullet$ to $\cH_S$ is, in general, not the rigid distance on $\cH_S$.
  Indeed, let $R = M_2(\bZ)$ so that $S=M_2(\bZ)^\bullet$, let $p \in \bN$ be prime, and set
  \[
    a = \begin{pmatrix} 0 & p \\ 1 & 0 \end{pmatrix} \quad\text{and}\quad \varepsilon = \begin{pmatrix} 0 & 1 \\ 1 & 0 \end{pmatrix}.
  \]
  Viewing scalar matrices as elements in $\mathbb Z$, $a^2 = p$ and $\varepsilon^2 = 1$.
  In $S$, we have rigid factorizations
  \[
    z=\rf{a,a,a,a} \quad\text{and}\quad z'=\rf{\varepsilon a, a, a, a \varepsilon}
  \]
  of $a^4=p^2$.
  Since $\varepsilon a \ne a \eta$ and $a \varepsilon \ne \eta a$ for all $\eta \in S^\times$, these are distinct rigid factorizations with $\sd^*(z,z') = 2$ and $\sd_p(z,z') = 0$.
  We have
  \begin{align*}
    \phi_{a,1}(z) &= \rf{Sa, a^{-1}(Sa)a, a^{-2}(Sa)a^2, a^{-3}(Sa)a^3}\quad\text{ and}\\
    \phi_{a,1}(z') &= \rf{Sa\varepsilon, \varepsilon a^{-1}(Sa)a\varepsilon, \varepsilon a^{-2}(Sa)a^2\varepsilon, \varepsilon a^{-3}(S\varepsilon a)a^3\varepsilon}.
  \end{align*}
  Now $Sa\varepsilon \ne Sa$ since there exists no $\eta \in S^\times$ with $\eta a = a \varepsilon$.
  Moreover, $Sa\varepsilon$ is not equal to any of the other atoms in $\phi_{a,1}(z)$ since their left orders differ.
  Continuing this argument, one finds that the atoms in $z$ are pairwise distinct from those in $z'$.
  Recalling that $\cH_S$ is reduced, we have $\sd_p(\phi_{a,1}(z), \phi_{a,1}(z')) = \sd^*(\phi_{a,1}(z), \phi_{a,1}(z')) = 4$.
\end{remark}

Let $Q$ be a quotient semigroup, and let $S$ be an arithmetical maximal order in $Q$.
If $\alpha$ denotes the set of maximal orders of $Q$ that are equivalent to $S$ and $\cF_v(\alpha)$ denotes the divisorial fractional left $S'$-ideals with $S' \in \alpha$, then, by \cite[Proposition 5.16]{Smertnig}, $\cF_v(\alpha)$ is, with a partial operation given by the $v$-ideal multiplication, an arithmetical groupoid.
Then the subcategory $\cI_v(\alpha)$ of $\cF_v(\alpha)$ consisting of divisorial left $S'$-ideals with $S' \in \alpha$ is just the subcategory of integral elements of $\cF_v(\alpha)$.
The category $\cH_S$ is a left- and right-saturated subcategory of $\cI_v(\alpha)$ (in particular, the usual multiplication on $\cH_S$ coincides with the $v$-ideal multiplication).

We denote by $\eta$ the abstract norm on $\cF_v(\alpha)$ and define $P_{S^\bullet} = \{\, \eta(Sq) : q \in Q^\bullet \,\} \subset \bG$, $\cgrp = \bG / P_{S^\bullet}$, and $\cgrp_M = \{\, [\eta(I)] \in \cgrp : I \in \cI_v(\alpha) \text{ is maximal integral} \,\}$.
If we also assume that a divisorial fractional left $S$-ideal $I$ is principal if and only if $\eta(I) \in P_{S^\bullet}$, then \labelcref{ass:norm} holds for $\cH_S$.
The block homomorphism $\theta\colon \cH_S \to \cB(C_M)$ induces a transfer homomorphism $S^\bullet \to \cB(\cgrp_M)$, again denoted by $\theta$, by means of $\theta(a) = \theta(Sa)$ for all $a \in S^\bullet$ (see \cite[Theorem 5.23.2]{Smertnig}).
Thus we have the following corollary to \cref{thm:cat-transfer}.
Recall that this also covers the case of normalizing Krull monoids as investigated in \cite{Geroldinger} (see \cref{sec:preliminaries} and \cite[Remarks 5.17.2 and 5.24.1]{Smertnig}).

\begin{corollary}\label{cor:cat-transfer-semigroup}
  Let $S$ be an arithmetical maximal order in a quotient semigroup $Q$ and let $\alpha$ denote the set of maximal orders of $Q$ equivalent to $S$.
  Let $\eta\colon \cF_v(\alpha) \to \bG$ be the abstract norm of $\cF_v(\alpha)$, let $\cgrp = \bG / P_{S^\bullet}$, and set $\cgrp_M = \{\, [\eta(I)] \in \cgrp : I \in \cI_v(\alpha) \text{ maximal integral} \,\}$.
  Assume that a divisorial fractional left $S$-ideal $I$ is principal if and only if $\eta(I) \in P_{S^\bullet}$.
  Let $\theta\colon S^\bullet \to \cB(\cgrp_M)$ be the transfer homomorphism induced by the block homomorphism of $\cH_S \subset \cI_v(\alpha)$.
  Let $\sd$ be a distance on $S^\bullet$ that is invariant under conjugation by normalizing elements.
  Then
  \[
    \sc_{\sd}(S^\bullet,\theta) \le 2.
  \]
  Moreover, for all $a \in S^\bullet$,
  \begin{align*}
    \sc_\sd(a) &\le \max\{ \sc_p(\theta(a)),\, 2 \},  & \sc_\sd(S^\bullet) & \le \max\{ \sc_p(\cB(\cgrp_M)), 2 \}, \\
    \sc_{\sd,\cmon}(a) &\le \max\{ \sc_{p,\cmon}(\theta(a)),\, 2 \},  & \sc_{\sd,\cmon}(S^\bullet) & \le \max\{ \sc_{p,\cmon}(\cB(\cgrp_M)), 2 \}, \\
    \sc_{\sd,\ceq}(a) &\le \max\{ \sc_{p,\ceq}(\theta(a)),\, 2 \},  & \sc_{\sd,\ceq}(S^\bullet) & \le \max\{ \sc_{p,\ceq}(\cB(\cgrp_M)), 2 \}.
  \end{align*}
\end{corollary}

\begin{proof}
  By setting $G=\cF_v(\alpha)$, $G_+=\cI_v(\alpha)$ and $H=\cH_S$, we are in our previous setting.
  The additional condition \labelcref{ass:norm} is satisfied by our assumptions on $S$.

  By \cref{prop:dist-bij}, the distance $\sd$ on $S^\bullet$ gives rise to a distance $\sd_\cH$ on $\cH_S$.
  Using the bijections between $\sZ_S^*(a)$ and $\sZ_{\cH_S}^*(q^{-1}(Sa)q)$, it is now clear that $\sc_\sd(a) = \sc_{\sd_\cH}(q^{-1}(Sa)q)$ for all $q \in Q^\bullet$, and the claim therefore follows from \cref{thm:cat-transfer}.
\end{proof}

We now apply this abstract machinery to classical maximal orders in central simple algebras over global fields.
Let $K$ be a global field, $\cO$ a holomorphy ring in $K$, $A$ a central simple algebra over $K$, and let $R$ be a classical maximal $\cO$-order.
Then we may identify $\eta$ with the reduced norm on left and right $R$-ideals (see \cite[Lemma 5.32]{Smertnig}).
Let $\cF^\times(\cO)$ denote the group of nonzero fractional ideals of $\cO$, let
\[
  \cP_A(\cO) = \{\, a\cO \,:\, a \in K^\times, a_v > 0\text{ for all archimedean places $v$ of $K$ where $A$ is ramified} \,\},
\]
and $\Cl_A(\cO) = \cF^\times(\cO) / \cP_A(\cO)$.
The ray class group $\Cl_A(\cO)$ is a finite abelian group, and under the identification of the abstract norm with the reduced norm, we find $C = C_M = \Cl_A(\cO)$ (see \cite[Theorem 5.28 and Section 6]{Smertnig} for details).

We immediately obtain the following result, which is a direct analogue to the corresponding abstract result for commutative Krull monoids (\cite[Corollary 3.4.12]{GHK06}) applied to holomorphy rings in global fields.

\begin{theorem}\label{thm:cat-transfer-csa}
  Let $K$ be a global field, $\cO$ a holomorphy ring in $K$, and $R$ a classical maximal $\cO$-order in a central simple algebra $A$ over $K$.
  Let $\sd$ be a distance on $R^\bullet$ that is invariant under conjugation by normalizing elements.
  Suppose that every stably free left $R$-ideal is free, and let $\theta\colon R^\bullet \to \cB(\Cl_A(\cO))$ be the transfer homomorphism induced by the block homomorphism.
  Then
  \[
    \sc_\sd(R^\bullet) \le \max\{ 2, \sc_p(\cB(\Cl_A(\cO))) \}.
  \]
  In particular, $\sc^*(R^\bullet) = \max\{ 2, \sc_p(\cB(\Cl_A(\cO))) \}$.
  Moreover, if $\Cl_A(\cO)$ is non-trivial, then
  \begin{align*}
    \sc_p(R^\bullet) = \sc_\dsubsim(R^\bullet) = \sc_\dsim(R^\bullet) &= \max\{ 2, \sc_p(\cB(\Cl_A(\cO))) \}, \\
    \sc_{p,\cmon}(R^\bullet) = \sc_{\dsubsim,\cmon}(R^\bullet) = \sc_{\dsim,\cmon}(R^\bullet) &= \max\{ 2, \sc_{p,\cmon}(\cB(\Cl_A(\cO))) \}, \\
    \sc_{p,\ceq}(R^\bullet) = \sc_{\dsubsim,\ceq}(R^\bullet) = \sc_{\dsim,\ceq}(R^\bullet) &= \max\{ 2, \sc_{p,\ceq}(\cB(\Cl_A(\cO))) \},
  \end{align*}
  and if $\Cl_A(\cO)$ is trivial, then $R^\bullet$ is $\sd_\dsim$-factorial and $\sd_\dsubsim$-factorial.
\end{theorem}

\begin{proof}
  We first show that, in $R^\bullet$, subsimilarity already implies similarity, so that $\sd_\dsim$ and $\sd_\dsubsim$ coincide.
  Indeed, let $a$ and $a'$ be subsimilar elements of $R^\bullet$.
  Then there exists a nonzero two-sided $R$-ideal $I$ contained in $Ra \cap Ra'$, and the quotient ring $R/I$ is Artinian and Noetherian.
  The ideal $I$ is contained in both $\ann_R(R/Ra)$ and $\ann_R(R/Ra')$.
  As finitely generated $R/I$-modules, the modules $R/Ra$ and $R/Ra'$ have finite length.
  Since each of $R/Ra$ and $R/Ra'$ embeds into the other, $R/Ra$ and $R/Ra'$ must be isomorphic as $R/I$-modules and therefore also as $R$-modules.

  The upper bound for $\sc_\sd(R^\bullet)$ follows from \cref{cor:cat-transfer-semigroup}.
  The corollary applies because the assumption that every stably free left $R$-ideal is free implies the corresponding assumption in \cref{cor:cat-transfer-semigroup} (see \cite[Lemma 6.2]{Smertnig}).
  \cref{rigid-2} implies $\sc^*(R^\bullet) \ge 2$, and hence $\sc^*(R^\bullet) = \max\{ 2, \sc_p(\cB(\Cl_A(\cO))) \}$.
  We now show $\sc_{\dsim}(R^\bullet) = \sc_{\dsubsim}(R^\bullet) \ge \sc_p(\cB(\Cl_A(\cO)))$, and for this it suffices to show that two similar atoms $u$,~$v \in \cA(R^\bullet)$ are mapped to the same element by $\theta$.
  Since in this context $\eta$ corresponds to the usual reduced norm (see \cite[Lemma 5.32]{Smertnig}), it suffices to show that $\nr(Ru) = \nr(Rv)$.
 Since $\nr(Ru)$ and $\nr(Rv)$ depend only on the isomorphism class of $R/Ru \cong R/Rv$ by \cite[Corollary 24.14]{Reiner}, $\nr(Ru) = \nr(Rv)$.

 Therefore
  \[
    \sc_p(\cB(\Cl_A(\cO))) \le \sc_{\dsubsim}(R^\bullet) = \sc_{\dsim}(R^\bullet) \le \sc_p(R^\bullet) \le \max\{ 2, \sc_p(\cB(\Cl_A(\cO))) \}.
  \]
  If $\card{\Cl_A(\cO)} > 2$, then $\sc_p(\cB(\Cl_A(\cO))) \ge 2$, which establishes the full claim in this case.
  If $\Cl_A(\cO)$ is trivial, then $R$ is a PID, and hence $\sd_\dsim$-factorial (and thus also $\sd_\dsubsim$-factorial).

  It remains to consider the case where $\Cl_A(\cO) \cong \sC_2$.
  Let $\fp$, $\fq \in \spec(\cO)$ be two distinct prime ideals representing the non-trivial class of $\Cl_A(\cO)$.
  Using \cref{nonunique}, we can find atoms $u$,~$v$,~$w$,~$w' \in \cA(R^\bullet)$ such that $uv=ww'$, $\nr(u)=\fp^2$, $\nr(v)=\fq^2$, and $\nr(w)=\nr(w')=\fp\fq$.
  Then $u$ and $v$ are similar to neither $w$ nor $w'$, as similar elements have the same reduced norm.
  Thus $\sc_p(R^\bullet) = \sc_\dsim(R^\bullet) = \sc_\dsubsim(R^\bullet) = 2$.

  The statements about the monotone and equal catenary degrees are obtained analogously.
\end{proof}

Suppose that $K$ is a number field, $\cO=\cO_K$ is its ring of algebraic integers and, contrary to the assumptions of the previous corollary, there exist stably free left $R$-ideals that are non-free. Then, by \cite[Theorem 1.2]{Smertnig}, it holds that $\Delta(R^\bullet) = \bN$, and by \subref{dl:delta} we have $\sc_\sd(R^\bullet) = \infty$ for any distance $\sd$.

Observe that even for $H=G_+$ we only have $\sc^*(G_+) \le 2$ in general.
Thus not even for PIDs that are bounded orders in quotient rings can we expect $\sc^*(R^\bullet) = 0$ but instead only have $\sc^*(R^\bullet) \le 2$.
We now give a more concrete example to illustrate this fact.

\begin{example}\label{quat poly}
  Let $K$ be a field, let $D$ be a quaternion division ring with center $K$, and denote by $\lbar{\,\cdot\,} \colon D \to D$ the anti-involution given by conjugation.
  Let $D[X]$ be the polynomial ring over $K$.
  Then $D[X]$ is a PID and $S = D[X]^\bullet$ satisfies the conditions of \cref{cor:cat-transfer-semigroup}.
  We have $\cgrp = \cgrp_M = \mathbf 0$, whence $\sc_p(\cB(\cgrp)) = 0$. However, if $a \in D \setminus K$, then, for all $b \in D^\bullet$,
  \[
    (X - a) (X - \lbar{a}) = X^2 - (a + \lbar{a}) X + a\lbar{a} X = (X - bab^{-1}) (X - \lbar{bab^{-1}})\, \in K[X],
  \]
  and so clearly $\sc^*(S) = 2$ and also $\sc_p(S) = 2$.
  Recall however that this ring, being a PID, is $\sd_{\dsubsim}$-factorial and $\sd_{\dsim}$-factorial, and hence we have $\sc_{\dsubsim}(S) = \sc_{\dsim}(S) = 0$.
\end{example}

\begin{corollary}\label{equfact}
  Let $K$ be a global field, $\cO$ a holomorphy ring in $K$, and $R$ a classical maximal $\cO$-order in a central simple algebra $A$ over $K$.
  Suppose that every stably free left $R$-ideal is free.
  Then the following statements are equivalent.
  \begin{enumerateequiv}
    \item\label{equfact:pid} Every left [right] $R$-ideal is principal.
    \item\label{equfact:cgrp} The ray class group $\Cl_A(\cO)$ is trivial.
    \item\label{equfact:dsim} $R^\bullet$ is $\sd_\dsim$-factorial.
    \item\label{equfact:dsubsim} $R^\bullet$ is $\sd_\dsubsim$-factorial.
  \end{enumerateequiv}
\end{corollary}

\begin{proof}
  The set of isomorphism classes of fractional left $R$-ideals is in bijection with $\Cl_A(\cO)$ via the reduced norm, and hence the equivalence of \ref*{equfact:pid} and \ref*{equfact:cgrp} holds.
  The equivalence of \ref*{equfact:cgrp}, \ref*{equfact:dsim} and \ref*{equfact:dsubsim} is an immediate consequence of \cref{thm:cat-transfer-csa} and the fact that $\sd$-factoriality is characterized by $\sc_\sd(R^\bullet) = 0$.
\end{proof}

\begin{remark}
  \envnewline
\begin{enumerate}
  \item
If $\Cl_A(\cO)$ is trivial, then $\sc_p(R^\bullet) \in [0,2]$. There are examples where $\sc_p(R^\bullet) = 0$ as well as examples where $\sc_p(R^\bullet) = 2$.
For the first, take $R = M_n(\cO)$ with $\cO$ a PID and $n \in \bN$.
For an example with $\sc_p(R^\bullet) = 2$, let $K$ be a number field, $\cO$ its ring of algebraic integers, and $R$ a classical maximal order in a totally definite quaternion algebra over $K$ such that $R$ is a PID.\spacefactor=\sfcode`\.{}
(Note that, amongst the totally definite quaternion algebras over all number fields, there indeed exist, up to isomorphism, finitely many such classical maximal orders, see, for instance, \cite[Table 8.1]{Kirschmer-Voight}. For instance, take $R$ to be the ring of  Hurwitz quaternions.)
Then $[R^\times:\cO^\times]<\infty$, while every totally positive prime element $p \in \cO$ that does not divide the discriminant has $\Norm_{K/\bQ}(p)+1$ rigid factorizations, with all atoms being similar.
Thus $\sc_p(p) = 2$, for $\Norm_{K/\bQ}(p)+1$ sufficiently large.
The previous corollary can therefore not be extended to permutable factoriality.

  \item The investigation of the $\omega_p$-invariant is, in the present setting, left open as it does not necessarily transfer via a non-isoatomic transfer homomorphism (i.e., if $\nr(u) \simeq \nr(v)$ does not imply $u \simeq v$ for atoms $u$,~$v \in R^\bullet$). A further investigation will require dealing with  additional difficulties similar to the ones in the example just given.
\end{enumerate}
\end{remark}

To demonstrate the usefulness of our transfer result, we state the following corollary, which is now an immediate consequence of known results about monoids of zero sum sequences.

\begin{corollary}[\cite{GGS11}]\label{cor:cat-transfer}
  Let $R$ and $C = \Cl_A(\cO)$ be as in \cref{thm:cat-transfer-csa}, and let $\sd$ be any of the distances $\sd^*$, $\sd_\dsim$, $\sd_\dsubsim$, or $\sd_p$ on $R^\bullet$.
  Then
  \begin{enumerate}
    \item $\sc_\sd(R^\bullet) \le 2$ if and only if $\card{C} \le 2$.
    \item $\sc_\sd(R^\bullet) = 3$ if and only if $C$ is isomorphic to one of the following groups: $\sC_3$, $\sC_2 \oplus \sC_2$, or $\sC_3 \oplus \sC_3$.
    \item $\sc_\sd(R^\bullet) = 4$ if and only if $C$ is isomorphic to one of the following groups: $\sC_4$, $\sC_2 \oplus \sC_4$, or $\sC_2 \oplus \sC_2 \oplus \sC_2$, or $\sC_3 \oplus \sC_3 \oplus \sC_3$.
  \end{enumerate}
  Suppose that $\card{C} \ge 3$, $C \cong \sC_{n_1} \oplus \cdots \oplus \sC_{n_r}$ with $r \in \bN$ and $1 < n_1 \mid n_2 \mid \cdots \mid n_r$, and that the following mild assumption on the Davenport constant of $C$ holds:
  \[
    \Big\lfloor \frac{1}{2} \sD(C) + 1 \Big\rfloor  \le \max\Big\{ n_r,\, 1 + \sum_{i=1}^r \Big\lfloor \frac{n_i}{2} \Big\rfloor \Big\}.
  \]
  Then
  \begin{enumerate}
    \item $\sc_\sd(R^\bullet) = \max \Delta(R^\bullet) + 2$ and
    \item $\sc_\sd(R^\bullet) \le \max\Big\{ n_r,\, \frac{1}{3} \big( 2 \sD(C) + \frac{1}{2} r n_r + 2^r \big) \Big\}$.
  \end{enumerate}
\end{corollary}

\begin{proof}
  The claims follow from \cref{thm:cat-transfer-csa} and the corresponding results for commutative Krull monoids (cf. Theorem 5.4 and Corollaries 4.3 and 5.6 in \cite{GGS11}).
\end{proof}

\medskip
\textbf{Acknowledgments.}
We would like to thank Alfred Geroldinger for providing valuable feedback on preliminary versions of this paper and its contents.
Moreover, we would like to thank the referee for his careful reading and a very helpful report.

\bibliographystyle{alphaabbr}
\bibliography{BaethSmertnigReferences}

\newcommand{\etalchar}[1]{$^{#1}$}
\begin{thebibliography}{CGSL{\etalchar{+}}06}

\bibitem[Ady60]{Adyan}
S.~I. Adyan.
\newblock On the embeddability of semigroups in groups.
\newblock {\em Soviet Math. Dokl.}, 1:819--821, 1960.

\bibitem[AM53]{Asano-Murata}
K.~Asano and K.~Murata.
\newblock Arithmetical ideal theory in semigroups.
\newblock {\em J. Inst. Polytech. Osaka City Univ. Ser. A. Math.}, 4:9--33,
  1953.

\bibitem[And97]{Anderson}
D.~D. Anderson, editor.
\newblock {\em Factorization in integral domains}, volume 189 of {\em Lecture
  Notes in Pure and Applied Mathematics}, New York, 1997. Marcel Dekker Inc.

\bibitem[BBG14]{Bachman-Baeth-Gossell}
D.~Bachman, N.~R. Baeth, and J.~Gossell.
\newblock Factorizations of upper triangular matrices.
\newblock {\em Linear Algebra Appl.}, 450:138--157, 2014.

\bibitem[BG14]{Baeth-Geroldinger}
N.~R. Baeth and A.~Geroldinger.
\newblock Monoids of modules and arithmetic of direct-sum decompositions.
\newblock {\em Pacific J. Math.}, 271(2):257--319, 2014.

\bibitem[BGSG11]{Blanco-GarciaSanchez-Geroldinger}
V.~Blanco, P.~A. Garc{\'{\i}}a-S{\'a}nchez, and A.~Geroldinger.
\newblock Semigroup-theoretical characterizations of arithmetical invariants
  with applications to numerical monoids and {K}rull monoids.
\newblock {\em Illinois J. Math.}, 55(4):1385--1414 (2013), 2011.

\bibitem[BPA{\etalchar{+}}11]{Baeth-Ponomarenko}
N.~R. Baeth, V.~Ponomarenko, D.~Adams, R.~Ardila, D.~Hannasch, A.~Kosh,
  H.~McCarthy, and R.~Rosenbaum.
\newblock Number theory of matrix semigroups.
\newblock {\em Linear Algebra Appl.}, 434(3):694--711, 2011.

\bibitem[Bru69]{Brungs}
H.-H. Brungs.
\newblock Ringe mit eindeutiger {F}aktorzerlegung.
\newblock {\em J. Reine Angew. Math.}, 236:43--66, 1969.

\bibitem[BW13]{Baeth-Wiegand}
N.~R. Baeth and R.~Wiegand.
\newblock Factorization theory and decompositions of modules.
\newblock {\em Amer. Math. Monthly}, 120(1):3--34, 2013.

\bibitem[CGSL{\etalchar{+}}06]{Chapman-GarciaSanchez-Llena-Ponomarenko-Rosales}
S.~T. Chapman, P.~A. Garc{\'{\i}}a-S{\'a}nchez, D.~Llena, V.~Ponomarenko, and
  J.~C. Rosales.
\newblock The catenary and tame degree in finitely generated commutative
  cancellative monoids.
\newblock {\em Manuscripta Math.}, 120(3):253--264, 2006.

\bibitem[Cha81]{Chamarie}
M.~Chamarie.
\newblock Anneaux de {K}rull non commutatifs.
\newblock {\em J. Algebra}, 72(1):210--222, 1981.

\bibitem[Cha84]{Chatters}
A.~W. Chatters.
\newblock Noncommutative unique factorization domains.
\newblock {\em Math. Proc. Cambridge Philos. Soc.}, 95(1):49--54, 1984.

\bibitem[Cha05]{Chapman}
S.~T. Chapman, editor.
\newblock {\em Arithmetical properties of commutative rings and monoids},
  volume 241 of {\em Lecture Notes in Pure and Applied Mathematics}. Chapman \&
  Hall/CRC, Boca Raton, FL, 2005.

\bibitem[CJ86]{Chatters-Jordan}
A.~W. Chatters and D.~A. Jordan.
\newblock Noncommutative unique factorisation rings.
\newblock {\em J. London Math. Soc. (2)}, 33(1):22--32, 1986.

\bibitem[CK15]{Cohn-Kumar}
H.~Cohn and A.~Kumar.
\newblock Metacommutation of {H}urwitz primes.
\newblock {\em Proc. Amer. Math. Soc.}, 143(4):1459--1469, 2015.

\bibitem[Coh85]{Cohn85}
P.~M. Cohn.
\newblock {\em Free rings and their relations}, volume~19 of {\em London
  Mathematical Society Monographs}.
\newblock Academic Press Inc. [Harcourt Brace Jovanovich Publishers], London,
  second edition, 1985.

\bibitem[Coh06]{Cohn06}
P.~M. Cohn.
\newblock {\em Free ideal rings and localization in general rings}, volume~3 of
  {\em New Mathematical Monographs}.
\newblock Cambridge University Press, Cambridge, 2006.

\bibitem[CS03]{Conway-Smith}
J.~H. Conway and D.~A. Smith.
\newblock {\em On quaternions and octonions: their geometry, arithmetic, and
  symmetry}.
\newblock A K Peters Ltd., Natick, MA, 2003.

\bibitem[DD13]{Deza-Deza}
M.~M. Deza and E.~Deza.
\newblock {\em Encyclopedia of distances}.
\newblock Springer, Heidelberg, second edition, 2013.

\bibitem[Deu68]{Deuring}
M.~Deuring.
\newblock {\em Algebren}.
\newblock Zweite, korrigierte Auflage. Ergebnisse der Mathematik und ihrer
  Grenzgebiete, Band 41. Springer-Verlag, Berlin, 1968.

\bibitem[DL07]{DL07}
J.~Delenclos and A.~Leroy.
\newblock Noncommutative symmetric functions and {$W$}-polynomials.
\newblock {\em J. Algebra Appl.}, 6(5):815--837, 2007.

\bibitem[EN89]{Estes-Nipp}
D.~R. Estes and G.~Nipp.
\newblock Factorization in quaternion orders.
\newblock {\em J. Number Theory}, 33(2):224--236, 1989.

\bibitem[Est91]{Estes}
D.~R. Estes.
\newblock Factorization in quaternion orders over number fields.
\newblock In {\em The mathematical heritage of {C}. {F}. {G}auss}, pages
  195--203. World Sci. Publ., River Edge, NJ, 1991.

\bibitem[FHL13]{Fontana-Houston-Lucas}
M.~Fontana, E.~Houston, and T.~Lucas.
\newblock {\em Factoring ideals in integral domains}, volume~14 of {\em Lecture
  Notes of the Unione Matematica Italiana}.
\newblock Springer, Heidelberg, 2013.

\bibitem[Ger09]{Geroldinger09}
A.~Geroldinger.
\newblock Additive group theory and non-unique factorizations.
\newblock In {\em Combinatorial number theory and additive group theory}, Adv.
  Courses Math. CRM Barcelona, pages 1--86. Birkh\"auser Verlag, Basel, 2009.

\bibitem[Ger13]{Geroldinger}
A.~Geroldinger.
\newblock Non-commutative {K}rull monoids: a divisor theoretic approach and
  their arithmetic.
\newblock {\em Osaka J. Math.}, 50(2):503--539, 2013.

\bibitem[GGRW05]{GGRW05}
I.~Gelfand, S.~Gelfand, V.~Retakh, and R.~L. Wilson.
\newblock Factorizations of polynomials over noncommutative algebras and
  sufficient sets of edges in directed graphs.
\newblock {\em Lett. Math. Phys.}, 74(2):153--167, 2005.

\bibitem[GGS11]{GGS11}
A.~Geroldinger, D.~J. Grynkiewicz, and W.~A. Schmid.
\newblock The catenary degree of {K}rull monoids {I}.
\newblock {\em J. Th\'eor. Nombres Bordeaux}, 23(1):137--169, 2011.

\bibitem[GH08]{GH08}
A.~Geroldinger and W.~Hassler.
\newblock Local tameness of {$v$}-{N}oetherian monoids.
\newblock {\em J. Pure Appl. Algebra}, 212(6):1509--1524, 2008.

\bibitem[GHK06]{GHK06}
A.~Geroldinger and F.~Halter-Koch.
\newblock {\em Non-unique factorizations}, volume 278 of {\em Pure and Applied
  Mathematics (Boca Raton)}.
\newblock Chapman \& Hall/CRC, Boca Raton, FL, 2006.
\newblock Algebraic, combinatorial and analytic theory.

\bibitem[GRSW05]{GRSW05}
I.~Gelfand, V.~Retakh, S.~Serconek, and R.~L. Wilson.
\newblock On a class of algebras associated to directed graphs.
\newblock {\em Selecta Math. (N.S.)}, 11(2):281--295, 2005.

\bibitem[GRW01]{GRW01}
I.~Gelfand, V.~Retakh, and R.~L. Wilson.
\newblock Quadratic linear algebras associated with factorizations of
  noncommutative polynomials and noncommutative differential polynomials.
\newblock {\em Selecta Math. (N.S.)}, 7(4):493--523, 2001.

\bibitem[Gry13]{Grynkiewicz}
D.~J. Grynkiewicz.
\newblock {\em Structural additive theory}, volume~30 of {\em Developments in
  Mathematics}.
\newblock Springer, Cham, 2013.

\bibitem[GY12]{Goodearl-Yakimov}
K.~Goodearl and M.~T. Yakimov.
\newblock From quantum {O}re extensions to quantum tori via noncommutative
  {UFD}s.
\newblock 2012.
\newblock preprint.

\bibitem[HR95]{HR95}
D.~Haile and L.~H. Rowen.
\newblock Factorizations of polynomials over division algebras.
\newblock {\em Algebra Colloq.}, 2(2):145--156, 1995.

\bibitem[Jac43]{Jacobson}
N.~Jacobson.
\newblock {\em The {T}heory of {R}ings}.
\newblock American Mathematical Society Mathematical Surveys, vol. I. American
  Mathematical Society, New York, 1943.

\bibitem[Jor89]{Jordan}
D.~A. Jordan.
\newblock Unique factorisation of normal elements in noncommutative rings.
\newblock {\em Glasgow Math. J.}, 31(1):103--113, 1989.

\bibitem[JW01]{Jespers-Wang}
E.~Jespers and Q.~Wang.
\newblock Noetherian unique factorization semigroup algebras.
\newblock {\em Comm. Algebra}, 29(12):5701--5715, 2001.

\bibitem[KV10]{Kirschmer-Voight}
M.~Kirschmer and J.~Voight.
\newblock Algorithmic enumeration of ideal classes for quaternion orders.
\newblock {\em SIAM J. Comput.}, 39(5):1714--1747, 2010.

\bibitem[Ler12]{L12}
A.~Leroy.
\newblock Noncommutative polynomial maps.
\newblock {\em J. Algebra Appl.}, 11(4):1250076, 16, 2012.

\bibitem[LL04]{LL04}
T.~Y. Lam and A.~Leroy.
\newblock Wedderburn polynomials over division rings. {I}.
\newblock {\em J. Pure Appl. Algebra}, 186(1):43--76, 2004.

\bibitem[LLO08]{LLO08}
T.~Y. Lam, A.~Leroy, and A.~Ozturk.
\newblock Wedderburn polynomials over division rings. {II}.
\newblock In {\em Noncommutative rings, group rings, diagram algebras and their
  applications}, volume 456 of {\em Contemp. Math.}, pages 73--98. Amer. Math.
  Soc., Providence, RI, 2008.

\bibitem[LLR06]{Launois-Lenagan-Rigal}
S.~Launois, T.~H. Lenagan, and L.~Rigal.
\newblock Quantum unique factorisation domains.
\newblock {\em J. London Math. Soc. (2)}, 74(2):321--340, 2006.

\bibitem[LO04]{LO04}
A.~Leroy and A.~Ozturk.
\newblock Algebraic and {$F$}-independent sets in 2-firs.
\newblock {\em Comm. Algebra}, 32(5):1763--1792, 2004.

\bibitem[MR01]{McConnell-Robson}
J.~C. McConnell and J.~C. Robson.
\newblock {\em Noncommutative {N}oetherian rings}, volume~30 of {\em Graduate
  Studies in Mathematics}.
\newblock American Mathematical Society, Providence, RI, revised edition, 2001.
\newblock With the cooperation of L. W. Small.

\bibitem[MVO12]{Marubayashi-VanOystaeyen}
H.~Marubayashi and F.~Van~Oystaeyen.
\newblock {\em {P}rime {D}ivisors and {N}oncommutative {V}aluation {T}heory},
  volume 2059 of {\em Lecture Notes in Mathematics}.
\newblock Springer, Berlin, 2012.

\bibitem[Phi10]{Philipp}
A.~Philipp.
\newblock A characterization of arithmetical invariants by the monoid of
  relations.
\newblock {\em Semigroup Forum}, 81(3):424--434, 2010.

\bibitem[Rei75]{Reiner}
I.~Reiner.
\newblock {\em Maximal orders}.
\newblock Academic Press [A subsidiary of Harcourt Brace Jovanovich,
  Publishers], London-New York, 1975.
\newblock London Mathematical Society Monographs, No. 5.

\bibitem[Rem80]{Remmers}
J.~H. Remmers.
\newblock On the geometry of semigroup presentations.
\newblock {\em Adv. in Math.}, 36(3):283--296, 1980.

\bibitem[Ret10]{R10}
V.~Retakh.
\newblock From factorizations of noncommutative polynomials to combinatorial
  topology.
\newblock {\em Cent. Eur. J. Math.}, 8(2):235--243, 2010.

\bibitem[Sme13]{Smertnig}
D.~Smertnig.
\newblock Sets of lengths in maximal orders in central simple algebras.
\newblock {\em J. Algebra}, 390:1--43, 2013.

\bibitem[WF74]{Wagner-Fischer}
R.~A. Wagner and M.~J. Fischer.
\newblock The string-to-string correction problem.
\newblock {\em J. Assoc. Comput. Mach.}, 21:168--173, 1974.

\bibitem[ZM08]{Zieve-Mueller}
M.~Zieve and P.~M\"uller.
\newblock On {R}itt's polynomial decomposition theorems.
\newblock 2008.
\newblock preprint.

\end{thebibliography}

\end{document}